\documentclass[11pt]{preprint}

\usepackage{difftrees}
\usepackage[full]{textcomp}
\usepackage[osf]{newtxtext} 
\usepackage[cal=boondoxo]{mathalfa}
\usepackage{colortbl}

\usepackage{tikz-cd}
\usetikzlibrary{cd}
\usepackage{comment}

\usepackage{amssymb}
\usepackage{mathtools}
\usepackage{hyperref}
\usepackage{breakurl}
\usepackage{mhenvs}
\usepackage{mhequ} 
\usepackage{mhsymb}
\usepackage{booktabs}
\usepackage{tikz}
\usepackage{tcolorbox}
\usepackage{mathrsfs}
\usepackage[utf8]{inputenc}
\usepackage{longtable}
\usepackage{wrapfig}
\usepackage{rotating}
\usepackage{microtype}
\usepackage{comment}
\usepackage{wasysym}
\usepackage{centernot}
\usepackage{enumitem}
\usepackage{bm}
\usepackage{stackrel}
\usepackage{graphicx}
\usepackage{diagrams}
\makeatletter
\newcommand{\globalcolor}[1]{%
  \color{#1}\global\let\default@color\current@color
}
\makeatother

\usetikzlibrary{calc}
\usetikzlibrary{decorations}
\usetikzlibrary{positioning}
\usetikzlibrary{shapes}

\newif\ifdark
\darkfalse

\ifdark

\definecolor{darkred}{rgb}{0.9,0.2,0.2}
\definecolor{darkblue}{rgb}{0.7,0.3,1}
\definecolor{darkgreen}{rgb}{0.1,0.9,0.1}
\definecolor{franck}{rgb}{0,0.8,1}
\definecolor{pagebackground}{rgb}{.15,.21,.18}
\definecolor{pageforeground}{rgb}{.84,.84,.85}
\pagecolor{pagebackground}
\AtBeginDocument{\globalcolor{pageforeground}}
\definecolor{symbols}{rgb}{0,0.7,1}
\colorlet{connection}{red!80!black}
\colorlet{boxcolor}{blue!50}

\else

\definecolor{darkred}{rgb}{0.7,0.1,0.1}
\definecolor{darkblue}{rgb}{0.4,0.1,0.8}
\definecolor{darkgreen}{rgb}{0.1,0.7,0.1}
\definecolor{franck}{rgb}{0,0,1}
\definecolor{pagebackground}{rgb}{1,1,1}
\definecolor{pageforeground}{rgb}{0,0,0}

\colorlet{symbols}{blue!90!black}
\colorlet{connection}{red!30!black}
\colorlet{boxcolor}{blue!50!black}

\fi

\def\slash{\leavevmode\unskip\kern0.18em/\penalty\exhyphenpenalty\kern0.18em}
\def\dash{\leavevmode\unskip\kern0.18em--\penalty\exhyphenpenalty\kern0.18em}

\DeclarePairedDelimiter\floor{\lfloor}{\rfloor}

\DeclareMathAlphabet{\mathbbm}{U}{bbm}{m}{n}

\DeclareFontFamily{U}{BOONDOX-calo}{\skewchar\font=45 }
\DeclareFontShape{U}{BOONDOX-calo}{m}{n}{
  <-> s*[1.05] BOONDOX-r-calo}{}
\DeclareFontShape{U}{BOONDOX-calo}{b}{n}{
  <-> s*[1.05] BOONDOX-b-calo}{}
\DeclareMathAlphabet{\mcb}{U}{BOONDOX-calo}{m}{n}
\SetMathAlphabet{\mcb}{bold}{U}{BOONDOX-calo}{b}{n}

\setlist{noitemsep,topsep=4pt,leftmargin=1.5em}

\DeclareMathAlphabet{\mathbbm}{U}{bbm}{m}{n}

\DeclareMathAlphabet{\mcb}{U}{BOONDOX-calo}{m}{n}
\SetMathAlphabet{\mcb}{bold}{U}{BOONDOX-calo}{b}{n}
\DeclareFontFamily{U}{mathx}{\hyphenchar\font45}
\DeclareFontShape{U}{mathx}{m}{n}{
      <5> <6> <7> <8> <9> <10>
      <10.95> <12> <14.4> <17.28> <20.74> <24.88>
      mathx10
      }{}
\DeclareSymbolFont{mathx}{U}{mathx}{m}{n}
\DeclareMathSymbol{\bigtimes}{1}{mathx}{"91}

\providecommand{\figures}{false}
{ \ifthenelse{\equal{\figures}{false}} {#1}{\[ {\rm Figure \ missing !} \]} }{}
\def\id{\mathrm{id}}

\def\CH{\mathcal{H}}

\def\CT{\mathcal{T}}

\tikzstyle{tinydots}=[dash pattern=on \pgflinewidth off \pgflinewidth]
\tikzstyle{superdense}=[dash pattern=on 4pt off 1pt]

\def\${|\!|\!|}

\newenvironment{DIFnomarkup}{}{} 

\theorembodyfont{\rmfamily}
\newtheorem{example}[lemma]{Example}

\newfont{\indic}{bbmss12}

\def\Nabla_#1{\nabla_{\!#1}}

\makeatletter
\pgfdeclareshape{crosscircle}
{
  \inheritsavedanchors[from=circle] 
  \inheritanchorborder[from=circle]
  \inheritanchor[from=circle]{north}
  \inheritanchor[from=circle]{north west}
  \inheritanchor[from=circle]{north east}
  \inheritanchor[from=circle]{center}
  \inheritanchor[from=circle]{west}
  \inheritanchor[from=circle]{east}
  \inheritanchor[from=circle]{mid}
  \inheritanchor[from=circle]{mid west}
  \inheritanchor[from=circle]{mid east}
  \inheritanchor[from=circle]{base}
  \inheritanchor[from=circle]{base west}
  \inheritanchor[from=circle]{base east}
  \inheritanchor[from=circle]{south}
  \inheritanchor[from=circle]{south west}
  \inheritanchor[from=circle]{south east}
  \inheritbackgroundpath[from=circle]
  \foregroundpath{
    \centerpoint%
    \pgf@xc=\pgf@x%
    \pgf@yc=\pgf@y%
    \pgfutil@tempdima=\radius%
    \pgfmathsetlength{\pgf@xb}{\pgfkeysvalueof{/pgf/outer xsep}}%
    \pgfmathsetlength{\pgf@yb}{\pgfkeysvalueof{/pgf/outer ysep}}%
    \ifdim\pgf@xb<\pgf@yb%
      \advance\pgfutil@tempdima by-\pgf@yb%
    \else%
      \advance\pgfutil@tempdima by-\pgf@xb%
    \fi%
    \pgfpathmoveto{\pgfpointadd{\pgfqpoint{\pgf@xc}{\pgf@yc}}{\pgfqpoint{-0.707107\pgfutil@tempdima}{-0.707107\pgfutil@tempdima}}}
    \pgfpathlineto{\pgfpointadd{\pgfqpoint{\pgf@xc}{\pgf@yc}}{\pgfqpoint{0.707107\pgfutil@tempdima}{0.707107\pgfutil@tempdima}}}
    \pgfpathmoveto{\pgfpointadd{\pgfqpoint{\pgf@xc}{\pgf@yc}}{\pgfqpoint{-0.707107\pgfutil@tempdima}{0.707107\pgfutil@tempdima}}}
    \pgfpathlineto{\pgfpointadd{\pgfqpoint{\pgf@xc}{\pgf@yc}}{\pgfqpoint{0.707107\pgfutil@tempdima}{-0.707107\pgfutil@tempdima}}}
  }
}
\makeatother

\def\symbol#1{\textcolor{symbols}{#1}}

\def\decorate#1#2{
        \ifnum#2>0
    		\foreach \count in {1,...,#2}{
	       	let
				\p1 = (sourcenode.center),
                \p2 = (sourcenode.east),
				\n1 = {\x2-\x1},
				\n2 = {1mm},
				\n3 = {(1.3+0.6*(\count-1))*\n1},
				\n4 = {0.7*\n1}
			in 
        		node[rectangle,fill=symbols,rotate=30,inner sep=0pt,minimum width=0.2*\n2,minimum height=\n2] at ($(sourcenode.center) + (\n3,\n4)$) {}
				}
		\fi
        \ifnum#1>0
    		\foreach \count in {1,...,#1}{
	       	let
				\p1 = (sourcenode.center),
                \p2 = (sourcenode.east),
				\n1 = {\x2-\x1},
				\n2 = {1mm},
				\n3 = {(1.3+0.6*(\count-1))*\n1},
				\n4 = {0.7*\n1}
			in 
        		node[rectangle,fill=symbols,rotate=-30,inner sep=0pt,minimum width=0.2*\n2,minimum height=\n2] at ($(sourcenode.center) + (-\n3,\n4)$) {}
				}
		\fi
}

\tikzset{
    dectriangle/.style 2 args={
        triangle,
        alias=sourcenode,
        append after command={\decorate{#1}{#2}}
    },
    dectriangle/.default={0}{0},
}

\tikzset{
	cross/.style={path picture={ 
  		\draw[symbols]
			(path picture bounding box.south east) -- (path picture bounding box.north west) (path picture bounding box.south west) -- (path picture bounding box.north east);
		}},
root/.style={circle,fill=green!50!black,inner sep=0pt, minimum size=1.2mm},
        dot/.style={circle,fill=pageforeground,inner sep=0pt, minimum size=1mm},
        blank/.style={circle,fill=white,inner sep=0pt, minimum size=1mm},
        dotred/.style={circle,fill=pageforeground!50!pagebackground,inner sep=0pt, minimum size=2mm},
        var/.style={circle,fill=pageforeground!10!pagebackground,draw=pageforeground,inner sep=0pt, minimum size=3mm},
        kernel/.style={semithick,shorten >=2pt,shorten <=2pt},
        kernels/.style={snake=zigzag,shorten >=2pt,shorten <=2pt,segment amplitude=1pt,segment length=4pt,line before snake=2pt,line after snake=5pt,},
        rho/.style={densely dashed,semithick,shorten >=2pt,shorten <=2pt},
           testfcn/.style={dotted,semithick,shorten >=2pt,shorten <=2pt},
        renorm/.style={shape=circle,fill=pagebackground,inner sep=1pt},
        labl/.style={shape=rectangle,fill=pagebackground,inner sep=1pt},
        xic/.style={very thin,circle,draw=symbols,fill=symbols,inner sep=0pt,minimum size=1.2mm},
        g/.style={very thin,rectangle,draw=symbols,fill=symbols!10!pagebackground,inner sep=0pt,minimum width=2.5mm,minimum height=1.2mm},
        xi/.style={very thin,circle,draw=symbols,fill=symbols!10!pagebackground,inner sep=0pt,minimum size=1.2mm},
	xies/.style={very thin,rectangle,fill=green!50!black!25,draw=symbols,inner sep=0pt,minimum size=1.1mm},
	xiesf/.style={very thin,rectangle,fill=green!50!black,draw=symbols,inner sep=0pt,minimum size=1.1mm},
        xix/.style={very thin,crosscircle,fill=symbols!10!pagebackground,draw=symbols,inner sep=0pt,minimum size=1.2mm},
        X/.style={very thin,cross,rectangle,fill=pagebackground,draw=symbols,inner sep=0pt,minimum size=1.2mm},
	xib/.style={thin,circle,fill=symbols!10!pagebackground,draw=symbols,inner sep=0pt,minimum size=1.6mm},
	xie/.style={thin,circle,fill=green!50!black,draw=symbols,inner sep=0pt,minimum size=1.6mm},
	xid/.style={thin,circle,fill=symbols,draw=symbols,inner sep=0pt,minimum size=1.6mm},
	xibx/.style={thin,crosscircle,fill=symbols!10!pagebackground,draw=symbols,inner sep=0pt,minimum size=1.6mm},
	kernel1/.style={thick},
	kernels2/.style={very thick,draw=connection,segment length=12pt},
	keps/.style={thin,draw=symbols,->},
	kepspr/.style={thick,draw=connection,->},
	krho/.style={thin,draw=symbols,superdense,->},
	krhopr/.style={thick,draw=connection,superdense,->},
	triangle/.style = { regular polygon, regular polygon sides=3},
	not/.style={thin,circle,draw=connection,fill=connection,inner sep=0pt,minimum size=0.5mm},
	diff/.style = {very thin,draw=symbols,triangle,fill=red!50!black,inner sep=0pt,minimum size=1.6mm},
	diff1/.style = {very thin,dectriangle={1}{0},fill=red!50!black,draw=symbols,inner sep=0pt,minimum size=1.6mm},
	diff2/.style = {very thin,dectriangle={1}{1},fill=red!50!black,draw=symbols,inner sep=0pt,minimum size=1.6mm},
		diffmini/.style = {very thin,rectangle,fill=black,draw=black,inner sep=0pt,minimum size=0.75mm},
	 kernelsmod/.style={very thick,draw=connection,segment length=12pt},
	 rec/.style = {very thin,rectangle,fill=black,draw=black,inner sep=0pt,minimum size=2mm},
	cerc/.style={very thin,circle,draw=black,fill=symbols,inner sep=0pt,minimum size=2mm},
	stars/.style={very thin,star,star points=6,star point ratio=0.5, draw=black,fill=red,inner sep=0pt,minimum size=0.7mm},
	>=stealth,
        }

\makeatletter
\def\DeclareSymbol#1#2#3{%
	\expandafter\gdef\csname MH@symb@#1\endcsname{\tikzsetnextfilename{symbol#1}%
	\tikz[baseline=#2,scale=0.15,draw=symbols,line join=round]{#3}}%
	\expandafter\gdef\csname MH@symb@#1s\endcsname{\scalebox{0.75}{\tikzsetnextfilename{symbol#1}%
	\tikz[baseline=#2,scale=0.15,draw=symbols,line join=round]{#3}}}%
	\expandafter\gdef\csname MH@symb@#1ss\endcsname{\scalebox{0.65}{\tikzsetnextfilename{symbol#1}%
	\tikz[baseline=#2,scale=0.15,draw=symbols,line join=round]{#3}}}%
	}
\def\<#1>{\ifthenelse{\boolean{mmode}}{\mathchoice{\csname MH@symb@#1\endcsname}{\csname MH@symb@#1\endcsname}{\csname MH@symb@#1s\endcsname}{\csname MH@symb@#1ss\endcsname}}{\csname MH@symb@#1\endcsname}}
\makeatother

\DeclareSymbol{Xi22}{0.5}{\draw (0,0) node[xi] {} -- (-1,1) node[not] {} -- (0,2) node[xi] {};} 
\DeclareSymbol{Xi2}{-2}{\draw (-1,-0.25) node[xi] {} -- (0,1) node[xi] {};} 
\DeclareSymbol{Xi2b}{-2}{\draw (-1,-0.25) node[xic] {} -- (0,1) node[xic] {};} 
\DeclareSymbol{Xi2g}{-2}{\draw (-1,-0.25) node[xies] {} -- (0,1) node[xi] {};} 
\DeclareSymbol{Xi2g2}{-2}{\draw (-1,-0.25) node[xi] {} -- (0,1) node[xies] {};} 
\DeclareSymbol{cXi2}{-2}{\draw (0,-0.25) node[xi] {} -- (-1,1) node[xic] {};}
\DeclareSymbol{Xi3}{0}{\draw (0,0) node[xi] {} -- (-1,1) node[xi] {} -- (0,2) node[xi] {};}
\DeclareSymbol{XiIIXi}{0}{\draw (0,0) node[xi] {} -- (-1,1); \draw[kernels2] (-1,1) node[not] {} -- (0,2) node[xi] {};}

\DeclareSymbol{Xi4}{2}{\draw (0,0) node[xi] {} -- (-1,1) node[xi] {} -- (0,2) node[xi] {} -- (-1,3) node[xi] {};}
\DeclareSymbol{Xi4_1}{2}{\draw (0,0) node[xic] {} -- (-1,1) node[xic] {} -- (0,2) node[xi] {} -- (-1,3) node[xi] {};}
\DeclareSymbol{Xi4_2}{2}{\draw (0,0) node[xic] {} -- (-1,1) node[xi] {} -- (0,2) node[xi] {} -- (-1,3) node[xic] {};}
\DeclareSymbol{Xi2X}{-2}{\draw (0,-0.25) node[xi] {} -- (-1,1) node[xix] {};}
\DeclareSymbol{XXi2}{-2}{\draw (0,-0.25) node[xix] {} -- (-1,1) node[xi] {};}
\DeclareSymbol{IIXi}{0}{\draw (0,-0.25) node[not] {} -- (-1,1) node[xi] {} -- (0,2) node[xi] {};}
\DeclareSymbol{IXi^2}{-1}{\draw (-1,1) node[xi] {} -- (0,0) node[not] {} -- (1,1) node[xi] {};}
\DeclareSymbol{IIXi^2}{-4}{\draw (0,-1.5) node[not] {} -- (0,0);
\draw[kernels2] (-1,1) node[xi] {} -- (0,0) node[not] {} -- (1,1) node[xi] {};}
\DeclareSymbol{XiX}{-2.8}{\node[xibx] {};}
\DeclareSymbol{tauX}{-2.8}{ \node[X] {};}
\DeclareSymbol{Xi}{-2.8}{\node[xib] {};}

\DeclareSymbol{IXiX}{-1}{\draw (0,-0.25) node[not] {} -- (-1,1) node[xix] {};}
\DeclareSymbol{IXi3}{2}{\draw (0,-0.25) node[not] {} -- (-1,1) node[xi] {} -- (0,2) node[xi] {} -- (-1,3) node[xi] {};}
\DeclareSymbol{IXi}{-2}{\draw (0,-0.25) node[not] {} -- (-1,1) node[xi] {};}
\DeclareSymbol{XiI}{-2}{\draw (0,-0.25) node[xi] {} -- (-1,1) node[not] {};}

\DeclareSymbol{Xi4b}{0}{\draw(0,1.5) node[xi] {} -- (0,0); \draw (-1,1) node[xi] {} -- (0,0) node[xi] {} -- (1,1) node[xi] {};}
\DeclareSymbol{Xi4b'}{0}{\draw(0,1.5) node[xi] {} -- (0,-0.2); \draw (-1,1) node[xi] {} -- (0,-0.2) node[not] {} -- (1,1) node[xi] {};}
\DeclareSymbol{Xi4c}{0}{\draw (0,1) -- (0.8,2.2) node[xi] {};\draw (0,-0.25) node[xi] {} -- (0,1) node[xi] {} -- (-0.8,2.2) node[xi] {};}
\DeclareSymbol{Xi4d}{-4.5}{\draw (0,-1.5) node[not] {} -- (0,0); \draw (-1,1) node[xi] {} -- (0,0) node[xi] {} -- (1,1) node[xi] {};}
\DeclareSymbol{Xi4e}{0}{\draw (0,2) node[xi] {} -- (-1,1) node[xi] {} -- (0,0) node[xi] {} -- (1,1) node[xi] {};}
\DeclareSymbol{Xi4e'}{0}{\draw (0,2) node[xi] {} -- (-1,1) node[xi] {} -- (0,-0.2) node[not] {} -- (1,1) node[xi] {};}

\DeclareSymbol{Xitwo}
{0}{\draw[kernels2] (0,0) node[not] {} -- (-1,1) node[not] {}
-- (-2,2) node[not]{} -- (-3,3) node[xi]  {};
\draw[kernels2] (0,0) -- (1,1) node[xi] {};
\draw[kernels2] (-1,1) -- (0,2) node[xi] {};
\draw[kernels2] (-2,2) -- (-1,3) node[xi] {};}

\DeclareSymbol{IXitwo}
{0}{\draw (-.7,1.2) node[xi] {} -- (0,-0.2) -- (.7,1.2) node[xi] {};}
\DeclareSymbol{I1Xitwo}
{0}{\draw[kernels2] (0,0) node[not] {} -- (-1,1) node[xi] {};
\draw[kernels2] (0,0) -- (1,1) node[xi] {};}

\DeclareSymbol{I1Xitwobis}
{0}{\draw[kernels2] (0,0) node[not] {} -- (-1,1) node[xies] {};
\draw[kernels2] (0,0) -- (1,1) node[xies] {};}

\DeclareSymbol{I1Xitwog}
{0}{\draw[kernels2] (0,0) node[not] {} -- (-1,1) node[xies] {};
\draw[kernels2] (0,0) -- (1,1) node[xi] {};}

\DeclareSymbol{cI1Xitwo}
{0}{\draw[kernels2] (0,0) node[not] {} -- (-1,1) node[xic] {};
\draw[kernels2] (0,0) -- (1,1) node[xi] {};}

\DeclareSymbol{I1IXi3}{0}{\draw (0,0) node[xi] {} -- (-1,1) ; 
\draw[kernels2] (-1,1) node[not] {} -- (0,2) node[xi] {};
\draw[kernels2] (-1,1) node[not] {} -- (-2,2) node[xi] {};}

\DeclareSymbol{I1Xi3c}{-1}{\draw[kernels2](0,1.5) node[xi] {} -- (0,0) node[not] {}; \draw (-1,1) node[xi] {} -- (0,0) ; \draw[kernels2] (0,0) -- (1,1) node[xi] {};}

\DeclareSymbol{I1Xi3cbis}{-1}{\draw[kernels2](0,1.5) node[xies] {} -- (0,0) node[not] {}; \draw (-1,1) node[xies] {} -- (0,0) ; \draw[kernels2] (0,0) -- (1,1) node[xies] {};}

\DeclareSymbol{I1IXi3b}{0}{\draw[kernels2] (0,0) node[not] {} -- (-1,1) ; \draw[kernels2] (0,0)   -- (1,1) node[xi] {} ;
\draw (-1,1) node[xi] {} -- (0,2) node[xi] {};
}

\DeclareSymbol{I1IXi3c}{0}{\draw[kernels2] (0,0) node[not] {} -- (-1,1) ; \draw[kernels2] (0,0)   -- (1,1) node[xi] {} ;
\draw[kernels2] (-1,1) node[not] {} -- (0,2) node[xi] {};
\draw[kernels2] (-1,1) node[not] {} -- (-2,2) node[xi] {};}

\DeclareSymbol{I1IXi3cbis}{0}{\draw[kernels2] (0,0) node[not] {} -- (-1,1) ; \draw[kernels2] (0,0)   -- (1,1) node[xies] {} ;
\draw[kernels2] (-1,1) node[not] {} -- (0,2) node[xies] {};
\draw[kernels2] (-1,1) node[not] {} -- (-2,2) node[xies] {};}

\DeclareSymbol{I1Xi}{0}{\draw[kernels2] (0,0) node[not] {} -- (-1,1)  node[xi] {} ;}

\DeclareSymbol{I1Xi4a}{2}{\draw[kernels2] (0,0) node[not] {} -- (-1,1) ; \draw[kernels2] (0,0) node[not] {} -- (1,1) node[xi] {} ;
\draw (-1,1) node[xi] {} -- (0,2) node[xi] {} -- (-1,3) node[xi] {};}

\DeclareSymbol{cI1Xi4a}{2}{\draw[kernels2] (0,0) node[not] {} -- (-1,1) ; \draw[kernels2] (0,0) node[not] {} -- (1,1) node[xic] {} ;
\draw (-1,1) node[xic] {} -- (0,2) node[xi] {} -- (-1,3) node[xi] {};}

\DeclareSymbol{I1Xi4b}{2}{\draw (0,0) node[xi] {} -- (-1,1) node[xi] {} -- (0,2) ; \draw[kernels2] (0,2) node[not] {} -- (-1,3) node[xi] {};\draw[kernels2] (0,2)  -- (1,3) node[xi] {};
}

\DeclareSymbol{cI1Xi4b}{2}{\draw (0,0) node[xic] {} -- (-1,1) node[xic] {} -- (0,2) ; \draw[kernels2] (0,2) node[not] {} -- (-1,3) node[xi] {};\draw[kernels2] (0,2)  -- (1,3) node[xi] {};
}

\DeclareSymbol{I1Xi4c}{2}{\draw (0,0) node[xi] {} -- (-1,1) node[not] {}; \draw[kernels2] (-1,1) -- (0,2) ; 
\draw[kernels2] (-1,1) -- (-2,2) node[xi] {} ;
\draw (0,2) node[xi] {} -- (-1,3) node[xi] {};}

\DeclareSymbol{cI1Xi4c}{2}{\draw (0,0) node[xic] {} -- (-1,1) node[not] {}; \draw[kernels2] (-1,1) -- (0,2) ; 
\draw[kernels2] (-1,1) -- (-2,2) node[xic] {} ;
\draw (0,2) node[xi] {} -- (-1,3) node[xi] {};}

\DeclareSymbol{I1Xi4ab}{2}{\draw[kernels2] (0,0) node[not] {} -- (-1,1) ; \draw[kernels2] (0,0) node[not] {} -- (1,1) node[xi] {};\draw (-1,1) node[xi] {} -- (0,2) ; \draw[kernels2] (0,2) node[not] {} -- (-1,3) node[xi] {};\draw[kernels2] (0,2)  -- (1,3) node[xi] {}; }

\DeclareSymbol{cI1Xi4ab}{2}{\draw[kernels2] (0,0) node[not] {} -- (-1,1) ; \draw[kernels2] (0,0) node[not] {} -- (1,1) node[xic] {};\draw (-1,1) node[xic] {} -- (0,2) ; \draw[kernels2] (0,2) node[not] {} -- (-1,3) node[xi] {};\draw[kernels2] (0,2)  -- (1,3) node[xi] {}; }

\DeclareSymbol{I1Xi4bc}{2}{\draw (0,0) node[xi] {} -- (-1,1) node[not] {}; \draw[kernels2] (-1,1) -- (0,2) ; 
\draw[kernels2] (-1,1) -- (-2,2) node[xi] {} ; \draw[kernels2] (0,2) node[not] {} -- (-1,3) node[xi] {};\draw[kernels2] (0,2)  -- (1,3) node[xi] {};
}

\DeclareSymbol{cI1Xi4bc}{2}{\draw (0,0) node[xic] {} -- (-1,1) node[not] {}; \draw[kernels2] (-1,1) -- (0,2) ; 
\draw[kernels2] (-1,1) -- (-2,2) node[xic] {} ; \draw[kernels2] (0,2) node[not] {} -- (-1,3) node[xi] {};\draw[kernels2] (0,2)  -- (1,3) node[xi] {};
}

\DeclareSymbol{I1Xi4abcc1}{2}{\draw[kernels2] (0,0) node[not] {} -- (-1,1) node[not] {}
-- (-2,2) node[not]{} -- (-3,3) node[xic]  {};
\draw[kernels2] (0,0) -- (1,1) node[xic] {};
\draw[kernels2] (-1,1) -- (0,2) node[xi] {};
\draw[kernels2] (-2,2) -- (-1,3) node[xi] {};
}

\DeclareSymbol{I1Xi4abcc1b}{2}{\draw[kernels2] (0,0) node[not] {} -- (-1,1) node[not] {}
-- (-2,2) node[not]{} -- (-3,3) node[xi]  {};
\draw[kernels2] (0,0) -- (1,1) node[xic] {};
\draw[kernels2] (-1,1) -- (0,2) node[xic] {};
\draw[kernels2] (-2,2) -- (-1,3) node[xi] {};
}

\DeclareSymbol{I1Xi4abcc2}{2}{\draw[kernels2] (0,0) node[not] {} -- (-1,1) node[not] {}
-- (-2,2) node[not]{} -- (-3,3) node[xic]  {};
\draw[kernels2] (0,0) -- (1,1) node[xi] {};
\draw[kernels2] (-1,1) -- (0,2) node[xi] {};
\draw[kernels2] (-2,2) -- (-1,3) node[xic] {};
}

\DeclareSymbol{I1Xi4ac}{2}{\draw[kernels2] (0,0) node[not] {} -- (-1,1) ; \draw[kernels2] (0,0) node[not] {} -- (1,1) node[xi] {}; 
\draw[kernels2] (-1,1) node[not] {} -- (0,2) ; 
\draw[kernels2] (-1,1) -- (-2,2) node[xi] {} ;
\draw (0,2) node[xi] {} -- (-1,3) node[xi] {};}

\DeclareSymbol{cI1Xi4ac}{2}{\draw[kernels2] (0,0) node[not] {} -- (-1,1) ; \draw[kernels2] (0,0) node[not] {} -- (1,1) node[xic] {}; 
\draw[kernels2] (-1,1) node[not] {} -- (0,2) ; 
\draw[kernels2] (-1,1) -- (-2,2) node[xic] {} ;
\draw (0,2) node[xi] {} -- (-1,3) node[xi] {};}

\DeclareSymbol{I1Xi4acc1}{2}{\draw[kernels2] (0,0) node[not] {} -- (-1,1) ; \draw[kernels2] (0,0) node[not] {} -- (1,1) node[xic] {}; 
\draw[kernels2] (-1,1) node[not] {} -- (0,2) ; 
\draw[kernels2] (-1,1) -- (-2,2) node[xi] {} ;
\draw (0,2) node[xic] {} -- (-1,3) node[xi] {};}

\DeclareSymbol{I1Xi4acc2}{2}{\draw[kernels2] (0,0) node[not] {} -- (-1,1) ; \draw[kernels2] (0,0) node[not] {} -- (1,1) node[xic] {}; 
\draw[kernels2] (-1,1) node[not] {} -- (0,2) ; 
\draw[kernels2] (-1,1) -- (-2,2) node[xi] {} ;
\draw (0,2) node[xi] {} -- (-1,3) node[xic] {};}

\DeclareSymbol{2I1Xi4}{2}{\draw[kernels2] (0,0) node[not] {} -- (-1,1) node[not] {};
\draw[kernels2] (0,0) -- (1,1) node[not] {};
\draw[kernels2] (-1,1) -- (-1.5,2.5) node[xi] {};
\draw[kernels2] (-1,1) -- (-0.5,2.5) node[xi] {};
\draw[kernels2] (1,1) -- (0.5,2.5) node[xi] {};
\draw[kernels2] (1,1) -- (1.5,2.5) node[xi] {};
}

\DeclareSymbol{2I1Xi4dis}{2}{\draw[kernels2] (0,0) node[not] {} -- (-1,1) node[not] {};
\draw[kernels2] (0,0) -- (1,1) node[not] {};
\draw[kernels2] (-1,1) -- (-1.5,2.5) node[xies] {};
\draw[kernels2] (-1,1) -- (-0.5,2.5) node[xies] {};
\draw[kernels2] (1,1) -- (0.5,2.5) node[xies] {};
\draw[kernels2] (1,1) -- (1.5,2.5) node[xies] {};
}

\DeclareSymbol{2I1Xi4c1}{2}{\draw[kernels2] (0,0) node[not] {} -- (-1,1) node[not] {};
\draw[kernels2] (0,0) -- (1,1) node[not] {};
\draw[kernels2] (-1,1) -- (-1.5,2.5) node[xic] {};
\draw[kernels2] (-1,1) -- (-0.5,2.5) node[xi] {};
\draw[kernels2] (1,1) -- (0.5,2.5) node[xic] {};
\draw[kernels2] (1,1) -- (1.5,2.5) node[xi] {};
}

\DeclareSymbol{2I1Xi4c2}{2}{\draw[kernels2] (0,0) node[not] {} -- (-1,1) node[not] {};
\draw[kernels2] (0,0) -- (1,1) node[not] {};
\draw[kernels2] (-1,1) -- (-1.5,2.5) node[xic] {};
\draw[kernels2] (-1,1) -- (-0.5,2.5) node[xic] {};
\draw[kernels2] (1,1) -- (0.5,2.5) node[xi] {};
\draw[kernels2] (1,1) -- (1.5,2.5) node[xi] {};
}

\DeclareSymbol{2I1Xi4b}{2}{\draw[kernels2] (0,0) node[not] {} -- (-1,1) ;
\draw[kernels2] (0,0) -- (1,1);
\draw (-1,1) node[xi] {} -- (-1,2.5) node[xi] {};
\draw (1,1)  node[xi] {} -- (1,2.5) node[xi] {};
}

\DeclareSymbol{2I1Xi4bb}{2}{\draw[kernels2] (0,0) node[not] {} -- (-1,1) ;
\draw[kernels2] (0,0) -- (1,1);
\draw (-1,1) node[xi] {} -- (-1,2.5) node[xiesf] {};
\draw (1,1)  node[xi] {} -- (1,2.5) node[xic] {};
}

\DeclareSymbol{2I1Xi4c}{2}{\draw[kernels2] (0,0) node[not] {} -- (-1,1);
\draw[kernels2] (0,0) -- (1,1) node[not] {};
\draw (-1,1)  node[xi] {} -- (-1,2.5) node[xi] {};
\draw[kernels2] (1,1) -- (0.4,2.5) node[xi] {};
\draw[kernels2] (1,1) -- (1.6,2.5) node[xi] {};
}

\DeclareSymbol{2I1Xi4cc1}{2}{\draw[kernels2] (0,0) node[not] {} -- (-1,1);
\draw[kernels2] (0,0) -- (1,1) node[not] {};
\draw (-1,1)  node[xic] {} -- (-1,2.5) node[xi] {};
\draw[kernels2] (1,1) -- (0.4,2.5) node[xic] {};
\draw[kernels2] (1,1) -- (1.6,2.5) node[xi] {};
}

\DeclareSymbol{2I1Xi4cc2}{2}{\draw[kernels2] (0,0) node[not] {} -- (-1,1);
\draw[kernels2] (0,0) -- (1,1) node[not] {};
\draw (-1,1)  node[xic] {} -- (-1,2.5) node[xic] {};
\draw[kernels2] (1,1) -- (0.4,2.5) node[xi] {};
\draw[kernels2] (1,1) -- (1.6,2.5) node[xi] {};
}

\DeclareSymbol{Xi4ba}{0}{\draw(-0.5,1.5) node[xi] {} -- (0,0); \draw (-1.5,1) node[xi] {} -- (0,0) node[not] {}; \draw[kernels2] (0,0) -- (1.5,1) node[xi] {};
\draw[kernels2] (0,0) -- (0.5,1.5) node[xi] {} ;}

\DeclareSymbol{Xi4badis}{0}{\draw(-0.5,1.5) node[xies] {} -- (0,0); \draw (-1.5,1) node[xies] {} -- (0,0) node[not] {}; \draw[kernels2] (0,0) -- (1.5,1) node[xies] {};
\draw[kernels2] (0,0) -- (0.5,1.5) node[xies] {} ;}

\DeclareSymbol{Xi4ba1}{0}{\draw(-0.5,1.5) node[xi] {} -- (0,0); \draw (-1.5,1) node[xi] {} -- (0,0) node[not] {}; \draw[kernels2] (0,0) -- (1.5,1) node[xic] {};
\draw[kernels2] (0,0) -- (0.5,1.5) node[xic] {} ;}

\DeclareSymbol{Xi4ba1b}{0}{\draw(-0.5,1.5) node[xic] {} -- (0,0); \draw (-1.5,1) node[xic] {} -- (0,0) node[not] {}; \draw[kernels2] (0,0) -- (1.5,1) node[xi] {};
\draw[kernels2] (0,0) -- (0.5,1.5) node[xi] {} ;}

\DeclareSymbol{Xi4ba1bdiff}{0}{\draw(-0.5,1.5) node[xic] {} -- (0,0); \draw (-1.5,1) node[xic] {} -- (0,0) node[not] {}; \draw (0,0) -- (1.5,1) node[xi] {};
\draw (0,0) -- (0.5,1.5) node[xi] {};
\draw(0,0) node[diff] {};}

\DeclareSymbol{Xi4ba1bb}{0}{\draw(-0.5,1.5) node[xic] {} -- (0,0); \draw (-1.5,1) node[xiesf] {} -- (0,0) node[not] {}; \draw[kernels2] (0,0) -- (1.5,1) node[xi] {};
\draw[kernels2] (0,0) -- (0.5,1.5) node[xi] {} ;}

\DeclareSymbol{Xi4ba2}{0}{\draw(-0.5,1.5) node[xi] {} -- (0,0); \draw (-1.5,1) node[xic] {} -- (0,0) node[not] {}; \draw[kernels2] (0,0) -- (1.5,1) node[xi] {};
\draw[kernels2] (0,0) -- (0.5,1.5) node[xic] {} ;}

\DeclareSymbol{Xi4ba2b}{0}{\draw(-0.5,1.5) node[xi] {} -- (0,0); \draw (-1.5,1) node[xic] {} -- (0,0) node[not] {}; \draw[kernels2] (0,0) -- (1.5,1) node[xi] {};
\draw[kernels2] (0,0) -- (0.5,1.5) node[xiesf] {} ;}

\DeclareSymbol{Xi4ca}{0}{\draw (0,1) -- (-1,2.2) node[xi] {};\draw (0,-0.25) node[xi] {} -- (0,1) ; \draw[kernels2] (0,1) node[not] {} -- (1,2.2) node[xi] {};
\draw[kernels2] (0,1) {} -- (0,2.7) node[xi] {};
}

\DeclareSymbol{Xi4cb}{0}{\draw (-1,1) -- (-2,2) node[xi] {};\draw[kernels2] (0,0)  -- (-1,1) node[xi] {} ; \draw[kernels2] (0,0) node[not] {} -- (1,1) node[xi] {} ; 
\draw (-1,1) node[xi] {} -- (0,2) node[xi] {};}

\DeclareSymbol{Xi4cbb}{0}{\draw (-1,1) -- (-2,2) node[xiesf] {};\draw[kernels2] (0,0)  -- (-1,1) node[xi] {} ; \draw[kernels2] (0,0) node[not] {} -- (1,1) node[xi] {} ; 
\draw (-1,1) node[xi] {} -- (0,2) node[xic] {};}

\DeclareSymbol{Xi4cbc1}{0}{\draw (-1,1) -- (-2,2) node[xic] {};\draw[kernels2] (0,0)  -- (-1,1) node[xic] {} ; \draw[kernels2] (0,0) node[not] {} -- (1,1) node[xi] {} ; 
\draw (-1,1) node[xic] {} -- (0,2) node[xi] {};}

\DeclareSymbol{Xi4cbc2}{0}{\draw (-1,1) -- (-2,2) node[xi] {};\draw[kernels2] (0,0)  -- (-1,1) node[xi] {} ; \draw[kernels2] (0,0) node[not] {} -- (1,1) node[xic] {} ; 
\draw (-1,1) node[xic] {} -- (0,2) node[xi] {};}

\DeclareSymbol{Xi4cab}{0}{\draw (-1,1) -- (-2,2) node[xi] {};\draw[kernels2] (0,0)  -- (-1,1); \draw[kernels2] (0,0) node[not] {} -- (1,1) node[xi] {} ; 
\draw[kernels2] (-1,1)  {} -- (0,2) node[xi] {};
\draw[kernels2] (-1,1) node[not] {} -- (-1,2.5) node[xi] {};
}

\DeclareSymbol{Xi4cabdis}{0}{\draw (-1,1) -- (-2,2) node[xies] {};\draw[kernels2] (0,0)  -- (-1,1); \draw[kernels2] (0,0) node[not] {} -- (1,1) node[xies] {} ; 
\draw[kernels2] (-1,1)  {} -- (0,2) node[xies] {};
\draw[kernels2] (-1,1) node[not] {} -- (-1,2.5) node[xies] {};
}

\DeclareSymbol{Xi4cabc1}{0}{\draw (-1,1) -- (-2,2) node[xi] {};\draw[kernels2] (0,0)  -- (-1,1); \draw[kernels2] (0,0) node[not] {} -- (1,1) node[xic] {} ; 
\draw[kernels2] (-1,1)  {} -- (0,2) node[xic] {};
\draw[kernels2] (-1,1) node[not] {} -- (-1,2.5) node[xi] {};
}

\DeclareSymbol{Xi4cabc2}{0}{\draw (-1,1) -- (-2,2) node[xic] {};\draw[kernels2] (0,0)  -- (-1,1); \draw[kernels2] (0,0) node[not] {} -- (1,1) node[xic] {} ; 
\draw[kernels2] (-1,1)  {} -- (0,2) node[xi] {};
\draw[kernels2] (-1,1) node[not] {} -- (-1,2.5) node[xi] {};
}

\DeclareSymbol{Xi4ea}{1.5}{\draw (-1,2.5) node[xi] {} -- (-1,1) node[xi] {} -- (0,0); 
 \draw[kernels2] (0,0)  -- (1,1) node[xi] {};
\draw[kernels2] (0,0) node[not] {} -- (0,1.5) node[xi] {}; }

\DeclareSymbol{Xi4eac1}{1.5}{\draw (-1,2.5) node[xic] {} -- (-1,1) node[xi] {} -- (0,0); 
 \draw[kernels2] (0,0)  -- (1,1) node[xic] {};
\draw[kernels2] (0,0) node[not] {} -- (0,1.5) node[xi] {}; }

\DeclareSymbol{Xi4eac1b}{1.5}{\draw (-1,2.5) node[xic] {} -- (-1,1) node[xi] {} -- (0,0); 
 \draw[kernels2] (0,0)  -- (1,1) node[xiesf] {};
\draw[kernels2] (0,0) node[not] {} -- (0,1.5) node[xi] {}; }

\DeclareSymbol{Xi4eac2}{1.5}{\draw (-1,2.5) node[xic] {} -- (-1,1) node[xic] {} -- (0,0); 
 \draw[kernels2] (0,0)  -- (1,1) node[xi] {};
\draw[kernels2] (0,0) node[not] {} -- (0,1.5) node[xi] {}; }

\DeclareSymbol{Xi4eact1}{1.5}{\draw (-1,2.5) node[xic] {} -- (-1,1) node[xi] {} -- (0,0); 
 \draw (0,0)  -- (1,1) node[xic] {};
\draw[rho] (0,0) node[not] {} -- (0,1.5) node[xi] {}; }

\DeclareSymbol{Xi4eact2}{1.5}{\draw[rho] (-1,2.5) node[xic] {} -- (-1,1) node[xi] {} -- (0,0); 
 \draw (0,0)  -- (1,1) node[xic] {};
\draw (0,0) node[not] {} -- (0,1.5) node[xi] {}; }

\DeclareSymbol{Xi4eabis}{1.5}{\draw (-1,2.5) node[xi] {} -- (-1,1) ; \draw[kernels2] (-1,1) node[xi] {} -- (0,0); 
 \draw (0,0)  -- (1,1) node[xi] {};
\draw[kernels2] (0,0) node[not] {} -- (0,1.5) node[xi] {}; }

\DeclareSymbol{Xi4eabisc1}{1.5}{\draw (-1,2.5) node[xic] {} -- (-1,1) ; \draw[kernels2] (-1,1) node[xi] {} -- (0,0); 
 \draw (0,0)  -- (1,1) node[xi] {};
\draw[kernels2] (0,0) node[not] {} -- (0,1.5) node[xic] {}; }

\DeclareSymbol{Xi4eabisc1b}{1.5}{\draw (-1,2.5) node[xic] {} -- (-1,1) ; \draw[kernels2] (-1,1) node[xi] {} -- (0,0); 
 \draw (0,0)  -- (1,1) node[xi] {};
\draw[kernels2] (0,0) node[not] {} -- (0,1.5) node[xiesf] {}; }

\DeclareSymbol{Xi4eabisc1bis}{1.5}{\draw (-1,2.5) node[xi] {} -- (-1,1) ; \draw[kernels2] (-1,1) node[xi] {} -- (0,0); 
 \draw (0,0)  -- (1,1) node[xi] {};
\draw[kernels2] (0,0) node[not] {} -- (0,1.5) node[xi] {};
\draw (-2,1) node[] {\tiny{$i$}};
\draw (-2,2.5) node[] {\tiny{$\ell$}};
\draw (2,1) node[] {\tiny{$k$}};
\draw (0,2.5) node[] {\tiny{$j$}};
 }

\DeclareSymbol{Xi4eabisc1tris}{1.5}{\draw (-1,2.5) node[xi] {} -- (-1,1) ; \draw[kernels2] (-1,1) node[xi] {} -- (0,0); 
 \draw (0,0)  -- (1,1) node[xi] {};
\draw[kernels2] (0,0) node[not] {} -- (0,1.5) node[xi] {};
\draw (-2,1) node[] {\tiny{i}};
\draw (-2,2.5) node[] {\tiny{j}};
\draw (2,1) node[] {\tiny{j}};
\draw (0,2.5) node[] {\tiny{i}};
 }

\DeclareSymbol{Xi4eabisc1quater}{1.5}{\draw (-1,2.5) node[xic] {} -- (-1,1) ; \draw[kernels2] (-1,1) node[xi] {} -- (0,0); 
 \draw (0,0)  -- (1,1) node[xic] {};
\draw[kernels2] (0,0) node[not] {} -- (0,1.5) node[xi] {};
 }

\DeclareSymbol{Xi4eabisc2}{1.5}{\draw (-1,2.5) node[xic] {} -- (-1,1) ; \draw[kernels2] (-1,1) node[xi] {} -- (0,0); 
 \draw (0,0)  -- (1,1) node[xic] {};
\draw[kernels2] (0,0) node[not] {} -- (0,1.5) node[xi] {}; }

\DeclareSymbol{Xi4eabisc2l}{1.5}{\draw (-1,2.5) node[xiesf] {} -- (-1,1) ; \draw[kernels2] (-1,1) node[xi] {} -- (0,0); 
 \draw (0,0)  -- (1,1) node[xic] {};
\draw[kernels2] (0,0) node[not] {} -- (0,1.5) node[xi] {}; }

\DeclareSymbol{Xi4eabisc2r}{1.5}{\draw (-1,2.5) node[xic] {} -- (-1,1) ; \draw[kernels2] (-1,1) node[xi] {} -- (0,0); 
 \draw (0,0)  -- (1,1) node[xiesf] {};
\draw[kernels2] (0,0) node[not] {} -- (0,1.5) node[xi] {}; }

\DeclareSymbol{Xi4eabisc3}{1.5}{\draw (-1,2.5) node[xic] {} -- (-1,1) ; \draw[kernels2] (-1,1) node[xic] {} -- (0,0); 
 \draw (0,0)  -- (1,1) node[xi] {};
\draw[kernels2] (0,0) node[not] {} -- (0,1.5) node[xi] {}; }

\DeclareSymbol{Xi4eb}{0}{
\draw[kernels2] (0,2) node[xi] {} -- (-1,1) ; \draw[kernels2] (-2,2)  node[xi] {} -- (-1,1) ; \draw (-1,1)  node[not] {} -- (0,0); 
 \draw (0,0) node[xi] {}  -- (1,1) node[xi] {};
}

\DeclareSymbol{Xi4eab}{1.5}{\draw[kernels2] (-1,2.5) node[xi] {} -- (-1,1) ; \draw[kernels2] (-2,2)  node[xi] {} -- (-1,1) ; \draw (-1,1)  node[not] {} -- (0,0); 
 \draw[kernels2] (0,0)  -- (1,1) node[xi] {};
\draw[kernels2] (0,0) node[not] {} -- (0,1.5) node[xi] {}; 
}

\DeclareSymbol{Xi4eabdis}{1.5}{\draw[kernels2] (-1,2.5) node[xies] {} -- (-1,1) ; \draw[kernels2] (-2,2)  node[xies] {} -- (-1,1) ; \draw (-1,1)  node[not] {} -- (0,0); 
 \draw[kernels2] (0,0)  -- (1,1) node[xies] {};
\draw[kernels2] (0,0) node[not] {} -- (0,1.5) node[xies] {}; 
}

\DeclareSymbol{Xi4eabc1}{1.5}{\draw[kernels2] (-1,2.5) node[xic] {} -- (-1,1) ; \draw[kernels2] (-2,2)  node[xi] {} -- (-1,1) ; \draw (-1,1)  node[not] {} -- (0,0); 
 \draw[kernels2] (0,0)  -- (1,1) node[xic] {};
\draw[kernels2] (0,0) node[not] {} -- (0,1.5) node[xi] {}; 
}

\DeclareSymbol{Xi4eabc2}{1.5}{\draw[kernels2] (-1,2.5) node[xi] {} -- (-1,1) ; \draw[kernels2] (-2,2)  node[xi] {} -- (-1,1) ; \draw (-1,1)  node[not] {} -- (0,0); 
 \draw[kernels2] (0,0)  -- (1,1) node[xic] {};
\draw[kernels2] (0,0) node[not] {} -- (0,1.5) node[xic] {}; 
}

\DeclareSymbol{Xi4eabbis}{1.5}{\draw[kernels2] (-1,2.5) node[xi] {} -- (-1,1) ; \draw[kernels2] (-2,2)  node[xi] {} -- (-1,1) ; \draw[kernels2] (-1,1)  node[not] {} -- (0,0); 
 \draw (0,0)  -- (1,1) node[xi] {};
\draw[kernels2] (0,0) node[not] {} -- (0,1.5) node[xi] {}; 
}

\DeclareSymbol{Xi4eabbisc1}{1.5}{\draw[kernels2] (-1,2.5) node[xic] {} -- (-1,1) ; \draw[kernels2] (-2,2)  node[xi] {} -- (-1,1) ; \draw[kernels2] (-1,1)  node[not] {} -- (0,0); 
 \draw (0,0)  -- (1,1) node[xic] {};
\draw[kernels2] (0,0) node[not] {} -- (0,1.5) node[xi] {}; 
}

\DeclareSymbol{Xi4eabbisc1perm}{1.5}{\draw[kernels2] (-1,2.5) node[xi] {} -- (-1,1) ; \draw[kernels2] (-2,2)  node[xic] {} -- (-1,1) ; \draw[kernels2] (-1,1)  node[not] {} -- (0,0); 
 \draw (0,0)  -- (1,1) node[xic] {};
\draw[kernels2] (0,0) node[not] {} -- (0,1.5) node[xi] {}; 
}

\DeclareSymbol{Xi4eabbisc2}{1.5}{\draw[kernels2] (-1,2.5) node[xi] {} -- (-1,1) ; \draw[kernels2] (-2,2)  node[xi] {} -- (-1,1) ; \draw[kernels2] (-1,1)  node[not] {} -- (0,0); 
 \draw (0,0)  -- (1,1) node[xic] {};
\draw[kernels2] (0,0) node[not] {} -- (0,1.5) node[xic] {}; 
}

\DeclareSymbol{Xi2cbis}{0}{\draw[kernels2] (0,1) -- (0.8,2.2) node[xi] {};\draw[kernels2] (0,-0.25) node[not] {} -- (0,1); \draw[kernels2] (0,1) node[not] {} -- (-0.8,2.2) node[xi] {};}

\DeclareSymbol{Xi2cbis1}{0}{\draw (0,1) -- (-0.8,2.2) node[xi] {};\draw[kernels2] (0,-0.25) node[not] {} -- (0,1) node[xi] {}; }

\DeclareSymbol{Xi2Xbis}{-2}{\draw[kernels2] (0,-0.25)  -- (-1,1) ; \draw (-1,1) node[xix] {};
\draw[kernels2] (0,-0.25) node[not] {} -- (1,1) node[xi] {};}

\DeclareSymbol{XXi2bis}{-2}{\draw[kernels2] (0,-0.25) -- (-1,1) node[xi] {};
\draw[kernels2] (0,-0.25) node[X] {} -- (1,1) node[xi] {};}

\DeclareSymbol{I1XiIXi}{0}{\draw[kernels2] (0,-0.25) -- (1,1) node[xi] {};
\draw (0,-0.25) node[not] {} -- (-1,1) node[xi] {};}

\DeclareSymbol{I1XiIXib}{0}{\draw  (0,-0.25) node[xi] {} -- (0,1) node[not] {};
\draw[kernels2] (0,1) -- (0,2.25) ; \draw (0,2.25) node[xi]{}; }

\DeclareSymbol{I1XiIXic}{0}{
\draw[kernels2] (0,0) -- (1,1) node[xi] {} ; 
\draw[kernels2] (0,0) node[not] {}  -- (-1,1) node[not] {} -- (0,2) node[xi] {};
}

\DeclareSymbol{thin}{1.4}{\draw[pagebackground] (-0.3,0) -- (0.3,0); \draw  (0,0) -- (0,2);}
\DeclareSymbol{thin2}{1.4}{\draw[pagebackground] (-0.3,0) -- (0.3,0); \draw[tinydots]  (0,0) -- (0,2);}

\DeclareSymbol{thick}{1.4}{\draw[pagebackground] (-0.3,0) -- (0.3,0); \draw[kernels2]  (0,0) -- (0,2);}
\DeclareSymbol{thick2}{1.4}{\draw[pagebackground] (-0.3,0) -- (0.3,0); \draw[kernels2,tinydots]  (0,0) -- (0,2);}

\DeclareSymbol{Xi4ind}{2}{\draw (0,0) node[xi,label={[label distance=-0.2em]right: \scriptsize  $ i $}]  { } -- (-1,1) node[xi,label={[label distance=-0.2em]left: \scriptsize  $ j $}] {} -- (0,2) node[xi,label={[label distance=-0.2em]right: \scriptsize  $ k $}] {} -- (-1,3) node[xi,label={[label distance=-0.2em]left: \scriptsize  $ \ell $}] {};}

\DeclareSymbol{Xi4c1}{2}{\draw (0,0) node[xic] {} -- (-1,1) node[xi] {} -- (0,2) node[xic] {} -- (-1,3) node[xi] {};} 
\DeclareSymbol{IXi2ex}{0}{\draw (0,-0.25) node[xie] {} -- (-1,1) node[xi] {} ; \draw (0,-0.25)-- (1,1) node[xi] {};}
\DeclareSymbol{IXi2ex1}{0}{\draw (0,-0.25) node[xie] {} -- (-1,1) node[xi] {} -- (0,2) node[xi] {};}

\DeclareSymbol{Xi4b1}{0}{\draw(0,1.5) node[xic] {} -- (0,0); \draw (-1,1) node[xic] {} -- (0,0) node[xi] {} -- (1,1) node[xi] {};}

\DeclareSymbol{Xi4ec1}{0}{\draw (0,2) node[xi] {} -- (-1,1) node[xic] {} -- (0,0) node[xic] {} -- (1,1) node[xi] {};}
\DeclareSymbol{Xi4ec2}{0}{\draw (0,2) node[xic] {} -- (-1,1) node[xi] {} -- (0,0) node[xic] {} -- (1,1) node[xi] {};}
\DeclareSymbol{Xi4ec3}{0}{\draw (0,2) node[xic] {} -- (-1,1) node[xic] {} -- (0,0) node[xi] {} -- (1,1) node[xi] {};}

\DeclareSymbol{I1Xi4ac1}{2}{\draw[kernels2] (0,0) node[not] {} -- (-1,1) ; \draw[kernels2] (0,0) node[not] {} -- (1,1) node[xic] {} ;
\draw (-1,1) node[xi] {} -- (0,2) node[xic] {} -- (-1,3) node[xi] {};}

\DeclareSymbol{I1Xi4ac2}{2}{\draw[kernels2] (0,0) node[not] {} -- (-1,1) ; \draw[kernels2] (0,0) node[not] {} -- (1,1) node[xic] {} ;
\draw (-1,1) node[xi] {} -- (0,2) node[xi] {} -- (-1,3) node[xic] {};}

\DeclareSymbol{I1Xi4bp}{2}{\draw (0,0) node[not] {} -- (-1,1) node[xi] {} -- (0,2) ; \draw[kernels2] (0,2) node[not] {} -- (-1,3) node[xi] {};\draw[kernels2] (0,2)  -- (1,3) node[xi] {};
}

\DeclareSymbol{I1Xi4bc1}{2}{\draw (0,0) node[xic] {} -- (-1,1) node[xi] {} -- (0,2) ; \draw[kernels2] (0,2) node[not] {} -- (-1,3) node[xi] {};\draw[kernels2] (0,2)  -- (1,3) node[xic] {};
}

\DeclareSymbol{I1Xi4bc2}{2}{\draw (0,0) node[xic] {} -- (-1,1) node[xi] {} -- (0,2) ; \draw[kernels2] (0,2) node[not] {} -- (-1,3) node[xic] {};\draw[kernels2] (0,2)  -- (1,3) node[xi] {};
}

\DeclareSymbol{I1Xi4cp}{2}{\draw (0,0) node[not] {} -- (-1,1) node[not] {}; \draw[kernels2] (-1,1) -- (0,2) ; 
\draw[kernels2] (-1,1) -- (-2,2) node[xi] {} ;
\draw (0,2) node[xi] {} -- (-1,3) node[xi] {};}

\DeclareSymbol{I1Xi4cc1}{2}{\draw (0,0) node[xic] {} -- (-1,1) node[not] {}; \draw[kernels2] (-1,1) -- (0,2) ; 
\draw[kernels2] (-1,1) -- (-2,2) node[xi] {} ;
\draw (0,2) node[xic] {} -- (-1,3) node[xi] {};}

\DeclareSymbol{I1Xi4cc2}{2}{\draw (0,0) node[xic] {} -- (-1,1) node[not] {}; \draw[kernels2] (-1,1) -- (0,2) ; 
\draw[kernels2] (-1,1) -- (-2,2) node[xi] {} ;
\draw (0,2) node[xi] {} -- (-1,3) node[xic] {};}

\DeclareSymbol{I1Xi4abc1}{2}{\draw[kernels2] (0,0) node[not] {} -- (-1,1) ; \draw[kernels2] (0,0) node[not] {} -- (1,1) node[xic] {};\draw (-1,1) node[xi] {} -- (0,2) ; \draw[kernels2] (0,2) node[not] {} -- (-1,3) node[xic] {};\draw[kernels2] (0,2)  -- (1,3) node[xi] {}; }

\DeclareSymbol{I1Xi4abc2}{2}{\draw[kernels2] (0,0) node[not] {} -- (-1,1) ; \draw[kernels2] (0,0) node[not] {} -- (1,1) node[xic] {};\draw (-1,1) node[xi] {} -- (0,2) ; \draw[kernels2] (0,2) node[not] {} -- (-1,3) node[xi] {};\draw[kernels2] (0,2)  -- (1,3) node[xic] {}; }

\DeclareSymbol{R1}{0}{\draw (-1,1) node[xi] {} -- (0,0) node[not] {};
\draw[kernels2] (0,1.5) node[xic] {} -- (0,0) -- (1,1) node[xic] {};}
\DeclareSymbol{R2}{0}{\draw (-1,1) node[xic] {} -- (0,0) node[not] {};
\draw[kernels2] (0,1.5)  {} -- (0,0) -- (1,1)  {};
\draw (0,1.5) node[xi] {};
\draw (1,1) node[xic] {};
}
\DeclareSymbol{R3}{1}{\draw[kernels2] (-1,1.5)  {} -- (0,0) node[not] {} -- (1,1.5);
\draw (-1,1.5) node[xi] {};
\draw[kernels2] (0,3) {} -- (1,1.5) -- (2,3)  {};
\draw  (0,3) node[xic] {} ;
\draw (2,3) node[xic] {};}
\DeclareSymbol{R4}{1}{\draw[kernels2] (-1,1.5) node[xic] {} -- (0,0) node[not] {} -- (1,1.5);
\draw[kernels2] (0,3) {} -- (1,1.5) -- (2,3) node[xic] {};
\draw (0,3) node[xi] {};}

\DeclareSymbol{I1Xi4bcp}{2}{\draw (0,0) node[not] {} -- (-1,1) node[not] {}; \draw[kernels2] (-1,1) -- (0,2) ; 
\draw[kernels2] (-1,1) -- (-2,2) node[xi] {} ; \draw[kernels2] (0,2) node[not] {} -- (-1,3) node[xi] {};\draw[kernels2] (0,2)  -- (1,3) node[xi] {};
}

\DeclareSymbol{I1Xi4bcc1}{2}{\draw (0,0) node[xic] {} -- (-1,1) node[not] {}; \draw[kernels2] (-1,1) -- (0,2) ; 
\draw[kernels2] (-1,1) -- (-2,2) node[xi] {} ; \draw[kernels2] (0,2) node[not] {} -- (-1,3) node[xi] {};\draw[kernels2] (0,2)  -- (1,3) node[xic] {};
}

\DeclareSymbol{I1Xi4bcc2}{2}{\draw (0,0) node[xic] {} -- (-1,1) node[not] {}; \draw[kernels2] (-1,1) -- (0,2) ; 
\draw[kernels2] (-1,1) -- (-2,2) node[xi] {} ; \draw[kernels2] (0,2) node[not] {} -- (-1,3) node[xic] {};\draw[kernels2] (0,2)  -- (1,3) node[xi] {};
} 

\DeclareSymbol{2I1Xi4bc1}{2}{\draw[kernels2] (0,0) node[not] {} -- (-1,1) ;
\draw[kernels2] (0,0) -- (1,1);
\draw (-1,1) node[xic] {} -- (-1,2.5) node[xi] {};
\draw (1,1)  node[xic] {} -- (1,2.5) node[xi] {};
}

\DeclareSymbol{2I1Xi4bc2}{2}{\draw[kernels2] (0,0) node[not] {} -- (-1,1) ;
\draw[kernels2] (0,0) -- (1,1);
\draw (-1,1) node[xi] {} -- (-1,2.5) node[xic] {};
\draw (1,1)  node[xic] {} -- (1,2.5) node[xi] {};
}

\DeclareSymbol{diff2I1Xi4bc2}{2}{\draw (0,0) node[diff] {} -- (-1,1) ;
\draw (0,0) -- (1,1);
\draw (-1,1) node[xi] {} -- (-1,2.5) node[xic] {};
\draw (1,1)  node[xic] {} -- (1,2.5) node[xi] {};
}

\DeclareSymbol{2I1Xi4bc3}{2}{\draw[kernels2] (0,0) node[not] {} -- (-1,1) ;
\draw[kernels2] (0,0) -- (1,1);
\draw (-1,1) node[xic] {} -- (-1,2.5) node[xic] {};
\draw (1,1)  node[xi] {} -- (1,2.5) node[xi] {};
}

\DeclareSymbol{Xi41}{0}{\draw (0,1) -- (0.8,2.2) node[xic] {};\draw (0,-0.25) node[xi] {} -- (0,1) node[xi] {} -- (-0.8,2.2) node[xic] {};} 

\DeclareSymbol{Xi42}{0}{\draw (0,1) -- (0.8,2.2) node[xi] {};\draw (0,-0.25) node[xic] {} -- (0,1) node[xi] {} -- (-0.8,2.2) node[xic] {};}

\DeclareSymbol{Xi4ca1}{0}{\draw (0,1) -- (-1,2.2) node[xic] {};\draw (0,-0.25) node[xi] {} -- (0,1) ; \draw[kernels2] (0,1) node[not] {} -- (1,2.2) node[xic] {};
\draw[kernels2] (0,1) {} -- (0,2.7) node[xi] {};
}

\DeclareSymbol{Xi4ca2}{0}{\draw (0,1) -- (-1,2.2) node[xi] {};\draw (0,-0.25) node[xi] {} -- (0,1) ; \draw[kernels2] (0,1) node[not] {} -- (1,2.2) node[xic] {};
\draw[kernels2] (0,1) {} -- (0,2.7) node[xic] {};
}

\DeclareSymbol{Xi4cap}{0}{\draw (0,1) -- (-1,2.2) node[xi] {};\draw (0,-0.25) node[not] {} -- (0,1) ; \draw[kernels2] (0,1) node[not] {} -- (1,2.2) node[xi] {};
\draw[kernels2] (0,1) {} -- (0,2.7) node[xi] {};
}

\DeclareSymbol{Xi3a}{0}{
 \draw (-1,1)  node[xi] {} -- (0,0); 
 \draw (0,0) node[xi] {}  -- (1,1) node[xi] {};
}

\DeclareSymbol{Xi4ebc1}{0}{
\draw[kernels2] (0,2) node[xi] {} -- (-1,1) ; \draw[kernels2] (-2,2)  node[xic] {} -- (-1,1) ; \draw (-1,1)  node[not] {} -- (0,0); 
 \draw (0,0) node[xic] {}  -- (1,1) node[xi] {};
}

\DeclareSymbol{Xi4ebc2}{0}{
\draw[kernels2] (0,2) node[xi] {} -- (-1,1) ; \draw[kernels2] (-2,2)  node[xi] {} -- (-1,1) ; \draw (-1,1)  node[not] {} -- (0,0); 
 \draw (0,0) node[xic] {}  -- (1,1) node[xic] {};
}

\DeclareSymbol{Xi2cbispex}{0}{\draw[kernels2] (0,1) -- (0.8,2.2) node[xi] {};\draw (0,-0.25) node[xie] {} -- (0,1); \draw[kernels2] (0,1) node[not] {} -- (-0.8,2.2) node[xi] {};}

\DeclareSymbol{Xi2cbis1p}{0}{\draw (0,1) -- (-0.8,2.2) node[xi] {};\draw (0,-0.25) node[not] {} -- (0,1) node[xi] {}; }

\DeclareSymbol{Xi2Xp}{-2}{\draw (0,-0.25) node[not] {} -- (-1,1) node[xix] {};} 

\DeclareSymbol{I1XiIXib}{0}{\draw  (0,-0.25) node[xi] {} -- (0,1) node[not] {};
\draw[kernels2] (0,1) -- (0,2.25) ; \draw (0,2.25) node[xi]{}; }

\DeclareSymbol{IXi2b}{0}{\draw  (0,-0.25) node[xi] {} -- (0,1) node[not] {};
\draw (0,1) -- (0,2.25) ; \draw (0,2.25) node[xi]{}; }

\DeclareSymbol{IXi2bex}{0}{\draw  (0,-0.25) node[xi] {} -- (0,1) node[xie] {};
\draw (0,1) -- (0,2.25) ; \draw (0,2.25) node[xi]{}; }

 \def\1{\mathbf{\symbol{1}}}

\def\one{\mathbf{1}}

\DeclareSymbol{diff}{0}{
\draw (0,0.5) node[diff] {};
}

\DeclareSymbol{diff1}{0}{
\draw (0,0.5) node[diff1] {};
}

\DeclareSymbol{diff2}{0}{
\draw (0,0.5) node[diff2] {};
}

\DeclareSymbol{geo}{0}{
\draw (0,0) node[diff] {};
\draw (0.3,0) node[diff] {};
}

\DeclareSymbol{generic}{0}{
\draw (0,0.6) node[xi] {};
}

\DeclareSymbol{g}{0}{
\draw (0,0.6) node[g] {};
}

\DeclareSymbol{Ito}{0}{
\draw (0,0.6) node[xies] {};
}

\DeclareSymbol{Itob}{0}{
\draw (0,0.6) node[xiesf] {};
}

\DeclareSymbol{greycirc}{0}{
\draw (0,0.3) node[xi] {};
}

\DeclareSymbol{not}{0}{
\draw (0,0.6) node[not] {};
\draw[tinydots] (0,0.6) circle (0.8);
}

\DeclareSymbol{genericb}{0}{
\draw (0,0.6) node[xic] {};
}

\DeclareSymbol{bluecirc}{0}{
\draw (0,0.3) node[xic] {};
}

\DeclareSymbol{genericxix}{0}{
\draw (0,0.6) node[xix] {};
}

\DeclareSymbol{genericX}{0}{
\draw (0,0.6) node[X] {};
}

\DeclareSymbol{diffIto}{1}{
\draw  (0,2.5) -- (0,0) ;
\draw (0,-0.1) node[diff] {};
\draw (0,2.5) node[xies] {};
}
\DeclareSymbol{Itodiff}{2}{
\draw(0,2.9) -- (0,-0.2);
\draw (0,2.9) node[diff] {};
\draw (0,-0.1) node[xies] {};
}

\DeclareSymbol{diffgeneric}{1}{
\draw  (0,2.5) -- (0,0) ;
\draw (0,-0.1) node[diff] {};
\draw (0,2.5) node[xi] {};
}

\DeclareSymbol{genericdiff}{2}{
\draw(0,2.9) -- (0,-0.2);
\draw (0,2.9) node[diff] {};
\draw (0,-0.1) node[xi] {};
}

\DeclareSymbol{diffdot}{2}{
\draw  (0,3) -- (0,-0.1) ;
\draw (0,3) node[not] {};
\draw (0,-0.1) node[diff] {};
}

\DeclareSymbol{diffdotmini}{0}{
\draw  (0,0) -- (0,1.2) ;
\draw (0,1.2) node[not] {};
\draw (0,0) node[diffmini] {};
}

\DeclareSymbol{dotdiff}{2}{
\draw[kernelsmod]  (0,3) -- (0,-0.1) ;
\draw (0,3) node[diff] {};
\draw (0,-0.1) node[not] {};
}

\DeclareSymbol{dotdiff1}{2}{
\draw[kernelsmod]  (0,3) -- (0,-0.1) ;
\draw (0,3) node[diff1] {};
\draw (0,-0.1) node[not] {};
}

\DeclareSymbol{dotdiff1mini}{0}{
\draw[kernelsmod]  (0,1.2) -- (0,0) ;
\draw (0,1.2) node[diffmini] {};
\draw (0,0) node[not] {};
}

\DeclareSymbol{dotdiff2}{2}{
\draw (0,3) -- (0,-0.1) ;
\draw (0,3) node[diff] {};
\draw (0,-0.1) node[not] {};
}

\DeclareSymbol{dotdiff2mini}{0}{
\draw (0,1.2) -- (0,0) ;
\draw (0,1.2) node[diffmini] {};
\draw (0,0) node[not] {};
}

\DeclareSymbol{dotdiffstraight}{0}{
\draw  (0,3) -- (0,-0.1) ;
\draw (0,3) node[diff] {};
\draw (0,-0.1) node[not] {};
}

\DeclareSymbol{arbre1}{0}{
\draw  (0,0) -- (1.5,1.5) ;
\draw (1.5,1.5) node[not] {};
\draw (0,0) node[not] {};
}

\DeclareSymbol{arbre2}{0}{
\draw  (0,0) -- (1.5,1.5) ;
\draw[kernelsmod] (0,0) -- (-1.5,1.5);
\draw (1.5,1.5) node[not] {};
\draw (0,0) node[not] {};
\draw (-1.5,1.5) node[xi] {};
}

\DeclareSymbol{arbre3}{0}{
\draw  (0,0) -- (1.5,1.5) ;
\draw[kernelsmod] (1.5,1.5) -- (0,3);
\draw (0,0) node[not] {};
\draw (1.5,1.5) node[not] {};
\draw (0,3) node[xi] {};
}

\DeclareSymbol{treeeval}{0}{
\draw (0,0) -- (1,1);
\draw (0,0) node[xi] {};
\draw (1.25,1.25) node[xi] {};
\draw (-0.6,0.6) node[]{\tiny{$i$}};
\draw (0.65,1.85) node[]{\tiny{$j$}};
}

\DeclareSymbol{testeval}{0}{
\draw (0,0) -- (1,1);
\draw (0,0) -- (-1,1);
\draw (0,0) node[xi] {};
\draw (1.25,1.25) node[xi] {};
\draw (-1.25,1.25) node[xi] {};
\draw (-0.6,-0.6) node[]{\tiny{$i$}};
\draw (0.65,1.85) node[]{\tiny{$j$}};
\draw (-1.95,1.85) node[]{\tiny{$k$}};
}

\DeclareSymbol{treeeval2}{0}{
\draw[kernelsmod] (-0.25,-1) -- (1,0.5) ;
\draw[kernelsmod] (1,0.5) -- (-0.25,2);
\draw (1,0.5) node[diff2] {};
\draw (-0.25,-1) node[not] {};
\draw (-0.25,2) node[xi] {};
\draw (-0.6,1.2) node[]{\tiny{1}};
}

\DeclareSymbol{arbreact}{1}{
\draw (0,0) node[not] {};
\draw[kernelsmod] (0,0) -- (1,1);
\draw[kernelsmod] (0,0) -- (-1,1);
\draw (-1,1) node[xic] {};
\draw  (0,2) -- (1,1) ;
\draw (0,2) node[xic] {};
\draw (1,1) node[xi] {};
}

\DeclareSymbol{arbreact1}{0}{
\draw (0,-1.5) -- (0,0);
\draw[kernelsmod] (0,0) -- (1,1);
\draw[kernelsmod] (0,0) -- (-1,1);
\draw  (0,2) -- (1,1) ;
\draw (0,-1.5) node[diff] {};
\draw (0,0) node[not] {};
\draw (-1,1) node[xic] {};
\draw (0,2) node[xic] {};
\draw (1,1) node[xi] {};
}

\DeclareSymbol{arbreact2}{0}{
\draw (0,-0.75) -- (-1,0.5); 
\draw (0,-0.75) -- (1,0.5);
\draw (0,1.5) -- (1,0.5);
\draw (0,1.5) node[xic] {};
\draw (1,0.5) node[xi] {};
\draw (-1,0.5) node[xic] {};
\draw (0,-0.75) node[diff] {};
}

\DeclareSymbol{arbreact3}{0}{
\draw[kernelsmod] (0,-0.75) -- (-1,0.5); 
\draw[kernelsmod] (0,-0.75) -- (1,0.5);
\draw (0,1.75) -- (1,0.5);
\draw (2,1.75) -- (1,0.5);
\draw (0,1.75) node[xic] {};
\draw (1,0.5) node[diff] {};
\draw (-1,0.5) node[xic] {};
\draw (2,1.75) node[xi] {};
\draw (0,-0.75) node[not] {};
}

\DeclareSymbol{pre_im_1}{0}{
\draw[kernels2] (0,-0.5) node[not] {} -- (-0.6,0.5) ;
\draw[kernels2] (0,-0.5) -- (0.6,0.5);
\draw (0,1.1)  -- (-0.55,2);
\draw (0,1.1)  -- (0.55,2);
\draw (0,0.7) node[g] {};
\draw (0,2.2) node[g] {};
}

\DeclareSymbol{disconnect}{0}{
\draw[kernels2] (0,-0.5) node[not] {} -- (-0.6,0.5) ;
\draw[kernels2] (0,-0.5) -- (0.6,0.5);
\draw (-0.55,1.1)  -- (-0.55,2.3);
\draw (0.55,2.3) -- (0.55,1.5) -- (1.2,1.5) -- (1.2,3.5) -- (0.55,3.5) -- (0.55,2.7);
\draw (0,0.7) node[g] {};
\draw (0,2.5) node[g] {};
}

\DeclareSymbol{pre_im_2}{2}{\draw[kernels2] (0,0) node[not] {} -- (-1,1) node[not] {};
\draw[kernels2] (0,0) -- (1,1) node[not] {};
\draw[kernels2] (-1,1) -- (-1.5,2.5);
\draw[kernels2] (-1,1) -- (-0.5,2.5);
\draw[kernels2] (1,1) -- (0.5,2.5);
\draw[kernels2] (1,1) -- (1.5,2.5);
\draw (-1,2.7) node[g] {};
\draw (1,2.7) node[g] {};
}

\DeclareSymbol{CX_rec}{0}{
\draw [black] (-0.3,1) to (-0.3,-0.3);
\draw [black] (0.3,1) to (0.3,-0.3);
\draw [black] (-0.3,1) to (-0.3,2.3);
\draw [black] (0.3,1) to (0.3,2.3);
\draw (0,1) node[rec] {};
}

\DeclareSymbol{CX_cerc}{0}{
\draw [black] (0,1) to (0,-0.3);
\draw (0,1) node[cerc] {};
}

\DeclareSymbol{proof0}{0}{
\draw (0,-3) node[] {\tiny$\tau_0$};
\draw (0,-2.3) -- (0,0.5);`
\draw (0,0.5) -- (-1.5,2.5);
\draw (0,0.5) -- (1.5,2.5);
\draw (-1,-2) node[] {\tiny$v$};
\draw (-1.5,3.1) node[] {\tiny$\tau_1$};
\draw (1.5,3.1) node[] {\tiny$\tau_2$};
\draw (0,0.5) node[diff] {};
}

\DeclareSymbol{proof0b}{0}{
\draw (0,0.5) -- (-1.5,2.5);
\draw (0,0.5) -- (1.5,2.5);
\draw (-1.5,3.1) node[] {\tiny$\tau_1$};
\draw (1.5,3.1) node[] {\tiny$\tau_2$};
\draw (0,0.5) node[diff] {};
}

\DeclareSymbol{proof}{0}{
\draw[kernelsmod] (-2,3) -- (0,0);
\draw[kernelsmod] (2,3) -- (0,0);
\draw (0,0) node[not] {};
}

\DeclareSymbol{prooftri}{0}{
\draw[kernelsmod] (-2,3) -- (0,0);
\draw[kernelsmod] (0,4) -- (0,0);
\draw (2,2.7)--(0,0);
\draw (0,0) node[not] {};
}

\DeclareSymbol{proofqua}{0}{
\draw[kernelsmod] (-3,3) -- (0,0);
\draw[kernelsmod] (-1,4) -- (0,0);
\draw (1,3.6)--(0,0);
\draw (3,2.7)--(0,0);
\draw (0,0) node[not] {};
}

\DeclareSymbol{proof1_1}{0}{
\draw (0,-2.7) node[] {\tiny$\tau_0$};
\draw (0,-2) -- (0,0.5);`
\draw[kernelsmod] (0,0.5) -- (-1.5,2.5); 
\draw[kernelsmod] (0,0.5) -- (1.5,2.5);
\draw (-1,-1.7) node[] {\tiny$v$};
\draw (-1.5,3.1) node[] {\tiny$\tau_1$};
\draw (1.5,3.1) node[] {\tiny$\tau_2$};
\draw (0,0.5) node[not] {};
}

\DeclareSymbol{proof1b_1}{0}{
\draw[kernelsmod] (0,0.5) -- (-1.5,2.5); 
\draw[kernelsmod] (0,0.5) -- (1.5,2.5);
\draw (-1.5,3.1) node[] {\tiny$\tau_1$};
\draw (1.5,3.1) node[] {\tiny$\tau_2$};
\draw (0,0.5) node[not] {};
}

\DeclareSymbol{proof1_2}{0}{
\draw (0,-2.7) node[] {\tiny$\tau_0$};
\draw[kernelsmod] (0,-2) -- (0,0.5);`
\draw[kernelsmod] (0,0.5) -- (-1.5,2.5); 
\draw[kernelsmod] (0,0.5) -- (1.5,2.5);
\draw (-1,-1.7) node[] {\tiny$v$};
\draw (-1.5,3.1) node[] {\tiny$\tau_1$};
\draw (1.5,3.1) node[] {\tiny$\tau_2$};
\draw (0,0.5) node[not] {};
}

\DeclareSymbol{proof2_1}{0}{
\draw (0,-2.7) node[] {\tiny$\tau_0$};
\draw (0,-2) -- (-1.8,-0.3);
\draw (-1.8,0.7) -- (0,2.5); 
\draw (1,-1.7) node[] {\tiny$v$};
\draw (-2.5,1.2) node[] {\tiny$r_1$};
\draw (0,3.1) node[] {\tiny$\tau_2$};
\draw (-2.5,0) node[] {\tiny$\tau_1$};
}

\DeclareSymbol{proof2b_1}{0}{
\draw (-1.8,0.7) -- (0,2.5); 
\draw (-2.5,1.2) node[] {\tiny$r_1$};
\draw (0,3.1) node[] {\tiny$\tau_2$};
\draw (-2.5,0) node[] {\tiny$\tau_1$};
}

\DeclareSymbol{proof2b_2}{0}{
\draw (-1.8,0.7) -- (0,2.5); 
\draw (-2.5,1.2) node[] {\tiny$r_2$};
\draw (0,3.1) node[] {\tiny$\tau_1$};
\draw (-2.5,0) node[] {\tiny$\tau_2$};
}

\DeclareSymbol{proof2_2}{0}{
\draw (0,-2.7) node[] {\tiny$\tau_0$};
\draw[kernelsmod] (0,-2) -- (-1.8,-0.3);
\draw (-1.8,0.7) -- (0,2.5); 
\draw (1,-1.7) node[] {\tiny$v$};
\draw (-2.5,1.2) node[] {\tiny$r_1$};
\draw (0,3.1) node[] {\tiny$\tau_2$};
\draw (-2.5,0) node[] {\tiny$\tau_1$};
}

\DeclareSymbol{proof3_1}{0}{
\draw (0,-2.7) node[] {\tiny$\tau_0$};
\draw (0,-2) -- (1.8,-0.3);
\draw (1.8,0.7) -- (0,2.5); 
\draw (-1,-1.7) node[] {\tiny$v$};
\draw (2.5,1.2) node[] {\tiny$r_2$};
\draw (0,3.1) node[] {\tiny$\tau_1$};
\draw (2.5,0) node[] {\tiny$\tau_2$};
}

\DeclareSymbol{proof3_2}{0}{
\draw (0,-2.7) node[] {\tiny$\tau_0$};
\draw[kernelsmod] (0,-2) -- (1.8,-0.3);
\draw (1.8,0.7) -- (0,2.5); 
\draw (-1,-1.7) node[] {\tiny$v$};
\draw (2.5,1.2) node[] {\tiny$r_2$};
\draw (0,3.1) node[] {\tiny$\tau_1$};
\draw (2.5,0) node[] {\tiny$\tau_2$};
}

\DeclareSymbol{proof4_1}{0}{
\draw (0,-2) node[] {\tiny$\tau_0$};
\draw (0,-1.3) -- (-1.5,1.8); 
\draw  (0,-1.3) -- (1.5,1.8);
\draw (-1,-1) node[] {\tiny$v$};
\draw (-1.5,2.4) node[] {\tiny$\tau_1$};
\draw (1.5,2.4) node[] {\tiny$\tau_2$};
}

\DeclareSymbol{proof4_2}{0}{
\draw (0,-2) node[] {\tiny$\tau_0$};
\draw[kernelsmod] (0,-1.3) -- (-1.5,1.8); 
\draw  (0,-1.3) -- (1.5,1.8);
\draw (-1,-1) node[] {\tiny$v$};
\draw (-1.5,2.4) node[] {\tiny$\tau_1$};
\draw (1.5,2.4) node[] {\tiny$\tau_2$};
}

\DeclareSymbol{proof4_3}{0}{
\draw (0,-2) node[] {\tiny$\tau_0$};
\draw (0,-1.3) -- (-1.5,1.8); 
\draw[kernelsmod]  (0,-1.3) -- (1.5,1.8);
\draw (-1,-1) node[] {\tiny$v$};
\draw (-1.5,2.4) node[] {\tiny$\tau_1$};
\draw (1.5,2.4) node[] {\tiny$\tau_2$};
}

\DeclareSymbol{prooftriple}{0}{
\draw (0,-2.7) node[] {\tiny$\tau_0$};
\draw (0,-2) -- (0,0.25);`
\draw (0,0.25) -- (-1.5,2.5); 
\draw (0,0.25) -- (1.5,2.5);
\draw (0,0.25) -- (0,2.5);
\draw (-1,-1.7) node[] {\tiny$v$};
\draw (-2,3.1) node[] {\tiny$\tau_1$};
\draw (0,3.1) node[] {\tiny$\tau_2$};
\draw (2,3.1) node[] {\tiny$\tau_3$};
\draw (0,0.25) node[diff] {};
}

\DeclareSymbol{prooftriple1}{0}{
\draw (0,-2.7) node[] {\tiny$\tau_0$};
\draw (0,-2) -- (0,0.25);
\draw[kernelsmod] (0,0.25) -- (-1.5,2.5); 
\draw (0,0.25) -- (1.5,2.5);
\draw[kernelsmod] (0,0.25) -- (0,2.5);
\draw (-1,-1.7) node[] {\tiny$v$};
\draw (-2,3.1) node[] {\tiny$\tau_1$};
\draw (0,3.1) node[] {\tiny$\tau_2$};
\draw (2,3.1) node[] {\tiny$\tau_3$};
\draw (0,0.25) node[not] {};
}

\DeclareSymbol{prooftripleperm1}{0}{
\draw (0,-2.7) node[] {\tiny$\tau_0$};
\draw (0,-2) -- (0,0.25);
\draw[kernelsmod] (0,0.25) -- (-3.5,2.5); 
\draw (0,0.25) -- (2.5,2.5);
\draw[kernelsmod] (0,0.25) -- (0,2.5);
\draw (-1,-1.7) node[] {\tiny$v$};
\draw (-3,3.1) node[] {\tiny$\tau_{\sigma_1}$};
\draw (0,3.1) node[] {\tiny$\tau_{\sigma_2}$};
\draw (3,3.1) node[] {\tiny$\tau_{\sigma_3}$};
\draw (0,0.25) node[not] {};
}

\DeclareSymbol{proofdouble}{0}{
\draw (0,0.5) -- (-1.5,2.5);
\draw (0,0.5) -- (1.5,2.5);
\draw (-1.5,3.1) node[] {\tiny$\tau_1$};
\draw (1.5,3.1) node[] {\tiny$\tau_2$};
\draw (0,0.5) node[diff] {};
}

\DeclareSymbol{proofquadruple}{0}{
\draw (0,0) -- (-2.5,2.5); 
\draw (0,0) -- (-1,2.5);
\draw (0,0) -- (1,2.5);
\draw (0,0) -- (2.5,2.5);
\draw (-3,3.1) node[] {\tiny$\tau_1$};
\draw (-1,3.1) node[] {\tiny$\tau_2$};
\draw (1,3.1) node[] {\tiny$\tau_3$};
\draw (3,3.1) node[] {\tiny$\tau_4$};
\draw (0,0) node[diff] {};
}

\DeclareSymbol{proofquadruple1}{0}{
\draw[kernelsmod] (0,0) -- (-2.5,2.5); 
\draw[kernelsmod] (0,0) -- (-1,2.5);
\draw (0,0) -- (1,2.5);
\draw (0,0) -- (2.5,2.5);
\draw (-3,3.1) node[] {\tiny$\tau_1$};
\draw (-1,3.1) node[] {\tiny$\tau_2$};
\draw (1,3.1) node[] {\tiny$\tau_3$};
\draw (3,3.1) node[] {\tiny$\tau_4$};
\draw (0,0) node[not] {};
}

\DeclareSymbol{proofquadrupleperm1}{0}{
\draw[kernelsmod] (0,-0.4) -- (-3.5,2.3); 
\draw[kernelsmod] (0,-0.4) -- (-1,2.3);
\draw (0,-0.4) -- (1,2.5);
\draw (0,-0.4) -- (3.5,2.5);
\draw (-4.5,3.1) node[] {\tiny$\tau_{\sigma_1}$};
\draw (-1.5,3.1) node[] {\tiny$\tau_{\sigma_2}$};
\draw (1.5,3.1) node[] {\tiny$\tau_{\sigma_3}$};
\draw (4.5,3.1) node[] {\tiny$\tau_{\sigma_4}$};
\draw (0,-0.4) node[not] {};
}

\DeclareMathAlphabet{\mathpzc}{OT1}{pzc}{m}{it}

\def\eqref#1{(\ref{#1})}

\makeatletter 
\newcommand*{\bigcdot}{}
\DeclareRobustCommand*{\bigcdot}{%
  \mathbin{\mathpalette\bigcdot@{}}%
}
\newcommand*{\bigcdot@scalefactor}{.5}
\newcommand*{\bigcdot@widthfactor}{1.15}
\newcommand*{\bigcdot@}[2]{%
  \sbox0{$#1\vcenter{}$}
  \sbox2{$#1\cdot\m@th$}%
  \hbox to \bigcdot@widthfactor\wd2{%
    \hfil
    \raise\ht0\hbox{%
      \scalebox{\bigcdot@scalefactor}{%
        \lower\ht0\hbox{$#1\bullet\m@th$}%
      }%
    }%
    \hfil
  }%
}
\makeatother

\tcbset
{colframe=boxcolor,colback=symbols!7!pagebackground,coltext=pageforeground,
fonttitle=\bfseries,nobeforeafter,center title,size=fbox,boxsep=1.5pt,
top=0mm,bottom=0mm,boxsep=0mm,tcbox raise base}

\def\two{{\<generic>\kern0.05em\<genericb>}}
\def\twoI{{\<Ito>\kern0.05em\<Itob>}}

\def\mail#1{\burlalt{#1}{mailto:#1}}
\title{Ramification of Volterra-type Rough Paths}

\author{Yvain Bruned$^1$, Foivos Katsetsiadis$^2$}
\institute{ 
 IECL (UMR 7502), Université de Lorraine
 \and University of Edinburgh \\
Email:\ \begin{minipage}[t]{\linewidth}
\mail{yvain.bruned@univ-lorraine.fr}
\\ \mail{F.I.Katsetsiadis@sms.ed.ac.uk}
\end{minipage}}

\begin{document}

\maketitle

\begin{abstract}

We extend the new approach introduced in \cite{HT} and \cite{HT2} for dealing with stochastic Volterra equations using the ideas of Rough Path theory and prove global existence and uniqueness results. The main idea of this approach is simple: Instead of the iterated integrals of a path comprising the data necessary to solve any equation driven by that path, now iterated integral convolutions with the Volterra kernel comprise said data. This leads to the corresponding abstract objects called Volterra-type Rough Paths, as well as the notion of the convolution product, an extension of the natural tensor product used in Rough Path Theory.

\end{abstract}

\tableofcontents

\section{Introduction}

\ \ \ \ Volterra Equations comprise a thoroughly studied class of differential equations with wide applicability in physics, engineering and other sciences, capable of adequately capturing the behavior of a wide range of natural models. Introducing stochasticity to these models via a random driving noise naturally yields the corresponding notion of Stochastic Volterra Equations. These equations are typically of the form:
\begin{equs}\label{Equation}
   u(t)= u_{0} + \sum_{i=0}^d \int_{0}^{t} k(t,r) f_i(u_r) dq^{i}_r \ , \quad u_0 \in \R^{e}
\end{equs}
where $k$ is a kernel that obeys the analytic condition \ref{condition_k} and is allowed to be singular in the diagonal $t=r$, the $f_i$ are sufficiently regular vector fields on $ \R^{e} $ and $q \in C^{\alpha}([0,T]; \R^{d+1})$ a.s. is a stochastic process on $\R^{d+1}$ with $q^{0}_r = r$ that is Hölder regular of some degree $\alpha \in (0,1]$ almost surely. Such equations are of independent theoretical interest, but also arise in the description of the dynamics of various natural systems and are becoming increasingly popular with the recent advent of so called Rough Volatility Models in Mathematical Finance after the paper \cite{GJR}, where it is empirically demonstrated that volatility is best described by a rough process. We refer the reader to Rosenbaum and coworkers in \cite{EFR}, \cite{ER16}, \cite{ER17}, who, based on stylized facts of modern market microstructure, construct a sequence of Hawkes processes suitably rescaled in time and space that converges in law to a rough volatility model of rough Heston form with the variance process of the underlying asset obeying a stochastic Volterra Equation. Therefore, there is ample interest in developing a robust solution theory to these equations.

  In cases where the noise is a Brownian Motion or a continuous semimartingale process, classical interpretations of these equations, such as the Ito interpretation by means of the classical stochastic integral are very much sufficient. Indeed, this approach has been undertaken in \cite{OZ93} and \cite{Zha10} where these equations are treated via the means of classical stochastic calculus. A mathematical challenge is presented when the driving noise is of wilder nature, yielding rougher realizations as its sample paths and/or appears without the semimartingale property, necessitating the search for finer interpretations of such equations. This challenge is not contained in the context of Volterra Equations, but can be seen as pertaining to a wide class of stochastic PDEs. Various tools have been introduced to tackle this problem and naturally, these tools have also been applied to yield appropriate solution theories to \ref{Equation}. We shall give a brief overview of these developments that span the last three decades.

 In 1998, T. Lyons introduces the Rough Path calculus in the seminal work \cite{Lyons98}. Lyons treats the problem by enhancing the process into an object known as a Rough Path, thereby introducing the higher-order calculi that allow one to obtain a "rough" formulation of the given equations. In 2004, Gubinelli introduces the concept of a controlled Rough Path in \cite{Gubinelli2004} and later on that of a Branched Rough Path in his work \cite{Gub06}. Advancements rapidly follow. In 2011, in one of the most impressive applications of the theory Hairer uses Rough Paths in \cite{KPZ} to provide a solution theory to the KPZ equation, a landmark achievement in the field of singular stochastic PDEs. Shortly after, Gubinelli, Imkeller and Perkowski, in \cite{GIP15} also introduce techniques of paracontrolled calculus in the study of singular SPDEs. Advancements culminate around the same time, when, in a major synthesis of ideas in \cite{reg}, Hairer develops the theory of regularity structures, being one of the first to introduce renormalization techniques in the field. This, together with the subsequent \cite{BHZ,CH16,BCCH}, form a body of work that covers a large class of parabolic stochastic PDEs. Finally, as an addition to the powerful arsenal developed through the theory of regularity structures, higher order paracontrolled calculus is also formulated in \cite{BB19}, generalising the approach introduced in \cite{GIP15}.
 
 The study of stochastic Volterra Equations follows these developments. In \cite{DT09} and \cite{DT11} Deya and Tindel use ideas of Rough Path theory for the treatment of non-singular Volterra equations. Furthermore, in their recent work \cite{PT18} Pr\"omel and Trabs treat the first order case by use of paracontrolled calculus. 

 As mentioned above, Regularity Structures have proven to be one of the most potent tools for assigning a well-posed interpretation to a large class of stochastic PDEs. Indeed, a more powerful approach introduced in \cite{BFGJS} and elaborated upon in \cite{BFG20} has been to interpret stochastic Volterra Equations via means of the theory of Regularity Structures; this has been done down to Hölder exponent $1/3$ when the driving noise is a fractional Brownian motion, with the expectation that methods therein are amenable to generalization in the case of arbitrarily low exponent. One should then be able to obtain existence and uniqueness results that are global (see for instance the first paragraph in Appendix C of \cite{BFGJS} for global existence).

 In this work, we are interested in extending the approach introduced lately in \cite{HT} and \cite{HT2}, which is to generalize the ideas of Rough Path theory in order to treat the case of stochastic Volterra Equations. The idea is that instead of the iterated integrals of a path one now keeps track of iterated integral convolutions with the Volterra kernel. An object encoding this information is called a Volterra-type Rough Path, or simply a Volterra Path. These objects are a generalization of Branched Rough Paths and satisfy a generalization of Chen's relation. In order to formulate this generalized relation a convolution product operation is introduced, serving as an extension of the natural tensor product used in Rough Path Theory.

 In their works \cite{HT} and \cite{HT2} the authors only treat the cases of Hölder regularity that is higher than  $1/3$ and $1/4$ respectively. Furthermore, the Hopf-algebraic framework necessary for the description of the objects has not yet been clarified by the authors. In particular, while use of the techniques of Branched Rough Paths is attempted, the classical Connes-Kreimer Hopf algebra is not a suitable choice in the context of the results. In this work, we introduce a general framework and extend these results in the case of arbitrarily low Hölder exponent. In the process, we provide an application of a Hopf-algebraic structure that yields a robust and general description of the objects at hand. The idea for this structure is based on a plugging coproduct used in \cite{deformation} for recovering the algebraic structures of \cite{BHZ}. It can be understood as a Connes-Kreimer type coproduct where one keeps along the edges that would normally be lost when performing an admissible cut. Pinpointing certain properties of convolutional integrals in a sufficiently regular setting as well as drawing from the ideas of Branched Rough Paths together with the more suitable framework provided by this algebraic structure, we are able to formulate the theory on any order and to extend the main results obtained in \cite{HT2}.

 We will now proceed to give a brief outline for the contents of the paper. In the next section, we begin by setting up the analytic and algebraic framework of interest by describing the analytic conditions imposed on our kernels and the ambient spaces for our objects, as well as introducing the Hopf algebraic structure that will be used for the description of the objects involved in our rough formulation. We also prove some motivating propositions that hold true when one is dealing with sufficiently regular functions (see Propositions \ref{propchen} and \ref{associativity_smooth}). Then, we move to the main objective of this section which is to introduce a convolution-type operation on our spaces of interest, which we call a convolution product and shall denote by $ \star$. This is done after recalling some results that were proven in \cite{HT} and \cite{HT2}, at which point we proceed to prove the main result of the section, which is Theorem \ref{Convolution Product}. The statement of the theorem can be seen as giving the definition for the convolution product operation.

 This operation is vital for our description of the new objects introduced, which we shall call Branched Rough Paths of Volterra type or simply Volterra Paths. These objects, while resembling the Branched Rough Paths introduced by Gubinelli in \cite{Gub06}, nevertheless serve as to model rough versions of convolutional integral expressions that in general cannot be operated upon by mere tensorization and be expected to yield another "integral" of the same form. We will therefore need an operation that will serve as a sort of "integral convolution" for our objects. The $\star$ operation introduced fulfills this role.
 
 In the third section, we begin by giving the proper definition for a controlled Rough Path of Volterra type in Definition~\ref{CVRP} and show how to integrate these objects against a given Volterra Path in Theorem~\ref{Rough Integration}. We also prove a version of the "chain rule" for our Rough Path calculus (see Proposition~\ref{Chain rule}), showing how to lift the composition of a function $y$ controlled by a given Volterra Path with a sufficiently regular function $f$ into the space of controlled paths. This will allow us to finally formulate and prove our main result, which is Theorem \ref{Existence and Uniqueness} on the existence and uniqueness of solutions to Rough Volterra Equations. 
 
 \subsection*{Acknowledgements}
 
 {\small
 The authors are very grateful to the referees for their careful reading of the manuscript.
 We are particularly indebted to the referee who aks clarifications on the definitions and the proofs concerning the convolution product. This leads to substantial improvements in the clarity of the
 exposition.
 }

\section{The convolution product}

\ \ We begin this section by presenting the convolution product operation proposed in \cite{HT} at the level of iterated Volterra integrals and proceed to describe how it can be generalized to ramified integral expressions with tree-indexed iterations. We will use this in the next section to define the concept of a Volterra Rough Path which is an abstraction of the collection of ramified Volterra integrals corresponidng to a path: it will be comprised of a sequence of functions that is \emph{stipulated} to satisfy a convolutional variant of Chen's relation. To accomplish this, we will face the challenge of defining this convolution product as an operation on the spaces these function terms reside in. This challenge has been tackled in \cite{HT} for Hölder exponents $\rho > 1/3$. We proceed to construct a theoretical framework that is more general and works for arbitrarily low Hölder exponent. Letting $ T > 0 $, we will denote by $\Delta_{n}([0,T])$ the subset $\{ 0 < t_1 < \cdots < t_n < T \}$ of $\R^{n}$ and by $ Q_n([0,T]) $ the $n$-th dimensional hypercube $\{ 0 < t_i < T, \, i \in \lbrace 1,...,n \rbrace \} \subset \R^{n}$. We will usually omit the interval $[0,T]$ and use $\Delta_n$ and $Q_n$ respectively.

Let us recall the analytic condition $(\textbf{H})$ imposed on the kernel $k$ in \cite{HT}. The assumption is that $k : \Delta_2 \rightarrow \R$ is such that there exists $ \gamma \in (0,1) $ so that for all
$ (s,r,q,\tau) \in \Delta_4 $ and $ \eta, \beta \in [0,1] $, we have:
\begin{equs} \label{condition_k}\begin{aligned}
| k(\tau,r) | & \lesssim |\tau-r|^{-\gamma}
\\ 
|  k(\tau,r) - k(q,r) | & \lesssim |q-r|^{-\gamma -\eta} | \tau - q|^{\eta}
\\ 
|  k(\tau,r) - k(\tau,s) | & \lesssim |\tau-r|^{-\gamma -\eta} | r - s|^{\eta}
\\ 
|  k(\tau,r) - k(q,r) -  k(\tau,s) + k(q,s) | & \lesssim |q-r|^{-\gamma -\beta} | r - s|^{\beta}
\\ 
|  k(\tau,r) - k(q,r) -  k(\tau,s) + k(q,s) | & \lesssim |q-r|^{-\gamma -\eta} | r - q|^{\eta}.
\end{aligned}
\end{equs}
The kernel $k$ is the basic block for the construction of iterated convolutional integrals. To represent these, we will need to introduce the appropriate algebraic structure. We first introduce a natural space of decorated rooted trees. 
 Let $ \hat \CT$ (resp. $ \CT $) be the set of rooted trees with nodes decorated by $\{0,...,d\}$ (resp. decorated by $\{0,...,d\}$ except for the root which carries no decoration). We grade elements $\tau \in \hat \CT$ (resp. $  \CT$)  by the number $|\tau|$ of their nodes having a decoration and we set
\[
\CT_n:=\{\tau \in \CT: |\tau| \leq n\},
\qquad n\in\N.
\]
(and resp. for $\hat \CT$). We denote by $\hat \CF$ (resp. $\CF$) the set of forests, i.e. sets consisting of trees in $\hat \CT $ (resp. $\CT$). For any $h \in \CF$ or $h \in \hat \CF$ we shall denote by $E_{h}$ the set of its edges, by $N_{h}$ the set of its nodes and by $L_{h}$ the set of its leaves.

 Any rooted tree $\tau \in \hat \CT$, different from the empty tree $\one$, can be written in terms of the $B^{i}_+$-operators, $ i \in  \{0,...,d\}$ which connect the roots of the trees in a forest $\tau_1 \cdots \tau_n \in \hat \CF$ to a new root decorated by $i$. Indeed, given any $\tau \in \hat \CT$ other than $\one$, we have that $\tau = B^{i}_+(\tau_1 \cdots \tau_n)$ for some $\tau_{1}, ..., \tau_{n} \in \hat \CF$.
 Elements of $ \CT $ are described similarly except that we use the operator $B_+$ since their roots carry no decoration. We introduce another operator on the linear span of decorated trees denoted by $\CI_{i} $ which acts as follows: for $\tau  \in \CT$, the tree $ \CI_{i}(\tau) $ is given by grafting the root of $\tau $ onto a new root with no decoration and then decorating with $i$ the node of the new tree corresponding to the root of $\tau $. Below, we ilustrate the various operations:
 \begin{equs}
B_+( \begin{tikzpicture}[scale=0.2,baseline=0.1cm]
        \node at (0,1)  [dot,label= {[label distance=-0.2em]below: \scriptsize $ i $  } ] (root) {};
     \end{tikzpicture}  \begin{tikzpicture}[scale=0.2,baseline=0.1cm]
        \node at (0,1)  [dot,label= {[label distance=-0.2em]below: \scriptsize $ j $  } ] (root) {};
     \end{tikzpicture} )   =  \begin{tikzpicture}[scale=0.2,baseline=0.1cm]
        \node at (0,0)  [dot,label= {[label distance=-0.2em]below: \scriptsize  } ] (root) {};
         \node at (-1,2)  [dot,label={[label distance=-0.2em]left: \scriptsize  $ i $} ] (left) {};
          \node at (1,2)  [dot,label={[label distance=-0.2em]right: \scriptsize  $ j $} ] (right) {};
          \draw[kernel1] (left) to
     node [sloped,below] {\small }     (root);
      \draw[kernel1] (right) to
     node [sloped,below] {\small }     (root);
     \end{tikzpicture}, \quad B^{k}_+( \begin{tikzpicture}[scale=0.2,baseline=0.1cm]
        \node at (0,1)  [dot,label= {[label distance=-0.2em]below: \scriptsize $ i $  } ] (root) {};
     \end{tikzpicture}  \begin{tikzpicture}[scale=0.2,baseline=0.1cm]
        \node at (0,1)  [dot,label= {[label distance=-0.2em]below: \scriptsize $ j $  } ] (root) {};
     \end{tikzpicture} )   =  \begin{tikzpicture}[scale=0.2,baseline=0.1cm]
        \node at (0,0)  [dot,label= {[label distance=-0.2em]below: \scriptsize $ k $  } ] (root) {};
         \node at (-1,2)  [dot,label={[label distance=-0.2em]left: \scriptsize  $ i $} ] (left) {};
          \node at (1,2)  [dot,label={[label distance=-0.2em]right: \scriptsize  $ j $} ] (right) {};
          \draw[kernel1] (left) to
     node [sloped,below] {\small }     (root);
      \draw[kernel1] (right) to
     node [sloped,below] {\small }     (root);
     \end{tikzpicture},\quad \CI_k ( \begin{tikzpicture}[scale=0.2,baseline=0.1cm]
        \node at (0,0)  [dot,label= {[label distance=-0.2em]below: \scriptsize  } ] (root) {};
         \node at (-1,2)  [dot,label={[label distance=-0.2em]left: \scriptsize  $ i $} ] (left) {};
          \node at (1,2)  [dot,label={[label distance=-0.2em]right: \scriptsize  $ j $} ] (right) {};
          \draw[kernel1] (left) to
     node [sloped,below] {\small }     (root);
      \draw[kernel1] (right) to
     node [sloped,below] {\small }     (root);
     \end{tikzpicture}) = \begin{tikzpicture}[scale=0.2,baseline=0.1cm]
        \node at (0,0)  [dot,label= {[label distance=-0.2em]below: \scriptsize  } ] (root) {};
         \node at (0,2)  [dot,label={[label distance=-0.2em]left: \scriptsize  $ k $} ] (center) {};
          \node at (1,4)  [dot,label={[label distance=-0.2em]right: \scriptsize  $ j $} ] (left) {};
          \node at (-1,4)  [dot,label={[label distance=-0.2em]left: \scriptsize  $ i $} ] (right) {};
          \draw[kernel1] (center) to
     node [sloped,below] {\small }     (root);
      \draw[kernel1] (right) to
     node [sloped,below] {\small }     (center);
     \draw[kernel1] (left) to
     node [sloped,below] {\small }     (center);
     \end{tikzpicture}.
 \end{equs}
 Concerning the map $|\cdot |$ which counts the decorated nodes in a tree, we have for example:
 \begin{equs}
 |\begin{tikzpicture}[scale=0.2,baseline=0.1cm]
        \node at (0,0)  [dot,label= {[label distance=-0.2em]below: \scriptsize  } ] (root) {};
         \node at (0,2)  [dot,label={[label distance=-0.2em]left: \scriptsize  $ k $} ] (center) {};
          \node at (1,4)  [dot,label={[label distance=-0.2em]right: \scriptsize  $ j $} ] (left) {};
          \node at (-1,4)  [dot,label={[label distance=-0.2em]left: \scriptsize  $ i $} ] (right) {};
          \draw[kernel1] (center) to
     node [sloped,below] {\small }     (root);
      \draw[kernel1] (right) to
     node [sloped,below] {\small }     (center);
     \draw[kernel1] (left) to
     node [sloped,below] {\small }     (center);
     \end{tikzpicture} |  = |  \begin{tikzpicture}[scale=0.2,baseline=0.1cm]
        \node at (0,0)  [dot,label= {[label distance=-0.2em]below: \scriptsize $ k $  } ] (root) {};
         \node at (-1,2)  [dot,label={[label distance=-0.2em]left: \scriptsize  $ i $} ] (left) {};
          \node at (1,2)  [dot,label={[label distance=-0.2em]right: \scriptsize  $ j $} ] (right) {};
          \draw[kernel1] (left) to
     node [sloped,below] {\small }     (root);
      \draw[kernel1] (right) to
     node [sloped,below] {\small }     (root);
     \end{tikzpicture} |  =3.
 \end{equs}
 A decorated tree that is of the form $ \CI_k(\tau)$ for some $\tau  \in \CT$ is called a planted tree. The set of these trees is denoted by $ \mathcal{PT} $ and $ \mathcal{P}\mathcal{T}_{n} \subset \mathcal{PT} $ denotes the set of planted trees of size $n$.

\begin{remark} We can potentially consider a finite family of kernels $ (k_i)_{i \in I} $ satisfying condition $(\textbf{H})$. Then, our decorated trees should also come equipped with edge decorations indexed by $I$. This type of structure has been introduced in the context of Regularity Structures in \cite{BHZ}.
\end{remark}

We are now ready to give the precise definition of what we mean by the $h-th$ iterated convolutional integral of a path $q$, where $h$ is a given tree.

\begin{definition} \label{def_trees_integral}

Let $q$ be a path in $\mathcal{C}^{1}( [0,T]; \R^{d+1})$ and $k: \Delta_{2} \rightarrow \R$ a Volterra kernel satisfying the analytic bounds imposed in condition $(\textbf{H})$. Let $h \in \CT $ be a rooted tree with $n + 1$ vertices. Let $h^{\star} \in \hat \CF$ denote the forest one obtains after removing the root of $h$ and its adjacent edges. Then, using the convention that $r_{\rho} = \tau$, where $ \rho $ is the root of $ h $, we define the $h-th$ \emph{iterated Volterra integral} as a mapping $\mathbf{z}^{h}: \Delta_{3} \rightarrow {\R}$ given for $ \tau > t > s $ by 
\begin{equs}
\mathbf{z}^{h, \tau}_{ts} =  \int_{A_{ts}^h \subseteq \R^{n}} \prod_{ (i,j) \in E_h} k(r_{i}, r_{j}) \prod_{\ell \in N_{h^{\star}}} dq_{r_{\ell}}^{i_{\ell}}
\end{equs}
where $ i_{\ell} $ is the decoration attached to the node $ \ell $ and the domain of integration is the set
\begin{equs}
A^{h}_{ts} = \bigcap_{ (i,j) \in E_{h^{\star}}} \{ t > r_{i} > r_{j} > s \}
\end{equs}
i.e. the order relations defining the variable ranges are directly given by the partial ordering induced by the forest $h^{\star} \in \hat \CF$. 
Let $V \subset N_{h^{\star}}$ be of cardinality $m$. Then, we also define:
\begin{equs}
\bar{\textbf{z}}^{h, \tau}_{ts}((r_\ell)_{\ell \in V} ) = \prod_{w\in V} \dot{q}^{i_w}_{r_w}  \int_{A^{h}_{ts}((r_\ell)_{\ell \in V}) \subseteq \R^{n-m}} \prod_{ (i,j) \in E_h} k(r_{i}, r_{j}) \prod_{\ell \in N_{h^{\star}} \setminus V}d q^{i_{\ell}}_{r_{\ell}}
\end{equs}
where $ A^{h}_{ts}((r_\ell)_{\ell \in V}) $ correponds to $A^{h}_{ts}$ when one fixes the values of  $ (r_\ell)_{\ell \in V} $.

\end{definition}

\begin{example}[Linearly Iterated Volterra Integrals]

Let $h$ be the ladder tree with $n + 1$ vertices, with decorations $i_1,...,i_{n}$. Then, the \emph{iterated Volterra integral} of order $n$ is a mapping $\textbf{z}^{h}: \Delta_{3} \mapsto \R $ given by 
\begin{equs}
(s, t, \tau) \mapsto \textbf{z}^{h, \tau}_{ts} =  \int_{t>r_{n}>... >r_{1}>s}k( \tau, r_{n}) \prod_{j=1}^{n-1}k(r_{j+1}, r_{j})dq^{i_1}_{r_{1}}...dq^{i_n}_{r_{n}}.
\end{equs}
This expression corresponds to the case of Volterrra-type Rough Paths as they were originally introduced in \cite{HT}.
\end{example}
Before stating Proposition~\ref{propchen} which presents the convolution in the smooth case, we introduce a coproduct $\Delta$ on $\CH$, the linear span of $ \mathcal{F} $ and a coaction $\hat \Delta$ on $\hat \CH$, the linear span of $ \hat{\mathcal{F}} $. These are similar to the Butcher-Connes-Kreimer coproduct and their combinatorial nature will allow us to encode the various formal operations on the Taylor expansions that approximate our objects locally, such as formal integration. We define the plugging coproduct $ \Delta : \CH \rightarrow \CH \otimes \CH $ and the coaction $\hat \Delta : \hat \CH \rightarrow \hat \CH \otimes  \CH $ as follows:
\begin{equs}
\Delta h = \sum_{C \in \scriptsize{\text{Adm}}(h) }  R^C(h) \otimes \tilde{P}^C(h)
\end{equs} 
where $\text{Adm}(h)$ is the set of admissible cuts, which are defined as collections of edges
with the property that any path from the root to a leaf contains at most one edge of the collection. We will also write $\text{Adm}(h)'$ for the set of all admissible cuts of $h$ except the one where the entire tree is cut. We denote by $\tilde{P}^C(h)$ the pruned forest that is formed by collecting all the edges at or above the cut, including the ones upon which the cut was performed, so that the edges that were attached to the same node in $h$ are part of the same tree. The term $R^C(h)$  corresponds to the "trunk", that is the subforest formed by the edges not lying above the ones upon which the cut was performed. This coproduct differs from the classical Butcher-Connes-Kreimer coproduct introduced in \cite{Butcher,CK}. It can also be constructed via a plugging pre-Lie product, see \cite[Sec 3.3]{deformation}. 
We shall use Sweedler's notation for the coproduct of a forest $h$ and write:
\begin{equs} \label{Delta_12}
\Delta h = \sum_{(h)} h^{(1)} \otimes h^{(2)}.
\end{equs}

 \begin{remark}
One can provide an alternative recursive definition of $ \Delta $ and $ \hat{\Delta} $ as follows:
\begin{equs} \label{recursive_formula} \begin{aligned}
\Delta B_+(h_1 \cdot ... \cdot h_n) & = \sum_{I \subset \lbrace 1,..., n \rbrace} \left( B_+(\cdot) \otimes B_+(\prod_{i \in I} h_i) \right) \hat \Delta \prod_{ i \in \lbrace 1,...,n\rbrace \setminus I } h_i. \\
\hat \Delta B_+^{k}(h_1 \cdot ... \cdot h_n) & = \sum_{I \subset \lbrace 1,..., n \rbrace} \left( B_+^k(\cdot) \otimes B_+(\prod_{i \in I} h_i) \right) \hat \Delta \prod_{i \in  \lbrace 1,...,n\rbrace \setminus I } h_i.
\end{aligned}
\end{equs}
The proof of this identity is similar to the recursive formulation obtained for the Butcher-Connes-Kreimer coproduct.
\end{remark}

\begin{example} \label{coproduct_exmp} To illustrate with an example, we compute the result below for a given tree:

\begin{equs} \label{ex_plugging}
\begin{aligned}
\Delta \begin{tikzpicture}[scale=0.2,baseline=0.1cm]
        \node at (0,0)  [dot,label= {[label distance=-0.2em]below: \scriptsize  } ] (root) {};
         \node at (-1,2)  [dot,label={[label distance=-0.2em]left: \scriptsize  $ k $} ] (left) {};
          \node at (1,2)  [dot,label={[label distance=-0.2em]right: \scriptsize  $ \ell $} ] (right) {};
          \node at (0,4)  [dot,label={[label distance=-0.2em]right: \scriptsize  $ j $} ] (leftr) {};
          \node at (-2,4)  [dot,label={[label distance=-0.2em]left: \scriptsize  $ i $} ] (leftl) {};
          \draw[kernel1] (left) to
     node [sloped,below] {\small }     (root);
      \draw[kernel1] (right) to
     node [sloped,below] {\small }     (root);
     \draw[kernel1] (leftl) to
     node [sloped,below] {\small }     (left);
      \draw[kernel1] (leftr) to
     node [sloped,below] {\small }     (left);
     \end{tikzpicture} & = \begin{tikzpicture}[scale=0.2,baseline=0.1cm]
        \node at (0,0)  [dot,label= {[label distance=-0.2em]below: \scriptsize  } ] (root) {};
         \node at (-1,2)  [dot,label={[label distance=-0.2em]left: \scriptsize  $ k $} ] (left) {};
          \node at (1,2)  [dot,label={[label distance=-0.2em]right: \scriptsize  $ \ell $} ] (right) {};
          \node at (0,4)  [dot,label={[label distance=-0.2em]right: \scriptsize  $ j $} ] (leftr) {};
          \node at (-2,4)  [dot,label={[label distance=-0.2em]left: \scriptsize  $ i $} ] (leftl) {};
          \draw[kernel1] (left) to
     node [sloped,below] {\small }     (root);
      \draw[kernel1] (right) to
     node [sloped,below] {\small }     (root);
     \draw[kernel1] (leftl) to
     node [sloped,below] {\small }     (left);
      \draw[kernel1] (leftr) to
     node [sloped,below] {\small }     (left);
     \end{tikzpicture} \otimes \one + \one \otimes \begin{tikzpicture}[scale=0.2,baseline=0.1cm]
        \node at (0,0)  [dot,label= {[label distance=-0.2em]below: \scriptsize  } ] (root) {};
         \node at (-1,2)  [dot,label={[label distance=-0.2em]left: \scriptsize  $ k $} ] (left) {};
          \node at (1,2)  [dot,label={[label distance=-0.2em]right: \scriptsize  $ \ell $} ] (right) {};
          \node at (0,4)  [dot,label={[label distance=-0.2em]right: \scriptsize  $ j $} ] (leftr) {};
          \node at (-2,4)  [dot,label={[label distance=-0.2em]left: \scriptsize  $ i $} ] (leftl) {};
          \draw[kernel1] (left) to
     node [sloped,below] {\small }     (root);
      \draw[kernel1] (right) to
     node [sloped,below] {\small }     (root);
     \draw[kernel1] (leftl) to
     node [sloped,below] {\small }     (left);
      \draw[kernel1] (leftr) to
     node [sloped,below] {\small }     (left);
     \end{tikzpicture} + \begin{tikzpicture}[scale=0.2,baseline=0.1cm]
        \node at (0,0)  [dot,label= {[label distance=-0.2em]below: \scriptsize  } ] (root) {};
         \node at (0,2)  [dot,label={[label distance=-0.2em]left: \scriptsize  $ k $} ] (center) {};
          \node at (1,4)  [dot,label={[label distance=-0.2em]right: \scriptsize  $ j $} ] (left) {};
          \node at (-1,4)  [dot,label={[label distance=-0.2em]left: \scriptsize  $ i $} ] (right) {};
          \draw[kernel1] (center) to
     node [sloped,below] {\small }     (root);
      \draw[kernel1] (right) to
     node [sloped,below] {\small }     (center);
     \draw[kernel1] (left) to
     node [sloped,below] {\small }     (center);
     \end{tikzpicture}  \otimes \begin{tikzpicture}[scale=0.2,baseline=0.1cm]
        \node at (0,0)  [dot,label= {[label distance=-0.2em]below: \scriptsize  } ] (root) {};
         \node at (0,2)  [dot,label={[label distance=-0.2em]left: \scriptsize  $ \ell $} ] (center) {};
      \draw[kernel1] (root) to
     node [sloped,below] {\small }     (center);
     \end{tikzpicture} + \begin{tikzpicture}[scale=0.2,baseline=0.1cm]
        \node at (0,0)  [dot,label= {[label distance=-0.2em]below: \scriptsize  } ] (root) {};
         \node at (0,2)  [dot,label={[label distance=-0.2em]left: \scriptsize  $ \ell $} ] (center) {};
      \draw[kernel1] (root) to
     node [sloped,below] {\small }     (center);
     \end{tikzpicture}  \otimes \begin{tikzpicture}[scale=0.2,baseline=0.1cm]
        \node at (0,0)  [dot,label= {[label distance=-0.2em]below: \scriptsize  } ] (root) {};
         \node at (0,2)  [dot,label={[label distance=-0.2em]left: \scriptsize  $ k $} ] (center) {};
          \node at (1,4)  [dot,label={[label distance=-0.2em]right: \scriptsize  $ j $} ] (left) {};
          \node at (-1,4)  [dot,label={[label distance=-0.2em]left: \scriptsize  $ i $} ] (right) {};
          \draw[kernel1] (center) to
     node [sloped,below] {\small }     (root);
      \draw[kernel1] (right) to
     node [sloped,below] {\small }     (center);
     \draw[kernel1] (left) to
     node [sloped,below] {\small }     (center);
     \end{tikzpicture}
     \\ & + \begin{tikzpicture}[scale=0.2,baseline=0.1cm]
        \node at (0,0)  [dot,label= {[label distance=-0.2em]below: \scriptsize  } ] (root) {};
         \node at (-1,2)  [dot,label={[label distance=-0.2em]left: \scriptsize  $ k $} ] (left) {};
          \node at (0,4)  [dot,label={[label distance=-0.2em]right: \scriptsize  $ j $} ] (leftr) {};
          \draw[kernel1] (left) to
     node [sloped,below] {\small }     (root);
      \draw[kernel1] (leftr) to
     node [sloped,below] {\small }     (left);
     \end{tikzpicture}  \otimes \begin{tikzpicture}[scale=0.2,baseline=0.1cm]
        \node at (0,0)  [dot,label= {[label distance=-0.2em]below: \scriptsize  } ] (root) {};
         \node at (0,2)  [dot,label={[label distance=-0.2em]left: \scriptsize  $ i $} ] (center) {};
      \draw[kernel1] (root) to
     node [sloped,below] {\small }     (center);
     \end{tikzpicture}  \begin{tikzpicture}[scale=0.2,baseline=0.1cm]
        \node at (0,0)  [dot,label= {[label distance=-0.2em]below: \scriptsize  } ] (root) {};
         \node at (0,2)  [dot,label={[label distance=-0.2em]left: \scriptsize  $ \ell $} ] (center) {};
      \draw[kernel1] (root) to
     node [sloped,below] {\small }     (center);
     \end{tikzpicture} + \begin{tikzpicture}[scale=0.2,baseline=0.1cm]
        \node at (0,0)  [dot,label= {[label distance=-0.2em]below: \scriptsize  } ] (root) {};
         \node at (-1,2)  [dot,label={[label distance=-0.2em]left: \scriptsize  $ k $} ] (left) {};
          \node at (0,4)  [dot,label={[label distance=-0.2em]right: \scriptsize  $ i $} ] (leftr) {};
          \draw[kernel1] (left) to
     node [sloped,below] {\small }     (root);
      \draw[kernel1] (leftr) to
     node [sloped,below] {\small }     (left);
     \end{tikzpicture} \otimes \begin{tikzpicture}[scale=0.2,baseline=0.1cm]
        \node at (0,0)  [dot,label= {[label distance=-0.2em]below: \scriptsize  } ] (root) {};
         \node at (0,2)  [dot,label={[label distance=-0.2em]left: \scriptsize  $ j $} ] (center) {};
      \draw[kernel1] (root) to
     node [sloped,below] {\small }     (center);
     \end{tikzpicture}  \begin{tikzpicture}[scale=0.2,baseline=0.1cm]
        \node at (0,0)  [dot,label= {[label distance=-0.2em]below: \scriptsize  } ] (root) {};
         \node at (0,2)  [dot,label={[label distance=-0.2em]left: \scriptsize  $ \ell $} ] (center) {};
      \draw[kernel1] (root) to
     node [sloped,below] {\small }     (center);
     \end{tikzpicture}  +  \begin{tikzpicture}[scale=0.2,baseline=0.1cm]
        \node at (0,0)  [dot,label= {[label distance=-0.2em]below: \scriptsize  } ] (root) {};
         \node at (-1,2)  [dot,label={[label distance=-0.2em]left: \scriptsize  $ k $} ] (left) {};
          \node at (1,2)  [dot,label={[label distance=-0.2em]right: \scriptsize  $ \ell $} ] (right) {};
          \draw[kernel1] (left) to
     node [sloped,below] {\small }     (root);
      \draw[kernel1] (right) to
     node [sloped,below] {\small }     (root);
     \end{tikzpicture} \otimes \begin{tikzpicture}[scale=0.2,baseline=0.1cm]
        \node at (0,0)  [dot,label= {[label distance=-0.2em]below: \scriptsize  } ] (root) {};
         \node at (-1,2)  [dot,label={[label distance=-0.2em]left: \scriptsize  $ i $} ] (left) {};
          \node at (1,2)  [dot,label={[label distance=-0.2em]right: \scriptsize  $ j $} ] (right) {};
          \draw[kernel1] (left) to
     node [sloped,below] {\small }     (root);
      \draw[kernel1] (right) to
     node [sloped,below] {\small }     (root);
     \end{tikzpicture}
     \\ & +  \begin{tikzpicture}[scale=0.2,baseline=0.1cm]
        \node at (0,0)  [dot,label= {[label distance=-0.2em]below: \scriptsize  } ] (root) {};
         \node at (0,2)  [dot,label={[label distance=-0.2em]left: \scriptsize  $ k $} ] (center) {};
      \draw[kernel1] (root) to
     node [sloped,below] {\small }     (center);
     \end{tikzpicture}    \otimes \begin{tikzpicture}[scale=0.2,baseline=0.1cm]
        \node at (0,0)  [dot,label= {[label distance=-0.2em]below: \scriptsize  } ] (root) {};
         \node at (0,2)  [dot,label={[label distance=-0.2em]left: \scriptsize  $ \ell $} ] (center) {};
      \draw[kernel1] (root) to
     node [sloped,below] {\small }     (center);
     \end{tikzpicture}   \begin{tikzpicture}[scale=0.2,baseline=0.1cm]
        \node at (0,0)  [dot,label= {[label distance=-0.2em]below: \scriptsize  } ] (root) {};
         \node at (-1,2)  [dot,label={[label distance=-0.2em]left: \scriptsize  $ i $} ] (left) {};
          \node at (1,2)  [dot,label={[label distance=-0.2em]right: \scriptsize  $ j $} ] (right) {};
          \draw[kernel1] (left) to
     node [sloped,below] {\small }     (root);
      \draw[kernel1] (right) to
     node [sloped,below] {\small }     (root);
     \end{tikzpicture} + \begin{tikzpicture}[scale=0.2,baseline=0.1cm]
        \node at (0,0)  [dot,label= {[label distance=-0.2em]below: \scriptsize  } ] (root) {};
         \node at (-1,2)  [dot,label={[label distance=-0.2em]left: \scriptsize  $ k $} ] (left) {};
          \node at (1,2)  [dot,label={[label distance=-0.2em]right: \scriptsize  $ \ell $} ] (right) {};
          \node at (0,4)  [dot,label={[label distance=-0.2em]right: \scriptsize  $ j $} ] (leftr) {};
          \draw[kernel1] (left) to
     node [sloped,below] {\small }     (root);
      \draw[kernel1] (right) to
     node [sloped,below] {\small }     (root);
      \draw[kernel1] (leftr) to
     node [sloped,below] {\small }     (left);
     \end{tikzpicture} \otimes \begin{tikzpicture}[scale=0.2,baseline=0.1cm]
        \node at (0,0)  [dot,label= {[label distance=-0.2em]below: \scriptsize  } ] (root) {};
         \node at (0,2)  [dot,label={[label distance=-0.2em]left: \scriptsize  $ i $} ] (center) {};
      \draw[kernel1] (root) to
     node [sloped,below] {\small }     (center);
     \end{tikzpicture} + \begin{tikzpicture}[scale=0.2,baseline=0.1cm]
        \node at (0,0)  [dot,label= {[label distance=-0.2em]below: \scriptsize  } ] (root) {};
         \node at (-1,2)  [dot,label={[label distance=-0.2em]left: \scriptsize  $ k $} ] (left) {};
          \node at (1,2)  [dot,label={[label distance=-0.2em]right: \scriptsize  $ \ell $} ] (right) {};
          \node at (0,4)  [dot,label={[label distance=-0.2em]right: \scriptsize  $ i $} ] (leftr) {};
          \draw[kernel1] (left) to
     node [sloped,below] {\small }     (root);
      \draw[kernel1] (right) to
     node [sloped,below] {\small }     (root);
      \draw[kernel1] (leftr) to
     node [sloped,below] {\small }     (left);
     \end{tikzpicture} \otimes \begin{tikzpicture}[scale=0.2,baseline=0.1cm]
        \node at (0,0)  [dot,label= {[label distance=-0.2em]below: \scriptsize  } ] (root) {};
         \node at (0,2)  [dot,label={[label distance=-0.2em]left: \scriptsize  $ j $} ] (center) {};
      \draw[kernel1] (root) to
     node [sloped,below] {\small }     (center);
     \end{tikzpicture}.
     \end{aligned}
\end{equs}

\end{example}

In the sequel, we use an extension of the map $\Delta$ by identifying the nodes. We will consider the nodes of $ h^{(1)} $ and $h^{(2)}$ as a subset of those of $h$. Therefore, the roots of the trees in $ h^{(2)} $ will be identified with some nodes in $h^{(1)}$. We still use the same notation with this identification.
This property is crucial in order to define a convolution operation on tree-indexed iterated integrals. We also define the reduced coproduct $\tilde{\Delta}$ as follows:
\begin{equs}
\tilde{\Delta} h = \Delta h - h \otimes \one - \one \otimes h = \sum_{(h)'} h^{(1)} \otimes h^{(2)}.
\end{equs}

\begin{remark}
While the coproduct $\Delta$ introduced is slightly different from the original Connes-Kreimer coproduct, one has a projection of the algebra $\mathcal{H}$ onto the classical algebra of forests that intertwines their actions. Hence, while the the coproduct $\Delta$ carries slightly more information, one has a morphism onto the original Connes-Kreimer Hopf algebra. To be specific, we define the operator $B_- $ as the left inverse of $ B_+ $. Then $B_- $ is an algebra morphism from $\CF$ to $\hat \CF$ when these are equipped with the corresponding forest products and given $\Delta_{\text{CK}}$, the Connes-Kreimer coproduct, one has
\begin{equs} \label{connection_plugging}
\left(B_- \otimes B_- \right) \Delta = \Delta_{\tiny{\text{CK}}} B_-, \quad \left( \id \otimes B_- \right) \hat \Delta  = \hat \Delta_{\tiny{\text{CK}}} B_-
\end{equs}
where
\begin{equs}
\label{recursive_formula_CK} \begin{aligned}
\Delta_{\tiny{\text{CK}}} h_1 \cdot ... \cdot h_n & = \sum_{I \subset \lbrace 1,..., n \rbrace} \left( \id \otimes \prod_{i \in I} h_i \right) \hat \Delta_{\tiny{\text{CK}}} \prod_{i \in \lbrace 1,...,n\rbrace \setminus I} h_i. \\
\hat \Delta_{\tiny{\text{CK}}} B_+^{k}(h_1 \cdot ... \cdot h_n) & = \sum_{I \subset \lbrace 1,..., n \rbrace} \left( B_+^k(\cdot) \otimes \prod_{i \in I} h_i \right) \hat \Delta_{\tiny{\text{CK}}} \prod_{i \in  \lbrace 1,...,n\rbrace \setminus I} h_i.
\end{aligned}
\end{equs}
Indeed, by applying $ B_- $ to \eqref{recursive_formula} one gets
\begin{equs} \label{identity_CK_1}
\left(B_- \otimes B_- \right) \Delta B_+(h_1 \cdot ... \cdot h_n) & =  \sum_{I \subset \lbrace 1,..., n \rbrace} \left( \id \otimes\prod_{i \in I} h_i \right) \\ &  (\id \otimes B_-) \hat \Delta \prod_{i \in  \lbrace 1,...,n\rbrace \setminus I } h_i.
\end{equs}
\begin{equs} \label{identity_CK_2}
\left( \id  \otimes B_-\right)\hat \Delta B_+^{k}(h_1 \cdot ... \cdot h_n) & = \sum_{I \subset \lbrace 1,..., n \rbrace} \left( B_+^k(\cdot) \otimes \prod_{i \in I} h_i \right) \\ & \left( \id \otimes B_-\right) \hat \Delta \prod_{i \in  \lbrace 1,...,n\rbrace \setminus I } h_i.
\end{equs}
Using an inductive argument, one can show \eqref{connection_plugging} by using \eqref{identity_CK_1} and \eqref{identity_CK_2}.
\end{remark}

\vspace{10pt}

Before introducing the convolution operation, we extend Definition~\ref{def_trees_integral} to objects with forest indices.

\begin{definition} \label{def_forests_integral}
Let $ h = h_1 \cdot ... \cdot h_n $ be in $ \CF $. We define the $h-th$ \emph{iterated Volterra integral} as:
\begin{equs}
\textbf{z}^{h, \tau_1,...,\tau_n}_{ts} = \prod_{i=1}^n \textbf{z}^{h_i, \tau_i}_{ts}
\end{equs}
where $ \tau_1,...,\tau_n \in [s,t] $. The function $ (\tau_1,...,\tau_n ) \mapsto \prod_{i=1}^n \textbf{z}^{h_i, \tau_i}_{ts} $ is obtained by tensorising the functions $\textbf{z}^{h_i, \cdot}_{ts}$.
\end{definition}

\begin{remark} Definition~\ref{def_forests_integral} is the reason why we have introduced the spaces $ \CT $ and $ \CF $. Elements of $ \CT $ can be seen as forests and in this case one could see the elements of $ \CF $ as "forests of forests". Indeed, we have a bijection between trees in $ \CT $ and forests in $ \hat \CF $. Given a tree $ h \in \CT  $, one just needs to remove the root and the edges adjacent to it in order to obtain a forest. Below, we illustrate this bijection:
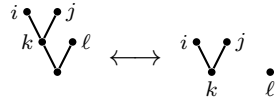
\begin{equs}
 \begin{tikzpicture}[scale=0.2,baseline=0.1cm]
        \node at (0,0)  [dot,label= {[label distance=-0.2em]below: \scriptsize  } ] (root) {};
         \node at (-1,2)  [dot,label={[label distance=-0.2em]left: \scriptsize  $ k $} ] (left) {};
          \node at (1,2)  [dot,label={[label distance=-0.2em]right: \scriptsize  $ \ell $} ] (right) {};
          \node at (0,4)  [dot,label={[label distance=-0.2em]right: \scriptsize  $ j $} ] (leftr) {};
          \node at (-2,4)  [dot,label={[label distance=-0.2em]left: \scriptsize  $ i $} ] (leftl) {};
          \draw[kernel1] (left) to
     node [sloped,below] {\small }     (root);
      \draw[kernel1] (right) to
     node [sloped,below] {\small }     (root);
     \draw[kernel1] (leftl) to
     node [sloped,below] {\small }     (left);
      \draw[kernel1] (leftr) to
     node [sloped,below] {\small }     (left);
     \end{tikzpicture}  \longleftrightarrow
     \begin{tikzpicture}[scale=0.2,baseline=0.1cm]
        \node at (0,0)  [dot,label= {[label distance=-0.2em]below: \scriptsize $ k $ } ] (root) {};
         \node at (-1,2)  [dot,label={[label distance=-0.2em]left: \scriptsize  $ i $} ] (left) {};
          \node at (1,2)  [dot,label={[label distance=-0.2em]right: \scriptsize  $ j $} ] (right) {};
          \draw[kernel1] (left) to
     node [sloped,below] {\small }     (root);
      \draw[kernel1] (right) to
     node [sloped,below] {\small }     (root);
     \end{tikzpicture} 
     \begin{tikzpicture}[scale=0.2,baseline=0.1cm]
        \node at (0,0)  [dot,label= {[label distance=-0.2em]below: \scriptsize $ \ell $ } ] (root) {};
     \end{tikzpicture}
\end{equs}
The use of the spaces $ \CT $ and $ \CF $ is also robust to the introduction of decorations on the edges. This structure is also reminiscent of the one used in \cite{BS} for dispersive equations, where a planted tree is associated to a frequence and planted trees with the same frequence can been seen as a tree.
\end{remark}

As mentioned, iterated Volterra integrals satisfy a generalized Chen identity. We prove this below.

\begin{proposition}\label{propchen}
Let $ h $ be a tree in $ \CT $ and $ (s,u,t,\tau) \in \Delta_4$, we have
\begin{equs} \label{convolution_chen}
\mathbf{z}^{h,\tau}_{ts} = \sum_{(h)}\mathbf{z}_{tu}^{h^{(1)},\tau} \star \mathbf{z}^{h^{(2)},\cdot}_{us}
\end{equs}
where  the convolution product $ \star $ is defined as follows
\begin{equs} \label{def_convol}
\mathbf{z}_{tu}^{h^{(1)},\tau} \star \mathbf{z}^{h^{(2)},\cdot}_{us} := \int_{\R^{m}} \bar{\mathbf{z}}_{tu}^{h^{(1)},\tau}((r_i)_{i \in V}) \prod_{i \in V} \mathbf{z}^{h^{(2)}_i,r_i}_{us} d r_i
\end{equs}
Here, $ V \subset N_{h^{\star}} $ is of cardinality $ m $ and is such that every $ i \in V $ considered as a node of $ h $ appears also as a root of a tree $ h^{(2)}_i $ in $ h^{(2)} $ and $ h^{(2)} = \prod_{i \in V} h_i^{(2)} $.
\end{proposition}

\begin{remark}
In the proposition above, the set $ V $ is implicitly given by the fact that we have identified the nodes of $ h^{(1)} $ and $ h^{(2)} $ with those of $ h $. This avoids carrying along notation of the type:
\begin{equs}
\mathbf{z}_{tu}^{h^{(1)},\tau} \star_{f_V} \mathbf{z}^{h^{(2)},\cdot}_{us}
\end{equs}
where here $V$ is a subset of the nodes of $h^{(1)}$ with cardinality equal to the number of roots in $h^{(2)}$ and $f_V   $ is a bijection between $ V $ and the roots of  $ h^{(2)} $.
\end{remark}

\begin{proof}
Let $h$ be a tree with $n +1$ vertices. Using the convention that $r_{\varrho} = \tau$ where $ \varrho  $ is the root of $ h $, we have
\begin{equs}
\textbf{z}^{h, \tau}_{ts} =  \int_{A_{ts}^h \subseteq \R^{n}} \prod_{ (i,j) \in E_h} k(r_{i}, r_{j}) \prod_{\ell \in N_{h^{\star}}} dq^{i_{\ell}}_{r_{\ell}}
\end{equs}
Let $u \in [s,t]$, then one notices that

$$
A_{ts}^h = \bigcup_{C \in \scriptsize{\text{Adm}}(h) } \bigcap_{ (i,j) \in E_{h^{\star}} } \bigcap_{ (k, l) \in C } \bigcap_{ \lambda \in L_{R^{C}(h)}} \{ u > r_{\lambda} \} \cap \{ r_{k} > u > r_{l} \} \\ \cap \{ t > r_{i} > r_{j} > s \}
$$

All unions in the above expression are pairwise disjoint. Indeed, if we consider two different admissible cuts $ C $ and $C'$, then there exists $ \tilde{\lambda} $ belonging to $ L_{R^{C}(h)} $ (or $ L_{R^{C'}(h)} $) such that there exists $ \lambda \in L_{R^{C'}(h)} $ with $ r_{\lambda} > r_{\tilde{\lambda}} $. Then $  \{ u > r_{\lambda} \}  $ and $  \{r_{\lambda}  >u > r_{\tilde{\lambda}} \}  $ are disjoint.
Therefore,
\begin{equs}
\textbf{z}^{h}_{ts} =  \sum_{C \in \scriptsize{\text{Adm}}(h)  } \int_{A^{C,h}_{ts}(u) \subseteq \R^{n}} \prod_{ (i,j) \in E_h} k(r_{i}, r_{j}) \prod_{\ell \in N_{h^{\star}}} dq^{i_{\ell}}_{r_{\ell}}
\end{equs}
where $A^{C,h}_{ts}(u)$ is defined as follows: 
\begin{equs} 
A^{C,h}_{ts}(u) = \bigcap_{ (i,j) \in E_{h^{\star}} } \bigcap_{ (k, l) \in C } \bigcap_{ \lambda \in L_{R^{C}(h)}} \{ u > r_{\lambda} \} \cap \{ r_{k} > u > r_{l} \} \cap \{ t > r_{i} > r_{j} < s \}
\end{equs}
If $ h^{(1)} = R^{C}(h) $ and $ h^{(2)} = \tilde{P}^C(h) $, then one has:
\begin{equs}
A_{ts}^{C, h}(u) = A_{tu}^{h^{(1)}} \times A_{us}^{h^{(2)}}.
\end{equs}
Then, we can rewrite the above expression for $A_{ts}^{C, h}(u)$ as
\begin{equs}
A^{C,h}_{ts}(u) & = \left( \bigcap_{ (i,j) \in E_{h^{(1)}} } \{ t > r_{i} > r_{j} > u \} \right) \times \\ & \left( \bigcap_{ (i,j) \in E_{h^{(2)}} } \bigcap_{\lambda \in L_{R^{C}(h)}} \{ u>r_{\lambda} \} \cap \{ u > r_{i} > r_{j} > s \} \right) 
\end{equs} 
where the sets intersected on the right hand side are each considered as a subset of an appropriate subspace of $\mathbf{R}^{n}$. We conclude by using Fubini.
\end{proof}

\begin{example} If $h \in \CT$ is the ladder tree consisiting of a root and $n$ additional vertices decorated by $ i_1,...,i_{n} $, then for $t > r_1 > s$ we have:
\begin{equs}
\bar{\mathbf{z}}^{ h, \tau}_{ts}(r_1) =  \int_{t>r_{n}>... >r_{2}> r_1 }k( \tau, r_{n}) \prod_{j=1}^{n-1}k(r_{j+1}, r_{j})\, \dot{q}^{i_1}_{r_{1}} dq^{i_2}_{r_2}...dq^{i_n}_{r_{n}}
\end{equs}
In similar spirit, in the case of $h$ being a linear tree, we get a simpler version of Chen's relation closer to the one for classical Rough Paths. Proposition \ref{propchen} reduces to
\begin{equs}
\mathbf{z}_{ts}^{n,\tau} = \sum_{i=0}^{n}  \mathbf{z}_{tu}^{n-i, \tau} \star \mathbf{z}_{us}^{i, \cdot}
\end{equs}
where the superscript $n$ is used to denote the linear tree with $n$ vertices and set of node decorations of cardinality one. From this relation we can recover the classical Chen's relation for the signature of a path with bounded variation, since the convolution product reduces to the tensor product by choosing a trivial kernel. 
\end{example}
 
Our aim now is to capture the algebraic properties of these generalized signatures in order to abstract from them the definition of a Branched Rough Path of Volterra-type. With this in mind, we prove the following proposition, which describes the convolution product in terms that can be generalized in the rough setting. The idea is that we can then use an expression similar to the one below as a definition for the convolution product in the rough case. Given a tree $ h $, we first extend \eqref{def_convol} to a function $ (\tau_i)_{i\in N_{h^{\star}}} \mapsto f_s^{(\tau_i)_{i\in N_{h^{\star}}}} $. We have:
\begin{equs}
  \mathbf{z}_{ts}^{h,\tau} \star f_{s}
  := \int_{\R^{m}} \bar{\mathbf{z}}_{ts}^{h,\tau}((r_i)_{i \in N_{h^{\star}}}) f_s^{(r_i)_{i\in N_{h^{\star}}}}\prod_{i \in N_{h^{\star}}}  d r_i
\end{equs}
where $ m $ is equal to the cardinality of $ N_{h^{\star}} $. From this definition, one can make sense of the following expression:
\begin{equs}
 \mathbf{z}^{h^{(1)}, \tau}_{vu} \star (\mathbf{z}_{us}^{h^{(2)}} \star f_{s})
\end{equs}
where $ h^{(1)} $ and $ h^{(2)} $ are coming from  \eqref{Delta_12} with the identification of their nodes with those of $ h $. The inner convolution is happening on the nodes $ N_{(h^{(2)})^{\star}} $ and the outer convolution is happening on the nodes $ N_{(h^{(1)})^{\star}} $.

\begin{proposition} \label{associativity_smooth}

Let $h$ be a tree with $n$ decorated vertices  and let $ \mathcal{P}$ denote a partition of $ [s,t]$ with mesh $|\mathcal{P}|$. Let $f_s: \R^n \rightarrow \R^{e}$ be a smooth function whose upper arguments are indexed by $  N_{h^{\star}}$. That is, for every $s \in \mathbf{R}$ we have a mapping: $ (\tau_i)_{i\in N_{h^{\star}}} \mapsto f_s^{(\tau_i)_{i\in N_{h^{\star}}}} $. Then, the following identity holds:
\begin{equs} \label{chen_smooth}
 \mathbf{z}^{h, \tau}_{ts} \star f_{s} = \lim_{|\mathcal{P}| \rightarrow 0} \sum_{[u, v] \in \mathcal{P}} \left(  \mathbf{z}^{h, \tau}_{vu} \star f_{s}  +  \sum_{(h)'}  \mathbf{z}^{h^{(1)}, \tau}_{vu} \star (\mathbf{z}_{us}^{h^{(2)}} \star f_{s}) \right)
\end{equs}
\end{proposition}

\begin{proof}

Let $h$ be a tree with $n +1 $ vertices. Assuming that $ r_1,..., r_m $ refer to the integration variables associated with the nodes $ \rho_1,...,\rho_m $ adjacent to the root, one has
\begin{equs}
A^{h}_{ts} & = \bigcup_{[u, v] \in \mathcal{P}}  A^{h}_{ts}(u, v)
\end{equs}
where the disjoint sets $ ( A^{h}_{ts}(u,v))_{[u,v] \in \mathcal{P}} $ are defined by
\begin{equs}
A^{h}_{ts}(u,v) := \bigcap_{(i,j) \in E_h}  \lbrace
u < r_1 \vee ... \vee r_m < v  \rbrace \cap \lbrace t > r_i > r_j > s \rbrace.
\end{equs}
Then, given $ (r_i)_{i \in N_{h}} \in A_{ts}^h(u,v) $, one has either $ u < r_i < v $ for all $ i \in N_h $, or there exists a unique $ C \in \text{Adm}(h)' $ with $ C  $ non empty such that $ u > r_{\lambda} $ for all $ \lambda \in N_{\tilde{P}^{C}(h)} $. As a result, 
\begin{equs}
A^{h}_{ts}(u, v) = \bigcup_{C \in Adm(h)'} A^{h, C}_{ts}(u,v)
\end{equs}
where
\begin{equs}
A^{h, C}_{ts}(u, v) & = \bigcap_{ (i,j) \in E_{h} }  \bigcap_{ (k, l) \in C } \bigcap_{\lambda \in N_{\tilde{P}^{C}(h)}}  \{ u < r_{1} \vee  ... \vee  r_{m} < v\} \cap \{ u > r_{\lambda} \} \\ & \cap \{ r_{k} > u > r_{l} \} \cap \{ t > r_{i} > r_{j} > s \}.
\end{equs}
 Therefore,
\begin{equs}
\mathbf{z}^{h}_{ts} = \sum_{[u, v] \in \mathcal{P}}  \sum_{C \in \scriptsize{\text{Adm}}(h)'}  \int_{A^{h, C}_{ts}(u, v) \subseteq \R^{n}} \prod_{ (i,j) \in h} k(r_{i}, r_{j})\prod_{\ell \in N_{h^{\star}}} dq^{i_{\ell}}_{r_{\ell}}
\end{equs}

Thus, since the above decomposition of $A^{h}_{ts}$ is true for any partition $\mathcal{P}$ of $[s,t]$, we can split each set $A^{h, C}_{ts}(u, v)$ into a Cartesian product and use Fubini's theorem to obtain the equality
\begin{equs} \label{chen_smooth}
 \mathbf{z}^{h, \tau}_{ts} \star f_{s} = \lim_{|\mathcal{P}| \rightarrow 0} \sum_{[u, v] \in \mathcal{P}}  \left(  \mathbf{z}^{h, \tau}_{vu} \star f_{s}  + \sum_{(h)'} \mathbf{z}^{h^{(1)}, \tau}_{vu} \star (\mathbf{z}_{us}^{h^{(2)}} \star f_{s}) \right)
\end{equs}
which is precisely the desired identity. This concludes the proof. 
\end{proof}

We now wish to extend this product operation in the rough setting. In order to do this, we will need a version of the Sewing Lemma different from the one used in classical Rough Path Theory. We will call it the Volterra Sewing Lemma, following the terminology used in \cite{HT}.
We first introduce the spaces needed for the formulation of this lemma. We recall \cite[Def. 2.3]{HT2}.
\begin{definition}\label{Volterra space}
Let $ (\alpha, \gamma) \in [0,1]^2 $ with $ \rho = \alpha - \gamma > 0 $, we define the following space $ \mathcal{V}^{(\alpha,\gamma)}(\Delta_2;\R) $ of functions $ z $ such that $ z_{0}^{\tau} = z_0 \in \R $ and equipped with the norm:
\begin{equs}
\Vert z \Vert_{(\alpha,\gamma)} = |z_0| + \Vert z \Vert_{(\alpha,\gamma),1} + \Vert z \Vert_{(\alpha,\gamma),1,2}
\end{equs}
where 
\begin{equs}
\Vert z \Vert_{(\alpha,\gamma),1} & = \sup_{(s,t,\tau) \in \Delta_3} \frac{|z_{st}^{\tau}|}{|\tau-t|^{-\gamma} |t-s|^{\alpha} \wedge |\tau-s|^{\varrho}} \\
\Vert z \Vert_{(\alpha,\gamma),1,2} &  = \sup_{\substack{(s,t,\tau,\tau') \in \Delta_4 \\ \eta \in [0,1], \zeta \in [0,\varrho)}} \frac{|z_{st}^{\tau \tau'}|}{
|\tau - \tau'|^{\eta} |\tau'-t|^{-\eta+ \zeta}  ( |\tau'-t|^{-\gamma-\zeta} |t-s|^{\alpha} \wedge |\tau' - s|^{\varrho-\zeta} )  } 
\end{equs}
where the increments $ z^{\tau}_{st} $ and $ z^{\tau \tau'}_{st} $ are given by
\begin{equs}
z^{\tau}_{st} = z^{\tau}_t - z^{\tau}_s, \quad  z^{\tau \tau'}_{st} = z^{\tau}_{t}-z^{\tau'}_{t}-z_{s}^{\tau}+z_{s}^{\tau'}.
\end{equs}
Given $\alpha$ and $\gamma$ in $[0,1]$ we also define the norm for elements $\omega$ of the $k$-fold tensor product $\bigotimes_{i=1}^{k} \mathcal{V}^{(\alpha, \gamma)}$ to be the projective tensor norm
\begin{equs}
 ||\omega||_{(\alpha, \gamma)} = \inf \{ \ \sum ||z^{1}||\cdot ... \cdot||z^{k}|| \ : \ \omega = \sum z^{1} \otimes ... \otimes z^{k} \}
\end{equs}
where the infimum is taken over all decompositions of $\omega$. This defines a \emph{crossnorm}, in the sense that 
\begin{equs}
|| z^{1} \otimes ... \otimes z^{k}||_{(\alpha, \gamma)}  = ||z^{1}||_{(\alpha, \gamma)} \cdot ... \cdot||z^{k}||_{(\alpha, \gamma)} 
\end{equs}
We denote the completion of the algebraic tensor product with respect to this norm by $\hat{\bigotimes}_{i=1}^{k} \mathcal{V}^{(\alpha, \gamma)}$.

\end{definition}

\begin{definition}
Given any Banach space $V$ we define the operator $\delta : \mathcal{C}(\Delta_{2}; V) \rightarrow \mathcal{C}(\Delta_{3}; V)$ as follows:

\begin{equs}
(\delta f)_{tus} = f_{ts} - f_{tu} - f_{us}.
\end{equs}

We will usually write this as 

\begin{equs}
\delta_{u}f_{ts} = f_{ts} - f_{tu} - f_{us}.
\end{equs}

\end{definition}

We now introduce the space of abstract integrands that is needed to formulate the Volterra Sewing Lemma, as given in \cite[Def. 2.9]{HT2}.

\begin{definition}\label{new abstract integrnds space}

Let $\alpha\in\left(0,1\right)$, $\gamma\in (0,1)$ with $\alpha-\gamma>0$, and $ \eta $ with $ \gamma + \eta \in (0,1) $. We suppose given three coefficients $(\beta,\kappa,\theta)$, with $(\kappa+\theta)\in (0,1)$ and $\beta\in(1,\infty)$. Denote by $\mathcal{V}^{(\alpha,\gamma, \eta)(\beta,\kappa,\theta)}(\Delta_{4};\R^{d})$, the space of all functions of the form $\Delta_{4}\ni(v,s,t,\tau) \mapsto (\varXi^{\tau}_{v})_{ts}\in \R^{d}$
such that the following norm is finite:
\begin{equation}\label{301}
\left\|\varXi\right\|_{\mathcal{V}^{\left(\alpha,\gamma, \eta \right)\left(\beta,\kappa,\theta\right)}}=\left\|\varXi\right\|_{\left(\alpha,\gamma, \eta\right)}+\left\|\delta\varXi\right\|_{\left(\beta,\kappa,\theta\right)}
\end{equation}
In equation \eqref{301}, the operator $\delta$ is the one introduced above and the term $\|\delta\varXi\|_{\left(\beta,\kappa,\theta\right)}$ takes the double singularity into account. Namely, we define
\[
\begin{aligned}
	\left\| \varXi\right\|_{\left(\alpha,\gamma,\eta\right)} &=\left\|\varXi\right\|_{\left(\alpha,\gamma,\eta\right),1}+\left\|\varXi\right\|_{\left(\alpha,\gamma,\eta\right),1,2}, \\
\left\|\delta\varXi\right\|_{\left(\beta,\kappa,\theta\right)}&=\left\|\delta\varXi\right\|_{\left(\beta,\kappa,\theta\right),1}+\left\|\delta\varXi\right\|_{\left(\beta,\kappa,\theta\right),1,2},
\end{aligned}
\]
where
\begin{flalign}\label{1 delta norm}
	\left\|\varXi\right\|_{\left(\alpha,\gamma,\eta\right),1}:=\sup_{\left(v,s,t,\tau\right)\in\Delta_{4}}\frac{\left|(\varXi^{\tau}_{v})_{ts}\right|}{\left[\left|\tau-t\right|^{-\gamma}\left|t-s\right|^{\alpha}\left|s-v\right|^{-\eta}\right]\wedge \left|\tau-v\right|^{\alpha-\gamma-\eta}},\\
\left\|\delta\varXi\right\|_{\left(\beta,\kappa,\theta\right),1}:=\sup_{\left(v,s,m,t,\tau\right)\in\Delta_{5}}\frac{\left|\delta_{m}(\varXi^{\tau}_{v})_{ts}\right|}{\left[\left|\tau-t\right|^{-\kappa}\left|t-s\right|^{\beta}\left|s-v\right|^{-\theta}\right]\wedge \left|\tau-v\right|^{\beta-\kappa-\theta}},
\end{flalign}
and 
\begin{flalign}\label{1,2 delta norm}
	\|\varXi\|_{\left(\alpha,\gamma,\eta\right),1,2}&:=\sup_{\substack{\left(v,s,t,\tau^\prime,\tau\right)\in\Delta_{5} \\ \bar{\eta}\in[0,1],\zeta\in[0,\alpha-\gamma-\eta)}}\frac{\left|(\varXi^{\tau\tau^\prime}_{v})_{ts}\right|}{f(v,s,t,\tau^\prime,\tau,\alpha,\gamma,\eta)}, \\
\|\delta\varXi\|_{\left(\beta,\kappa,\theta\right),1,2}&:=\sup_{\substack{\left(v,s,m,t,\tau^\prime,\tau\right)\in\Delta_{6} \\ \bar{\eta}\in[0,1],\zeta\in[0,\beta-\kappa-\theta)}}\frac{\left|\delta_{m}(\varXi^{\tau\tau^\prime}_{v})_{ts}\right|}{f(v,s,t,\tau^\prime,\tau,\beta,\kappa,\theta)},
\end{flalign}
where the function $f$ is given by 
\begin{equs}\label{304}
f(v,s,t,\tau^\prime,\tau,\beta,\kappa,\theta) & =\left|\tau-\tau^{\prime}\right|^{\bar{\eta}}\left|\tau^{\prime}-t\right|^{-\bar{\eta}+\zeta} \times \\ & \times \left(\left[\left|\tau^{\prime}-t\right|^{-\kappa-\zeta}\left|t-s\right|^{\beta}\left|s-v\right|^{-\theta}\right]\wedge \left|\tau^{\prime}-v\right|^{\beta-\kappa-\theta-\zeta}\right).
\end{equs}
Notice that we will use $\mathcal{V}^{\left(\alpha,\gamma, \eta\right)\left(\beta,\kappa,\theta\right)}$
as a space of abstract Volterra integrands with a double singularity. 

\end{definition}

We are now ready to state the Volterra Sewing Lemma with two singularities (see \cite[Lem. 3.2]{HT2}).

\begin{lemma}[Volterra Sewing Lemma]\label{Volterra Sewing Lemma}

We consider the following  exponents $(\alpha,\gamma, \eta)$, and $(\beta,\kappa,\theta)$, with $\beta\in (1,\infty)$, $(\kappa+\theta)\in (0,1)$, $(\gamma+\eta)\in (0,1)$, $\alpha\in\left(0,1\right)$ and $\gamma\in(0,1)$ such that $\alpha-\gamma >0$.  Let $\mathcal{V}^{(\alpha,\gamma,\eta)(\beta,\kappa,\theta)}$  and $\mathcal{V}^{\left(\alpha,\gamma\right)}$ be the spaces given in Definition \ref{new abstract integrnds space} and Definition \ref{Volterra space} respectively. Then, there exists a linear continuous map $\mathcal{I}:\mathcal{V}^{(\alpha,\gamma, \eta)(\beta,\kappa,\theta)}\left(\Delta_{4};\R^{d}\right)\rightarrow\mathcal{V}^{\left(\alpha,\gamma\right)}\left(\Delta_{3};\R^{d}\right)$
such that the following hold true.
\begin{itemize}
\item[(i)]
The quantity
\begin{equs}
\mathcal{I}(\varXi^{\tau}_{v})_{ts}:=\lim_{|\mathcal{P}|\rightarrow 0} \sum_{[u,w]\in\mathcal{P}} (\varXi^{\tau}_{v})_{wu}
\end{equs}
exists for all $(v,s,t,\tau)\in \Delta_{4}$, where $\mathcal{P}$  is a generic partition of $[s,t]$  and $|\mathcal{P}|$  denotes the mesh size of the partition.  Furthermore, we define $\mathcal{I}(\varXi^{\tau}_{v})_{t}:=\mathcal{I}(\varXi^{\tau}_{v})_{t0}$, and have $\mathcal{I}(\varXi^{\tau}_{v})_{ts}=\mathcal{I}(\varXi^{\tau}_{v})_{t0}-\mathcal{I}(\varXi^{\tau}_{v})_{s0}$.
\item[(ii)] For all $(v,s,t,\tau)\in \Delta_{4}$ we have 
\begin{equs}
|\mathcal{I}\left(\varXi^{\tau}_{v}\right)_{ts}-(\varXi^{\tau}_{v})_{ts}|\lesssim & \|\delta\varXi\|_{(\beta,\kappa,\theta),1}\left(\left[\left|\tau-t\right|^{-\kappa}\left|t-s\right|^{\beta}\left|s-v\right|^{-\theta}\right]\wedge\left|\tau-v\right|^{\beta-\kappa-\theta}\right),
\end{equs} 
while for $(v,s,t,\tau^{\prime},\tau)\in \Delta_{5}$ we get
\begin{equs}
\left|\mathcal{I}(\varXi^{\tau\tau^\prime}_{v})_{ts}-(\varXi^{\tau\tau^\prime}_{v})_{ts}\right|\lesssim  \|\delta\varXi\|_{\left(\beta,\kappa,\theta\right),1,2}f(v,s,t,\tau^{\prime},\tau),
\end{equs}
where $f$ is the function given by \eqref{304}.
\end{itemize}
\end{lemma}

We also have the following theorem, which was proven in \cite{HT}.

\begin{theorem}\label{thm:Regularity of Volterra path} 
Let $x\in\mathcal{C}^{\alpha}$ and $k$ be a Volterra kernel of order $-\gamma$ satisfying \ref{condition_k}, such that $\rho=\alpha-\gamma>0$. We define an element $\varXi_{ts}^{\tau}=k(\tau,s)x_{ts}$. Then the following holds true:

\begin{enumerate}[wide, labelwidth=!, labelindent=0pt, label=(\roman*)]

\item
There exists some coefficients $\beta > 1$ and $\kappa >0$ with $\beta-\kappa=\alpha-\gamma$ such that
$\varXi \in \mathcal{V}^{(\alpha,\gamma)(\beta,\kappa)}$, where $\mathcal{V}^{(\alpha,\gamma)(\beta,\kappa)}$  is given in Definition \ref{new abstract integrnds space} where there is no dependency on $ v $ that is $ \Xi_v = \Xi $  Therefore, we can drop the extra parameters $ \eta $ and $ \theta $. It follows that the element $\mathcal{I}\left(\varXi^{\tau}\right)$  obtained in  Lemma \ref{Volterra Sewing Lemma}, the version with no dependency on $ v $, is well defined as an element of $\mathcal{V}^{(\alpha,\gamma)}$ and we set $z_{ts}^{\tau}\equiv\mathcal{I}\left(\varXi^{\tau}\right)_{ts}=\int_{s}^{t}k(\tau,r)dx_{r}$. 

\item
For $(s,t,\tau)\in \Delta_{3}$ $z$ satisfies the bound  
\begin{equation*}
\left|z_{ts}^{\tau}-k(\tau,s)x_{ts}\right|\lesssim  \left[\left|\tau-t\right|^{-\gamma}\left|t-s\right|^{\alpha}\right]\wedge \left|\tau-s\right|^\rho,  
\end{equation*}
and in particular it holds that $|z|_{(\alpha,\gamma),1}<\infty$.

\item
For any $\eta \in [0,1]$  and  any $(s,t,q,p)\in \Delta_{4}$  we have 
\begin{equation*}
\left|z_{ts}^{pq}\right|\lesssim \left|p-q\right|^{\eta}\left|q-t\right|^{-\eta+\zeta}\left(\left[\left|q-t\right|^{-\gamma-\zeta}\left|t-s\right|^{\alpha}\right]\wedge \left|q-s\right|^{\rho-\zeta}\right),
\end{equation*}
where $z_{ts}^{pq}=z^{p}_{t}-z^{q}_{t}-z_{s}^{p}+z_{s}^{q}$. In particular it holds that $|z|_{(\alpha,\gamma),1,2}<\infty$.

\end{enumerate}
\end{theorem}

A next step now is to find the right spaces on which to define the $\star$ operation so that we may abstract from the notion of a collection of iterated Volterra integrals that of a Volterra-type Rough Path. However, even with the appropriate abstract function spaces at hand, the definition of the product will still pose a challenge: defining the resulting expression in terms of the kernels involved as one big iterated integral is conceptually straightforward when our functions are indeed expressible as such integrals; but defining the product on an abstract function space in a way that generalizes these cases of interest is a delicate matter, as we shall see.
 We begin by giving the definition of the convolution product for increments. We recall \cite[Thm. 26]{HT}:
 
 \begin{proposition} \label{convolution_init}
Let $ z  \in \mathcal{V}^{(\alpha,\gamma)}(\Delta_2, \R) $ and $ y \in \mathcal{V}^{(\alpha,\gamma)}(\Delta_2,\R^{e}) $. We set $\varrho = \alpha - \gamma$ and assume that $ \varrho > 0 $. Then, we define the convolution product between $ z $ and $y  $ as:
\begin{equs}
z_{tu}^{\tau} \star y_{us}^{\cdot} := \lim_{|\mathcal{P}| \rightarrow 0 } \sum_{[u',v'] \in \mathcal{P}} z^{\tau}_{v'u'} y^{u'}_{us}
\end{equs} 
where $ (s,t,u,\tau) \in \Delta_4 $. Moreover, one gets the following inequalities:
\begin{equs}
| z^{\tau}_{tu} \star y^{\cdot}_{us} | & \lesssim 
\Vert z \Vert_{(\alpha,\gamma),1} \Vert y \Vert_{(\alpha.\gamma),1,2} \left( [  |\tau - t |^{-\gamma} | t-s |^{2 \varrho + \gamma}   ] \wedge |\tau - s|^{2 \varrho} \right) \\
| z^{\tau' \tau}_{tu} \star y^{\cdot}_{us} | & \lesssim 
\Vert z \Vert_{(\alpha,\gamma),1,2} \Vert y \Vert_{(\alpha.\gamma),1,2} | \tau' -\tau|^{\eta} | \tau - t |^{-\eta + \zeta} \\ &  \left( [  |\tau - t |^{-\gamma} | t-s |^{2 \varrho + \gamma}   ] \wedge |\tau - s|^{2 \varrho-\zeta} \right) 
\end{equs}
where $ \eta \in [0,1] $, $\zeta \in [0,2 \varrho]$ and $ (s,u,t,\tau,\tau') \in \Delta_5 $.
 \end{proposition}

We will need to extend the convolution product $\star $ to functions $y$ with more variables. To this end, we introduce the spaces $\mathcal{V}^{(\alpha,\gamma)}_{n}$ of functions $y$ defined on the hypercube $Q_{n+1}$. The norms of these spaces serve to control the increments of each variable.

\begin{definition}
We define the space $\mathcal{V}^{(\alpha,\gamma)}_{n}(Q_{n+1},\R^e)$ to be the space of functions $y: Q_{n+1} \rightarrow \R^{e}$ such that $ y_{0}^{r_1,...,r_n} = y_0 $ and
\begin{equs}
||| y |||_{(\alpha, \gamma), n} = \sum_{k=1}^{n} ||y||_{(\alpha, \gamma), n, k} < \infty
\end{equs}
where, for every $k \leq n$, we define the norms
\begin{equs}
||y||_{(\alpha, \gamma), n, k} = \underset { \underset{\eta \in [0,1], \zeta \in [0,\alpha-\gamma)}{(s,t, r_{1}, ..., r, u ,..., r_{n}) \in Q_{n+2}} }{\sup} \ \ \frac{|y_{ts}^{r_{1}, ..., r, ..., r_{n}} - y_{ts}^{r_{1}, ..., u, ..., r_{n}}|}{h_{\eta,\zeta}(s,t,r_1,..,r,u,...,r_n)}
\end{equs}
where 
\begin{equs}
& h_{\eta,\zeta}(s,t,r_1,..,r,u,...,r_n)  = |r-u|^{\eta} |\min(r_1,..,r,u,...,r_n) - t|^{-\eta + \zeta} \\ & \times \left(  | \min(r_1,..,r,u,...,r_n) - t |^{-\gamma-\zeta} |t-s|^{\alpha}
\wedge | \min(r_1,..,r,u,...,r_n) - s |^{\alpha-\gamma-\zeta}
\right).
\end{equs}
Here, the values $r$ and $u$ appearing as superscripts in the expression for the numerator are the values of the $k$-th variable of $y$. 
\end{definition}

\begin{remark}
Note that the constituent norms introduced above are similar to the norm $||\cdot||_{(\alpha, \gamma), 1, 2}$
\end{remark}

\vspace{5pt}

We now introduce the family of spaces and the corresponding convolution product. They will be defined recursively.

\vspace{5pt}

\begin{definition}\label{spaces}

Let $\alpha, \gamma \in (0,1)$ with $ \rho = \alpha - \gamma > 0$ and $n \leq 1/\alpha$ be fixed. We define a tree-indexed family of spaces $\mathcal{V}^{h, \alpha, \gamma}$ and a family of products
$$
(\textbf{z}, y) \in \mathcal{V}^{h, \alpha, \gamma} \times \mathcal{V}^{\alpha, \gamma}_{|h|} \ \ \mapsto \ \ \textbf{z} \star y
$$
recursively, by first setting for the tree with one vertex:
$$
\mathcal{V} ^{\bm{\bm{\bm{\bm{\cdot}}}}, \alpha, \gamma} := \mathcal{V} ^{(\alpha, \gamma)}
$$
and defining the product $\star$ as
\begin{equs}
z_{tu}^{\tau} \star y_{us}^{\cdot} := \lim_{|\mathcal{P}| \rightarrow 0 } \sum_{[u',v'] \in \mathcal{P}} z^{\tau}_{v'u'} y^{u'}_{us}
\end{equs} 
where $ (s,t,u,\tau) \in \Delta_4 $ and convergence is guaranteed by Proposition \ref{convolution_init} in this case. 
We then define for any $h \in \mathcal{T}$ with $|h| \leq n \leq 1/\alpha$ the space $\mathcal{V}^{h, \alpha, \gamma} \subset \mathcal{V}^{(|h|\rho + \gamma, \gamma)}$  consisting of all $z^{h} \in \mathcal{V}^{(|h|\rho + \gamma, \gamma)}$ such that 
$$
\delta_{u} \mathbf{z}^{h}_{ts} = \sum_{(h)'} \mathbf{z}^{h^{(1)}}_{tu} \star  \mathbf{z}^{h^{(2)}}_{us}
$$
for some functions
$$
\mathbf{z}^{h^{(1)}} \in \mathcal{V}^{h^{(1)}, \alpha, \gamma}, \ \ \ \ \mathbf{z}^{h^{(2)}} \in \bigotimes_{i=1}^{m} \mathcal{V}^{h^{(2)}_{i}, \alpha, \gamma} 
$$
with $h^{(2)} = h^{(2)}_{1} \cdot ... \cdot h^{(2)}_{m}$ and such that, for every $y \in \mathcal{V}^{(\alpha,\gamma)}_{n}$, one has
\begin{equs} \label{chen2}
\delta_{u} \textbf{z}^{h,\tau}_{ts} \star y_s = \sum_{(h)'} \left( \textbf{z}^{h^{(1)},\tau}_{tu} \star \left(  \textbf{z}^{h^{(2)}}_{us} \star  y_s \right) \right)
\end{equs}
We define the product $(\textbf{z}, y) \in \mathcal{V}^{h, \alpha, \gamma} \times \mathcal{V}^{\alpha, \gamma}_{|h|} \ \ \mapsto \ \ \textbf{z} \star y$ as follows:
\begin{equs}
\mathbf{z}^{h, \tau}_{ts} \star y_{s} := \lim_{|\mathcal{P}| \rightarrow 0} \sum_{[u,v] \in \mathcal{P}} \mathbf{z}^{h, \tau}_{vu} \otimes y^{u, ..., u}_{s} +  \sum_{(h)'}  \mathbf{z}^{h^{(1)}, \tau}_{vu} \star ( \mathbf{z}^{h^{(2)}, \tau}_{us} \star y_{s} )
\end{equs}
where we use the inductive hypothesis that the spaces $\star$ product has already been defined for $h$ of order less than or equal to $|h| - 1$. Convergence of the expression on the right is proven below, in Theorem \ref{Convolution Product}.
\end{definition}

\begin{remark}
The spaces $\mathcal{V}^{h, \alpha, \gamma}$ are metric subspaces of the corresponding $\mathcal{V}^{(|h|\rho + \gamma, \gamma)}$ space. However, they do not appear to be linear subspaces.
\end{remark}

\begin{remark} The identity \eqref{chen2} can be viewed as an extension of the Chen's relation where the terms $\textbf{z}_{ts}^{h}$ are considered as operators. This property appears in the smooth case, as can be seen by the proof of Proposition~\ref{associativity_smooth}.
\end{remark}

 We now proceed to prove convergence of the defining expression for $\star$ by using the Volterra Sewing Lemma \ref{Volterra Sewing Lemma} in an essential way.

\begin{theorem}\label{Convolution Product}

Let $h$ be a tree in $ \CT_n $ and let $\mathbf{z}^{h} \in \mathcal{V}^{h, \alpha, \gamma}$ and  $y  \in \mathcal{V}^{(\alpha,\gamma)}_{|h|}$ with $a, \gamma \in (0,1)$ and $\rho = a - \gamma >0$. Then, for all fixed $(s, t, \tau) \in \Delta_{3}$ the expression
\begin{equs}
\mathbf{z}^{h, \tau}_{ts} \star y_{s} := \lim_{|\mathcal{P}| \rightarrow 0} \sum_{[u,v] \in \mathcal{P}} \mathbf{z}^{h, \tau}_{vu} \otimes y^{u, ..., u}_{s} +  \sum_{(h)'}  \mathbf{z}^{h^{(1)}, \tau}_{vu} \star ( \mathbf{z}^{h^{(2)}, \tau}_{us} \star y_{s} )
\end{equs}
yields a well-defined Volterra-Young integral. It follows that $\star$ is a well defined operation between the three-parameter Volterra function $\mathbf{z}^{h}$ and an $|h|$-parameter path $y$, linear in its second argument. Moreover, we have the following two inequalities: 
\begin{equs}\label{ineq1}
& |\mathbf{z}_{ts}^{h, \tau} \star \left( y_{s}^{\cdot, ..., \cdot} \ -  y_{s}^{s, ..., s} \right)| \lesssim \left( \sum_{ ((h))} \prod_{i=1}^k ||\mathbf{z}^{h^{(i)}}||_{(|h^{(i)}|\rho + \gamma, \gamma)} \right) \\ & \times |||y|||_{ (\alpha, \gamma), |h|} \times  
 \left( |\tau - t|^{-\gamma} |t-s|^{|h|\rho + \gamma} \wedge |\tau - s|^{|h|\rho} \right)
\end{equs}
\begin{equs}\label{ineq2}
& |\mathbf{z}_{ts}^{h, \tau' \tau} \star \left( y_{s}^{\cdot, ..., \cdot} - y_{s}^{s, ..., s} \right) |  \lesssim \left( \sum_{ ((h))} \prod_{i=1}^k ||\mathbf{z}^{h^{(i)}}||_{(|h^{(i)}|\rho + \gamma, \gamma)} \right) \\ & \times |||y|||_{ (\alpha, \gamma), |h|}  | \tau' -\tau|^{\eta} | \tau - t |^{-\eta + \zeta}  
 \left( |\tau - t|^{-\gamma} |t-s|^{|h|\rho + \gamma} \wedge |\tau - s|^{|h|\rho - \zeta} \right)
 \end{equs}
In these bounds, our notation of summing over $((h))$ means that the summation is over all elements $ h^{(1)} \otimes ... \otimes h^{(k)} $ that appear in the expansion of
\begin{equs}
\bar{\Delta} h = \sum_{k}  \Delta^k h , \quad \Delta^{i+1} = \left( \one \otimes \Delta^i \right) \tilde{\Delta}, \quad \Delta^{0} = \id.
\end{equs}
 It can also be understood as iterated admissible cuts.




\end{theorem}

\begin{proof}

We induct on the number of vertices of the tree $h$. The base case of $h \in \mathcal{T_{1}}$ is covered by Proposition~\ref{convolution_init}. We now assume that the statement holds for all $h \in \mathcal{T}_{p-1}$ and proceed to prove it is true for all $h \in \mathcal{T}_{p}$.

\vspace{10pt}

\emph{Step 1:} Let us denote by $\mathcal{I}_{\mathcal{P}}$ the approximation of the right hand side above, that is 
\begin{equation} \label{Riemann sums2}
\mathcal{I}_{\mathcal{P}}:= \sum_{\left[u,v \right]\in\mathcal{P}} (\varXi_s)_{vu}^{\tau}:=\sum_{\left[u,v\right]\in\mathcal{P}} \mathbf{z}_{vu}^{h,\tau}\otimes y^{u}_{s}+(\delta_{u} \mathbf{z}_{vs}^{h,\tau}) \star y_{s}^{\cdot, ..., \cdot}
\end{equation}
 where $ y^{u}_s $ is a short hand notation for $ y^{u,...,u}_s $.
 Our goal is to apply Lemma~\ref{Volterra Sewing Lemma} to the increment $\varXi$. We  must therefore check the regularity of the integrand under the action of $\delta$. To this aim, two simple computations using that $\delta_{r} \mathbf{z}_{vu}^{h,\tau}= \sum_{(h)'} \mathbf{z}^{h^{(1)}}_{vr} \star  \mathbf{z}^{h^{(2)}}_{ru}$ reveal 
\begin{align}\label{delta rule}
\delta_{r} (\mathbf{z}_{vu}^{h,\tau}\otimes y^{u}_{s}) = -\mathbf{z}_{vr}^{h,\tau}\otimes (y^{r}_{s}-y^{u}_{s}) + \sum_{(h)'} \mathbf{z}^{h^{(1)}}_{vr} \star  \mathbf{z}^{h^{(2)}}_{ru} \otimes y^{u, ..., u}_{s} ,
\\\label{delta two rule}
\delta_{r}((\delta_{u} \mathbf{z}_{vs}^{h,\tau}) \star y_{s}) = - \sum_{(h)'} \mathbf{z}^{h^{(1)}}_{vr}  \star \left( \mathbf{z}^{h^{(2)}}_{ru} \star y_{s} \right), 
\end{align}
To prove \eqref{delta two rule} we first expand the left hand side of the identity using the definition of the operator $\delta$ and get
\begin{equs}
\delta_{r}((\delta_{u} \mathbf{z}_{vs}^{h,\tau}) \star y^{\cdot}_{s}) &  = \sum_{(h)'}  (\mathbf{z}^{h^{(1)}}_{vu} \star  (\mathbf{z}^{h^{(2)}}_{us} \star y_{s})) -  (\mathbf{z}^{h^{(1)}}_{vr} \star  (\mathbf{z}^{h^{(2)}}_{rs} \star y_{s}))
\\ & - (\mathbf{z}^{h^{(1)}}_{ru} \star  (\mathbf{z}^{h^{(2)}}_{us} \star y_{s}))
\end{equs}
We know by Chen's relation, that one has
\begin{equs}
\mathbf{z}^{h^{(1)}}_{vu} = \sum_{(h^{(1)})} \mathbf{z}^{h^{(11)}}_{vr} \star \mathbf{z}^{h^{(12)}}_{ru} 
\end{equs}
Therefore, 
\begin{equs}
\mathbf{z}^{h^{(1)}}_{vu} \star  (\mathbf{z}^{h^{(2)}}_{us} \star y_{s}) - \mathbf{z}^{h^{(1)}}_{ru} \star  (\mathbf{z}^{h^{(2)}}_{us} \star y_s) = \sum_{(h^{(1)}),h^{(11)} \neq \one} (\mathbf{z}^{h^{(11)}}_{vr} \star \mathbf{z}^{h^{(12)}}_{ru}) \star ( \mathbf{z}_{us}^{h^{(2)}} \star y_s)
\end{equs}
Similarly, one has
\begin{equs}
\mathbf{z}^{h^{(1)}}_{vr} \star  (\mathbf{z}^{h^{(2)}}_{rs} \star y_{s}) = \sum_{(h^{(2)})} \mathbf{z}^{h^{(1)}}_{vr} \star (( \mathbf{z}^{h^{(21)}}_{ru} \star \mathbf{z}^{h^{(22)}}_{us}) \star y_{s})
\end{equs}
 One can see that $ h^{(1)} \neq \one $ and $ h^{(2)} \neq \one $ which is not the case for the range of $ h^{(22)}$ in the sum above. Now, after making these substitutions in the original sum, by coassociativity of the coproduct and \eqref{chen2} we see that most of the terms possess a counterpart with identical forest-indices and opposite sign. These terms cancel out and we eventually obtain \eqref{delta two rule}.

\vspace{10pt}

\emph{Step 2:} Let us now analyse the regularity of the terms in \eqref{delta rule}-\eqref{delta two rule}, starting with the right hand side of \eqref{delta rule}. 
 Namely, we recall our assumption that $\mathbf{z}^h\in\mathcal{V}^{(|h|\rho+\gamma,\gamma)}$. We have also assumed that $|||y|||_{(\alpha,\gamma), |h|} < \infty$.  We can therefore multiply the bounds afforded by the finite norms of $\mathbf{z}^{h}$ and $y$ and use the fact that $u \leq r \leq v$ to obtain the following bound:
\begin{equs}\label{reg z2 con y}
| \mathbf{z}_{vr}^{h,\tau} \otimes (y^{r}_{s}-y^{u}_{s}) | & \lesssim |||y|||_{ (\alpha ,\gamma ), |h| } \|\mathbf{z}^{h}\|_{(|h|\rho+\gamma,\gamma),1}  \\ & \times |u-s|^{-\eta}|\tau-v|^{-\gamma}|v-u|^{|h|\rho+\gamma+\eta}, 
\end{equs}

 We then choose $\eta \in [0,1]$  such that  $|h|\rho+\gamma+\eta>1$, which is always possible since $\rho>0$.

\vspace{10pt}

\emph{Step 3:} In order to treat the remaining terms in (\ref{delta rule}) and (\ref{delta two rule}), observe that formula (\ref{Riemann sums2}) trivially yields (recall again that $y_s^{u}$ has to be considered as a constant  in the lower variable)
\begin{equation*}
\mathbf{z}^{h,\tau}_{ts}\star  y_{s}^{u}= \mathbf{z}^{h,\tau}_{ts}\otimes y_{s}^{u}.
\end{equation*}
Therefore, we can gather our two remaining terms into 
\begin{equs}
\sum_{(h)} \mathbf{z}^{h^{(1)}}_{vr} \star  \mathbf{z}^{h^{(2)}}_{ru} \otimes y^{u, ..., u}_{s} - \sum_{(h)} \mathbf{z}^{h^{(1)}}_{vr} \star  \mathbf{z}^{h^{(2)}}_{ru} \star y_{s}
= \sum_{(h)} \mathbf{z}^{h^{(1)}}_{vr} \star  \mathbf{z}^{h^{(2)}}_{ru} \star (y^{u, ..., u}_{s} - y_{s})
\end{equs}

We introduce the following notation for $ y  $ and $ h_1 $ a subforest of $h$:
$\hat y_s^{h_1, r} $ which means that $y_s$ is evaluated at $r$ for all the variables corresponding to the nodes of $ h_1 $. The other variables are free. If the nodes of $ h_1 $ receive the evaluation $r_1,...,r_{|h_1|}$, we use the notation $ \hat y_s^{h_1,r_1,...,r_{|h_1|}} $.

Now, using the inequality in the statement of this theorem, which is assumed to hold for all $h \in \mathcal{F}_{p-1}$ by our inductive hypothesis (see Proposition~\ref{convolution_init} for the base case) and using that $|||y|||_{(\alpha, \gamma), |h^{(1)}|} < \infty$ we get
\begin{equs}
& |\mathbf{z}^{h^{(1)}, \tau}_{vr} \star  \mathbf{z}^{h^{(2)} \cdot, ..., \cdot}_{ru} \star (y^{h,u, ..., u}_{s} - y^{\cdot, ..., \cdot}_{s})|  \lesssim  \left( \sum_{ ((h^{(1)}))} \prod_{i=1}^k ||\mathbf{z}^{h^{(1i)}}||_{(|h^{(1i)}|\rho + \gamma, \gamma)} \right)
\\ 
& \times |||w|||_{(\alpha, \gamma), |h^{(1)}|} \cdot \left( |\tau - v|^{-\gamma} |v-r|^{|h|\rho + \gamma} \wedge |\tau - r|^{|h|\rho} \right)
\end{equs}
where $w_{ru}^{r_{1}, ..., r_{|h^{(1)}|}} = \mathbf{z}^{h^{(2)}, \tau, r_{1}, ..., r_{|h^{(1)}|}}_{ru} \star (y^{u, ..., u}_{s} - \hat y_s^{h^{(1)},r_1,...,r_{|h_1|}})$. Next, we want to bound $|||w|||_{(\alpha, \gamma), |h^{(1)}|}$. We have
\begin{equs}
|||w|||_{(\alpha, \gamma), |h^{(1)}|} \lesssim ||\mathbf{z}^{h^{(2)}}||_{(| h^{(2)} | \rho + \gamma, \gamma)}   ||y||_{(\alpha, \gamma),|h|} |v-u|^{\eta}|u-s|^{-\eta}
\end{equs}
where we have repeatedly applied the bound \ref{ineq2}, since in the expression for $w$ the factors comprising the term $\mathbf{z}^{h^{(2)}}$ are convolved with respect to different variables of $y^{u, ..., u}_{s} - \hat y^{h^{(1)}, r_{1}, ..., r_{|h^{(1)}|}}$. This yields a bound that involves the product of the norms of the factors comprising $\mathbf{z}^{h^{(2)}}$ which, if $\mathbf{z}^{h^{(2)}}$ is seen as an element of a tensor power of $\mathcal{V} ^{(\alpha, \gamma)}$, is equal to the factor $||\mathbf{z}^{h^{(2)}}||_{(| h^{(2)} | \rho + \gamma, \gamma)}$ appearing in the inequality above by virtue of Definition \ref{Volterra space}. Thus, combining the above estimates we get

\begin{equs}
|\mathbf{z}^{h^{(1)}, \tau}_{vr} \star  \mathbf{z}^{h^{(2)} \cdot}_{ru} \star (y^{u, ..., u}_{s} - y^{\cdot, ..., \cdot}_{s})| &  \lesssim \left( \sum_{ ((h))} \prod_{i=1}^k ||\mathbf{z}^{h^{(i)}}||_{(|h^{(i)}|\rho + \gamma, \gamma)} \right) ||y||_{(\alpha, \gamma),|h|} \\ & \times |\tau-v|^{-\gamma}|u-s|^{-\eta}|v-u|^{|h|\rho + \gamma + \eta}  
\end{equs}
where we have used that $|v-r| \leq |v-u|$ as well as the following identity:
\begin{equs}
 \sum_{(h)'} \left( \sum_{ ((h^{(1)}))} \prod_{i=1}^k ||\mathbf{\mathbf{z}}^{h^{(1i)}}||_{(|h^{(1i)}|\rho + \gamma, \gamma)} \right) ||\mathbf{\mathbf{z}}^{h^{(2)}}||_{(|h^{(2)}|\rho + \gamma, \gamma)} =  \sum_{ ((h))} \prod_{i=1}^k ||\mathbf{\mathbf{z}}^{h^{(i)}}||_{(|h^{(i)}|\rho + \gamma, \gamma)}.
\end{equs}
We notice that the regularity obtained in the last inequality above is the same as for (\ref{reg z2 con y}). Hence, choosing $\eta \in [0,1]$ as after (\ref{reg z2 con y}) and recalling (\ref{delta rule}) and  (\ref{delta two rule}), we obtain that 
\begin{equs} 
|\delta_{r}(\varXi_s)_{vu}^\tau|\lesssim c_{y,z}|\tau-v|^{-\gamma}|u-s|^{-\eta}|v-u|^{|h|\rho + \gamma + \eta}  
\end{equs}
where $\mu=|h|\rho+\gamma+\eta>1$ and where the constant $c_{y,z}$ in this inequality is the same as that in the right hand side of the last inequality above it.

Using this bound one can readily check that  $||\delta \Xi||_{(|h|\rho +\gamma + \eta, \gamma, \eta), 1} <\infty$. One can similarly check that $||\delta \Xi||_{(|h|\rho +\gamma + \eta, \gamma, \eta), 1, 2} < \infty$. Therefore, using the Volterra Sewing Lemma with two singularities (lemma \ref{Volterra Sewing Lemma}) we get that the Riemann sums defined by \eqref{Riemann sums2} converge as $|\mathcal{P}|\rightarrow 0$, and we define 
\begin{equs}\label{four sixty six}
\mathbf{z}_{ts}^{h,\tau}\star y_{s}:=\lim _{|\mathcal{P}|\rightarrow 0} \mathcal{I}_{\mathcal{P}}.
\end{equs}
We also immediately obtain the two bounds afforded to us by the Volterra Sewing Lemma, thus proving the inequalities \eqref{ineq1} and \eqref{ineq2} in the statement of the theorem. This completes the proof.
\end{proof}

\section{Rough Volterra Equations}

\ \ \ \ \ \ In this section, we introduce the Rough Path calculus for Volterra Equations. Our aim is to prove existence and uniqueness of solutions to Rough Volterra Equations with driving noise of arbitrarily low Hölder exponent. In order to formulate the main theorem (\ref{Existence and Uniqueness}), we will need to develop all the tools of classical Branched Rough Paths in a manner suited to treat the case of Volterra Equations.

 We begin this section by giving the definition of a Volterra-type Rough Path. We then go on to define the concept of a controlled Volterra Path. As in the classical Branched Rough Path setting, the latter will comprise the class of paths that can be integrated against a given Volterra-type Rough Path. 

\begin{definition} \label{Volterra_rough_path}

Fix $\alpha, \gamma \in (0,1)$ with $\alpha - \gamma > 0$ be fixed. Let $ (z_i)_{i \in \lbrace 0,...,d \rbrace} $ such that $ z_i \in \mathcal{V}^{(\alpha,\gamma)}(\Delta_2,\R) $. For $ n $ with $ (n+1)\varrho + \gamma > 1 $,   we suppose given a tree-indexed family of iterated integrals $ (\textbf{z}_{ts}^{h,\tau})_{|h| \leq n} $ indexed by the trees of $ \CT $ such that 
\begin{equs}
  \mathbf{z}^{e_i} = z_i,  \quad e_i = \begin{tikzpicture}[scale=0.2,baseline=0.1cm]
        \node at (0,0)  [dot,label= {[label distance=-0.2em]below: \scriptsize  } ] (root) {};
         \node at (0,2)  [dot,label={[label distance=-0.2em]left: \scriptsize  $ i $} ] (center) {};
      \draw[kernel1] (root) to
     node [sloped,below] {\small }     (center);
     \end{tikzpicture} 
\end{equs}
  and one has:
\begin{equs} \label{chen1}
\delta_{u} \textbf{z}^{h, \tau}_{ts} = \sum_{ (h)' } \textbf{z}^{h^{(1)},\tau}_{tu} \star  \textbf{z}^{h^{(2)}}_{us}
\end{equs}
where the Sweedler's notation corresponds to the reduced coproduct $ \tilde{\Delta} $.
Let $ h $ a tree with $ m \leq n  $ nodes. We suppose that for every $ y \in \mathcal{V}^{(\alpha,\gamma)}_{m}$, one has
\begin{equs}
\delta_{u} \textbf{z}^{h,\tau}_{ts} \star y_s = \sum_{(h)'} \left( \textbf{z}^{h^{(1)},\tau}_{tu} \star \left(  \textbf{z}^{h^{(2)}}_{us} \star  y_s \right) \right)
\end{equs}
where the product $ \star $ is defined inductively via Theorem~\ref{Convolution Product} and Proposition~\ref{convolution_init}. We also assume that $ \textbf{z}^h \in \mathcal{V}^{(|h| \varrho + \gamma, \gamma)} $. We then say that $\textbf{z}$ is a Volterra-type Rough Path. We define the norm $ {|||} \cdot {|||}_{(\alpha, \gamma)} $ as:
\begin{equs}
{|||} \textbf{z} {|||}_{(\alpha, \gamma)} = 
\sum_{h \in \CF_n}  \Vert \textbf{z}^h \Vert_{(|h| \rho + \gamma, \gamma)}.
\end{equs}
\end{definition}

\vspace{10pt}

We shall now introduce the concept of a controlled Volterra Path. Central to this is the concept of a Gubinelli derivative. These functions are basically the coefficients that appear when one has a local Taylor approximation to the initial path attached to every point. This type of local approximation property is stipulated to hold for the path and all of its derivatives, which are in turn assumed to possess the appropriate formally differentiated version of the initial Taylor expansion. In order to encode formal differentiation we make use of the $\Delta$ operation, exploiting the Hopf-algebraic structure that our forest algebra $\mathcal{H}$ was additionally equipped with.

\begin{definition}[Controlled Volterra Path] \label{CVRP}

Let $\mathbf{z}$ be an $\alpha$-Hölder Volterra-type Rough Path and let $n = \floor{1/\alpha}$. A Volterra Branched Rough Path controlled by $\mathbf{z}$ is a function $y = (y^{h})_{h \in \CT_{n-1}}$ such that, for every $h \in \CT_{n-1}$ we have $y^{h}\in \mathcal{V}^{(\alpha,\gamma)}_{|h| +1}(Q_{|h|+1}, \mathbf{R}^e)$ and the remainder terms, for every $ \tau \in [s,t] $,
 \begin{equs} \label{approximation_R}
R_{ts} = y_{ts} - \sum_{\rho \in \CF_{n-1}}  \, \mathbf{z}_{ts}^{\rho} \star y_{s}^{\rho} 
\end{equs}
satisfy $R \in \mathcal{V}^{(n\alpha, n\gamma)}_{1}$ where $ y_{ts}: =y^{h}_{ts}  $ for $ h $ being the empty tree and
\begin{equs} \label{approximation_R}
R^h_{ts} = y_{ts}^{h} - \sum_{\rho \in \CF_{n-1}} \sum_{ \sigma \in \CT_{n-1}} c(\sigma, h,\rho) \, \mathbf{z}_{ts}^{\rho} \star y_{s}^{\sigma} 
\end{equs}
satisfy $R^{h} \in \mathcal{V}^{((n-|h|)\alpha, (n-|h|)\gamma)}_{|h|+1}$. Here $c(\sigma, h, \rho)$ is the counting function for the number of appearances of the term $h \otimes \rho$ in the expansion of the reduced coproduct $\tilde{\Delta} \sigma$. The space of such functions is called the space of Controlled Volterra Paths. We equip it with the norm 
\begin{equs}
\Vert y \Vert_{\mathbf{z},(\alpha,\gamma)} =  \sum_{h \in \CT_{n-1}} \left(  |y^{h}_{0}| + ||| y^h |||_{(\alpha,\gamma), |h|}+ ||R^{h}||_{((n-|h|)\alpha, \ (n-|h|)\gamma)} \right)
\end{equs}
We shall use $\mathcal{D}_{\mathbf{z}}^{(\alpha, \gamma)}$ to denote the space of $(\alpha, \gamma)$-Hölder Volterra Paths controlled by $\mathbf{z}$. If $ y $ is zero on trees which are not planted and $ y^{h} $ is a function of $ |h| $ variables, none of them being associated with the root of $ h $, we denote this space as $\hat{\mathcal{D}}_{\mathbf{z}}^{(\alpha, \gamma)}$.
We define $\mathcal{D}_{\mathbf{z}, \mathbf{y}_0}^{(\alpha, \gamma)}$ as the following affine space:
\begin{equs}
\mathcal{D}_{\mathbf{z}, \mathbf{y}_0}^{(a, \gamma)} = \lbrace y \in \mathcal{D}_{\mathbf{z}}^{(\alpha, \gamma)}\,|\, \mathbf{y}_0  = (y_0^{h,\tau,...,\tau})_{h \in \CT_{n-1}}) = (y_0^{h})_{h \in \CT_{n-1}}) \rbrace
\end{equs}
 We define in the same way the space $\mathcal{D}_{\mathbf{z}, \mathbf{y}_0}^{(a, \gamma)}$ where $\mathcal{D}_{\mathbf{z}}^{(\alpha, \gamma)}$ is replaced by $\hat{\mathcal{D}}_{\mathbf{z}}^{(\alpha, \gamma)}$. 
\end{definition}

\begin{remark}

We shall consider a controlled path $Y$ as defined on the hypercube $Q_{n}$. It is possible to refine this domain of definition to a set of the form \ $\bigcap_{ord(h)} \{ t_{i} < t_{j} \} \subset Q_{|h| +1}$ where the order relations imposed on the components is given directly by the partial ordering induced by the tree $h$. We choose to work on $Q_{n}$ for the sake of simplicity.

\end{remark}

We will now show how one can integrate a controlled path against a given Volterra-type Rough Path. This will in essence be done by applying a formal "shift" on the Taylor-like approximation to the object $Y$. In the case of Terry Lyons's rough paths this is done by using the shift operator on the tensor algebra of words indexing the expansion. In our case the role of the shift operator is played by a "grafting" operator $\CI$ from the tensor algebra of forests to the tensor algebra of trees. We formulate the rough integration theorem below:
 
\begin{theorem}\label{Rough Integration}
Let $q \in \mathcal{C}^{\alpha}$ and $k$ be a Volterra kernel satisfying the analytic bounds \eqref{condition_k}.  Define $z^{i,\tau}_{t} = \int_{0}^{t} k(\tau,r)dq^i_{r}$ using Theorem~\ref{thm:Regularity of Volterra path} and assume there exists a Volterra-type Rough Path $\mathbf{z} $ of order $p$ built from  Definition~\ref{Volterra_rough_path} . Additionally, suppose that the $p$ components of $\mathbf{z}$ are uniformly bounded. We now consider a controlled Volterra path $(y^{h})_{h \in \CT_{p-1}} \in \mathcal{D}^{(\alpha, \gamma)}_{\mathbf{z}}$. Then, one has:
\begin{itemize}
\item[(i)] The following limit exists for all $(s,t, \tau) \in \Delta_{3}$,
\begin{equs} \label{integral_w}
w^{\tau}_{ts} = \int_{s}^{t} k(\tau,r) y^r_{r}dq^{i}_{r} := \lim_{|\mathcal{P}| \rightarrow 0} \sum_{[u,v] \in \mathcal{P}} \sum_{h \in \CT_{p-1}} \mathbf{z}^{\CI_i(h), \tau}_{vu} \star y^{h, \cdot}_{u}
\end{equs}
 where the tree $ \CI_{i}(\tau) $ is given by grafting the root of $\tau $ onto a new root with no decoration and then decorating with $i$ the node of the new tree corresponding to the root of $\tau $.
\item[(ii)] One has the following bounds for all $(s, t, p, q) \in \Delta_{4}$, $\eta \in [0,1]$ and $\zeta \in [0, p\rho]$ 
\begin{equs}\label{Rough Integration bound 1}
|w^{\tau}_{ts} - \Xi^{\tau}_{ts}| & \lesssim \| y \|_{\mathbf{z},\left(\alpha,\gamma\right)} ( 1 + |||\mathbf{z}|||_{(\alpha,\gamma)})^p \\ &  \left( [|\tau - t|^{-\gamma} |t-s|^{p \rho + \gamma}] \wedge |\tau - s|^{p \rho} \right) 
\end{equs}
\begin{equs}\label{Rough Integration bound 2}
|w^{qp}_{ts} - \Xi^{qp}_{ts}| & \lesssim  \| y \|_{\mathbf{z},\left(\alpha,\gamma\right)} (1+ |||\mathbf{z}|||_{(\alpha, \gamma)})^p \, |p-q| ^{\eta} |q-t| ^{- \eta + \zeta} \\ & \left( [|q-t| ^{-\gamma - \zeta} |t-s|^{p \rho + \gamma} ] \wedge |q-s|^{p  \rho - \zeta} \right)
\end{equs}
 where $ \Xi_{ts}^{\tau}  $ is given as
\begin{equs}
\Xi_{vu}^{\tau} =  \sum_{h \in \mathcal{T}_{p-1}} \mathbf{z}^{\CI_i(h), \tau}_{vu} \star y^{h, \cdot}_{u}
\end{equs}
where $ \ \mathbf{z}^{\CI_i(h), \tau}_{vu} \star y^{h, \cdot}_{u}$ is
understood according to Theorem~\ref{Convolution Product}. Indeed, $(y^{h})_{h \in \mathcal{T}_{p-1}} \in \mathcal{D}^{(\alpha,\gamma)}_{\mathbf{z}}$ implies that $y^{h}\in \mathcal{V}^{(\alpha,\gamma)}_{|h| +1}$.
\item[(iii)] The tuple $(w^{h})_{h \in \mathcal{T}_{p-1}}$ is a controlled Volterra path in $\mathcal{D}^{(\alpha, \gamma)}_{\mathbf{z}}$ where $w^{\CI_i(h)}_{t}= y^{h}_{t}$, $w^{h}_{t}= 0 $ for $ h $ not a planted tree and $w^{\CI_i(\one)}$ is the integral defined in $ \eqref{integral_w}$.
 This implies that $(w^{h})_{h \in \mathcal{T}_{p-1}}$ is in $\hat{\mathcal{D}}^{(\alpha, \gamma)}_{\mathbf{z}}$.   
\end{itemize}
\end{theorem}
\begin{proof}
 Our first step is to invoke the Volterra Sewing Lemma~\ref{Volterra Sewing Lemma}
in order to define 
\begin{equation}\nonumber
w_{ts}^\tau=\int_{s}^{t}k\left(\tau,r\right)y_{r}^{r}dq^i_{r}=\mathcal{I}\left(\Xi \right)_{ts}.
\end{equation}
To this aim, similarly to the proof of Theorem~\ref{Convolution Product},
we need to check that $\delta \Xi$ is sufficiently regular. This is what we proceed to do below.
    Combining \eqref{301} and \eqref{chen2}, we get the following relation for $(u,m,v,\tau)\in \Delta_4$,
\begin{equs}
\delta_{m}\Xi_{vu}^{\tau} & = \sum_{h \in \mathcal{T}_{p-1}} \(  \mathbf{z}^{\CI_i(h), \tau}_{vu} \star y^{h, \cdot}_{u} -  \mathbf{z}^{\CI_i(h), \tau}_{vm} \star y^{h, \cdot}_{m}  -  \mathbf{z}^{\CI_i(h), \tau}_{mu} \star y^{h, \cdot}_{u} \) \\
& = \sum_{h \in \mathcal{T}_{p-1}} \delta_{m}\mathbf{z}^{\CI_i(h), \tau}_{vu} \star y^{h, \cdot}_{u} -\mathbf{z}^{\CI_i(h), \tau}_{vm} \star y^{h, \cdot}_{mu} \\
& = \sum_{h \in \mathcal{T}_{p-1}}  \sum_{  (\CI_i(h))' } \mathbf{z}_{vm}^{h^{(1)}, \tau} \star \(  \mathbf{z}_{mu}^{h^{(2)}} \star y_{u}^{h, \cdot} \) - \mathbf{z}^{\CI_i(h), \tau}_{vm} \star y^{h, \cdot}_{mu} 
\end{equs}
Now, we resort to the fact that $y$ is a controlled Volterra Path and thus satisfies 
the local approximation property \eqref{approximation_R}. Substituting the appropriate expression in place of $y^{h}$ then allows us to write:
\begin{equs}
\mathbf{z}^{\CI_i(h), \tau}_{vm} \star y^{h, \cdot}_{mu}  =  \sum_{\rho \in \CF_{p-1}}\sum_{\sigma \in \mathcal{T}_{p-1}} c(\sigma, h, \rho) \, \mathbf{z}^{\CI_i(h), \tau}_{vm} \star ( \mathbf{z}_{mu}^{\rho} \star y_{u}^{\sigma} ) + \mathbf{z}^{\CI_i(h), \tau}_{vm} \star R^{h}_{mu}
\end{equs}
Plugging this into the expression we have obtained for $\delta \Xi$ we get
\begin{equs}
\delta_{m} \Xi_{vu}^{\tau} & = \sum_{h \in \mathcal{T}_{p-1}}  \( \sum_{  (\CI_i(h))'} \mathbf{z}_{vm}^{h^{(1)}, \tau} \star  \mathbf{z}_{mu}^{h^{(2)}, \cdot} \) \star y_{u}^{h, \cdot} 
\\ & 
- \sum_{\rho \in \CF_{p-1}} \sum_{h, \sigma \in \mathcal{T}_{p-1}} \ c(\sigma, h, \rho) \, \mathbf{z}^{\CI_i(h), \tau}_{vm} \star ( \mathbf{z}_{mu}^{\rho} \star y_{u}^{\sigma} ) -
\sum_{h \in \mathcal{T}_{p-1}} \mathbf{z}^{\CI_i(h), \tau}_{vm} \star R^{h, \cdot}_{mu} 
\end{equs}
where $ h^{(j)} $ is a short hand notation for $ (\mathcal{I}_i(h))^{(j)} $. By using the fact that
\begin{equs}
\sum_{h \in \mathcal{T}_{p-1}}   & \sum_{  (\CI_i(h))'} h^{(1)} \otimes  h^{(2)} \otimes h - \sum_{\rho \in \CF_{p-1}} \sum_{h, \sigma \in \mathcal{T}_{p-1}} \ c(\sigma, h, \rho) \, \CI_i(h) \otimes  \rho \otimes \sigma 
\\ & = \sum_{h \in \mathcal{T}_{p-1} \setminus \mathcal{T}_{p-2}} \CI_i(h) \otimes h
\end{equs}
we get
\begin{equs}
\delta_{m} \Xi_{vu}^{\tau} = \sum_{h \in \mathcal{T}_{p-1} \setminus \mathcal{T}_{p-2}} \mathbf{z}^{\CI_i(h), \tau}_{vm} \star y_{mu}^{h} - \sum_{h \in \mathcal{T}_{p-1}} \mathbf{z}^{\CI_i(h), \tau}_{vm} \star R^{h}_{mu} 
\end{equs}
Using this expression, we can now analyze the regularity of $\delta\Xi^\tau$. Indeed,  invoking Theorem \ref{Convolution Product}  we get for all $h \in \mathcal{T}_{p-1} \setminus \mathcal{T}_{p-2}$ the following inequalities:
\begin{equs} 
|\mathbf{z}_{vm}^{\CI_i(h),\tau} \star y_{mu}^{h}| & \lesssim \| y^{h} \|_{\left((p-1)\rho+\gamma ,\gamma \right)}
(1+{|||} \mathbf{z} {|||}_{\left(\alpha,\gamma\right)})^{p}|u-m|^{\rho}|\tau-v|^{-\gamma}|v-m|^{p \rho+\gamma} \\ 
| \mathbf{z}_{vm}^{\CI_i(h),\tau} \star R_{mu}^{h}| & \lesssim \| R^{h}\|_{\left((p-|h|)\rho +\gamma,\gamma\right)} (1 +{|||}\mathbf{z} {|||}_{\left(\alpha,\gamma\right)})^p \\ & |\tau-v|^{-\gamma}|v-m|^{(|h|+1)\rho+\gamma}|u-m|^{(p-|h|)\rho}.
\end{equs}
 Recalling that $\tau>v>m>u$, we thus obtain that 
\begin{equs}\label{a ineq}
|\delta_{m}\Xi^\tau_{vu}|\lesssim C_{y,z} \cdot |\tau-v|^{-\gamma}|v-u|^{(p+1)\rho+\gamma}. 
\end{equs}
 Since $(p+1) \cdot \rho+\gamma>1$, we can apply the Volterra Sewing Lemma~\ref{Volterra Sewing Lemma} to define $w_{ts}^\tau:=  \mathcal{I}(\Xi^\tau)_{ts}$ and at the same time obtain the bounds needed for part (ii) of the theorem. Finally, for part (iii), the bound \eqref{Rough Integration bound 1} guarantees that $ w_{ts}^{\tau} $ has the correct Taylor expansion and is a controlled Volterra rough path with $w^{\CI_i(h)}_{t}= y^{h}_{t}$. This completes the proof.
\end{proof}
\ \ \ \ \ The next step is to show how we can lift the composition of a controlled rough path $y$ with a sufficiently regular function $f$ in the space  $\mathcal{D}^{(\alpha, \gamma)}_{\mathbf{z}}$ of controlled Volterra Paths. This will finally allow us to formulate our equation abstractly in the Rough Path sense.

The following theorem shows that if $f$ is a sufficiently smooth function and $y$ is a controlled path of Volterra-type, then $f(y)$ may also be seen as a controlled path. It also gives us the form of the higher-order Gubinelli derivatives of $f(y)$. Its role is analogous to that of the chain rule in differential calculus. 

\begin{proposition}\label{Chain rule}
Let  $f \in \mathcal{C}_{b}^{p}\left( \R^e\right)$ and assume $(y^{h})_{h \in \CT_{p-1}}\in\hat{\mathcal{D}}_{\mathbf{z}}^{\left(\alpha,\gamma\right)}(\R^{e})$. 
Then the composition $(f(y)^{h})_{h \in \CT_{p-1}}$ is a controlled Volterra Path in $\mathcal{D}_{\mathbf{z}}^{\left(\alpha,\gamma\right)}\left( \R^{e} \right)$ and its Gubinelli derivatives are given for $ h = \prod_{i=1}^m h_i  \in \CT_{p-1} $  by
\begin{equs}
f(y)^{h,\tau}_{s} =   \frac{\prod_{i=1}^m S(h_i)}{m! \, S(h)} D^{m}f(y_{s}^{\tau}) \, y_{s}^{h_{1}} \otimes ... \otimes y_{s}^{h_{m}},
\end{equs}
where $ S(h) $ is the symmetry factor associated to $ h $ and $ \prod_{i=1}^m h_i $ is the tree product of the planted trees $ h_i $.
Moreover, one has
\begin{equs}\label{bound for function composed with controlled path}
& \|  f(y) \|_{\mathbf{z},(\alpha,\gamma)} \lesssim \left(1+ {|||} \mathbf{z}{|||}_{(\alpha,\gamma)}\right)^{p-1} \times   \\ &  \left[( \sum_{h \in \mathcal{T}_{p-1}} |y^h_{0}|+\|y\|_{\mathbf{z},(\alpha,\gamma)})\vee \left( \sum_{h \in \mathcal{T}_{p-1}} |y^h_{0}|+\|y\|_{\mathbf{z},(\alpha,\gamma)}) \right)^{p-1}\right]. 
\end{equs}
\end{proposition}

\begin{proof}
From \eqref{approximation_R}, we get
\begin{equs} \label{expansion_y}
 y_{ts}^{h} =  \sum_{\rho \in \CF_{p-1}} \sum_{\sigma \in \mathcal{T}_{p-1}} c(\sigma, h, \rho) \, \mathbf{z}_{ts}^{\rho, } \star y_{s}^{\sigma} + R^{h}_{ts}
\end{equs}
Using Taylor's theorem, we obtain:
\begin{equs} \label{expansion_f}
f(y^{\tau}_{t}) - f(y^{\tau}_{s}) & =  \sum_{ h =  \prod_{i=1}^m h_i \in \CT_{p-1}} \frac{\prod_{i=1}^m S(h_i)}{S(h)}\frac{D^{m} f(y^{\tau}_{s})}{m! } \prod_{i=1}^m \mathbf{z}_{ts}^{h_i, \tau} \star y_{s}^{h_i}   \\ & + \tilde{R}^{\tau}_{ts} 
\end{equs}
where the remainder $ \tilde{R}_{ts} $ is given by
\begin{equs}
& \tilde{R}^{\tau}_{ts}  =  \sum_{m= 1}^{p-1} \frac{D^{m}f(y^{\tau}_{s})}{m!}  \left( \sum_{\underset{ r < m }{ h = \prod_{i=1}^r h_i \in \CT_{p-1}}}  \frac{\prod_{i=1}^r S(h_i)}{S(h)}{m \choose m-r}(R_{ts}^{\tau})^{m-r} \prod_{i=1}^r \mathbf{z}_{ts}^{h_i, \tau} \star y_{s}^{h_i} \right. \\ & + \left. \sum_{\underset{ |h| \geq p }{ h = \prod_{i=1}^m h_i\in \CT_{p-1}}} \frac{\prod_{i=1}^m S(h_i)}{S(h)} \prod_{i=1}^m \mathbf{z}_{ts}^{h_i, \tau} \star y_{s}^{h_i} \right) \\ & + \frac{(y^{\tau}_{ts})^{m+1}}{m!}\int_0^1 (1-\theta)^m f^{(m+1)}(\theta y^{\tau}_t + (1 - \theta) y_{s}^{\tau} )) d \theta.
\end{equs}
By bounding each term of the previous sum, one can easily check that $  R \in \mathcal{V}^{(p \alpha, p \gamma)} $. Let $ h = \prod_{i=1}^m h_i \in \mathcal{T}_{p-1} $, then one has
\begin{equs}
& f(y)^{h,\tau}_{t} - f(y)^{h,\tau}_{s} = \frac{\prod_{i=1}^m S(h_i)}{m! \, S(h)} \frac{1}{m!} \left( D^{m}f(y_{t}^{\tau}) -D^{m}f(y_{s}^{\tau}) \right) \, y_{s}^{h_{1}} \otimes ... \otimes y_{s}^{h_{m}} + \\ & \frac{\prod_{i=1}^m S(h_i)}{m! \, S(h)} \sum_{i=1}^m \frac{1}{m!}  D^{m}f(y_{t}^{\tau})  \, y_{t}^{h_{1}} \otimes ... \otimes \left( y^{h_i}_t - y^{h_i}_s \right)  \otimes ... \otimes y_{s}^{h_{m}}
\end{equs}
Then, one has to plug the expansions \eqref{expansion_f} and \eqref{expansion_y} in order to conclude. This completes the proof.
\end{proof}
\begin{example}
We consider, as an example, the special case of Proposition~\ref{Chain rule} with $p=3$. Then, one has the following forests:
\begin{equs}
h_1 = \begin{tikzpicture}[scale=0.2,baseline=0.1cm]
        \node at (0,0)  [dot,label= {[label distance=-0.2em]below: \scriptsize  } ] (root) {};
         \node at (0,2)  [dot,label={[label distance=-0.2em]left: \scriptsize  $ i $} ] (center) {};
      \draw[kernel1] (root) to
     node [sloped,below] {\small }     (center);
     \end{tikzpicture},\quad h_2 =    \begin{tikzpicture}[scale=0.2,baseline=0.1cm]
        \node at (0,0)  [dot,label= {[label distance=-0.2em]below: \scriptsize  } ] (root) {};
         \node at (0,2)  [dot,label={[label distance=-0.2em]left: \scriptsize  $ j $} ] (center) {};
      \draw[kernel1] (root) to
     node [sloped,below] {\small }     (center);
     \end{tikzpicture}, \quad h_3 = \begin{tikzpicture}[scale=0.2,baseline=0.1cm]
        \node at (0,0)  [dot,label= {[label distance=-0.2em]below: \scriptsize } ] (root) {};
         \node at (-1,2)  [dot,label={[label distance=-0.2em]left: \scriptsize  $ i $} ] (left) {};
          \node at (1,2)  [dot,label={[label distance=-0.2em]right: \scriptsize  $ j $} ] (right) {};
          \draw[kernel1] (left) to
     node [sloped,below] {\small }     (root);
      \draw[kernel1] (right) to
     node [sloped,below] {\small }     (root);
     \end{tikzpicture} , \quad h_4 = 
\begin{tikzpicture}[scale=0.2,baseline=0.1cm]
        \node at (0,0)  [dot,label= {[label distance=-0.2em]below: \scriptsize  } ] (root) {};
         \node at (-1,2)  [dot,label={[label distance=-0.2em]left: \scriptsize  $ i $} ] (left) {};
          \node at (0,4)  [dot,label={[label distance=-0.2em]right: \scriptsize  $ j $} ] (leftr) {};
          \draw[kernel1] (left) to
     node [sloped,below] {\small }     (root);
      \draw[kernel1] (leftr) to
     node [sloped,below] {\small }     (left);
     \end{tikzpicture} . 
\end{equs}
 We then obtain the corresponding Gubinelli derivatives:
\begin{equs}
f(y)^{ h_1,\tau}_s  = f'(y^{\tau}_{s})\, y^{h_1}_s, \quad 
f(y)^{ h_3, \tau}_s  =  \frac{1}{2} f''(y^{\tau}_{s}) \, y^{h_1}_s \otimes y^{h_2}_s, \quad
f(y)^{ h_4, \tau}  = f'(y^{\tau}_{s})  y^{h_4}_s. 
\end{equs}
\end{example}

Before formulating our main theorem, we introduce some new spaces. Let $ 0 \leq a \leq b \leq T $, we consider 
\begin{equs}
\Delta^{T}_2([a,b]) = \lbrace (s,\tau) \in [a,b] \times [0,T] | a\leq s \leq \tau \leq T  \rbrace.
\end{equs}

\vspace{10pt}

We are now ready to formulate and prove the main theorem on existence and uniqueness of solutions to Rough Volterra Equations. We prove these properties by setting up the appropriate solution space and viewing the solution as a fixed point of an appropriately chosen mapping. We then establish that this map possesses the contraction property. This allows us to use the Banach fixed point theorem to obtain a unique fixed point, thus proving existence and uniqueness of solutions for the given equation.

\begin{theorem}\label{Existence and Uniqueness}

Let $q \in C^{\alpha}([0,T]; \R^{d})$ and let $k$ be a kernel of order $\gamma$ satisfying \ref{condition_k} with $p = \floor*{\frac{1}{\alpha - \gamma }}$. We define $\mathbf{z} \in \mathcal{V}^{(\alpha, \gamma)}$ by $z^{\tau}_{t} = \int k(\tau, r)dq_{r}$. Assume that $\mathbf{z}$ satisfies the same hypothesis as in Theorem \ref{Rough Integration} and suppose that $f_i \in \mathcal{C}^{p+1}_{b}(\R^e)$ for $ i \in \lbrace 0,...,d \rbrace $. Then, there exists a unique solution in $\hat{\mathcal{D}}_{\mathbf{z}}^{(\alpha, \gamma)}(\R^e)$ to the equation
\begin{equs}\label{RVE}
y_{t} = y^{\tau}_0 + \sum_{i=0}^{d} \int_{0}^{t}k(\tau, r)f_{i}(y^{r}_{r})dq^{i}_{r}  \ \ \ t \in [0,T], \ \ \ \ y_0 \in \R^e
\end{equs}
where the integrals above are understood as rough Volterra integrals in the sense of Theorem \ref{Rough Integration} .
\end{theorem}
\begin{proof}
We begin by defining:
\begin{equs}
(\Xi_i^{h})_{h \in \mathcal{T}_{p-1}} := (f_i(y)^{h})_{h \in \mathcal{T}_{p-1}} \in \mathcal{D}^{(p \rho + \gamma, \gamma)}_{\mathbf{z}}
\end{equs}
as given by Proposition~\ref{Chain rule}. Restricting to a subinterval $[0, T]$,  Theorem~\ref{Rough Integration} allows us to define the map
\begin{equs}
\mathcal{M}_{T}\left( (y^{h})_{h \in \mathcal{T}_{p-1}} \right)^{\tau}_t = ( \tilde{\Xi}^{h, \tau}_t)_{h \in \mathcal{T}_{p-1}}, \quad (t,\tau) \in \Delta^{T}_{2}([0, \bar{T}])
\end{equs}
where
\begin{equs}
\tilde{\Xi}^{\one,\tau}_t & = \tilde{\Xi}^{\tau}_t = y_0 + \sum_{i=0}^d \int^{t}_{0} k(\tau, r) \Xi^{r}_{i,r}dq^{i}_{r}, \\ \tilde{\Xi}^{\CI_{\ell}(h),\tau}_{t} & =  \Xi^{h, \tau}_{\ell,t}, \quad \tilde{\Xi}^{h, \tau}_t = 0, \, \, h \notin \mathcal{P}\mathcal{T}_{p-1}.
\end{equs}
The solution of our Rough Volterra Equation \eqref{RVE} on $[0,T]$ will be constructed as a fixed point of the map $\mathcal{M}_{T}$. We consider a parameter $ \beta $ such that $ \beta \leq \alpha $ and $ \beta - \gamma \geq \frac{1}{p-1} $.

\vspace{10pt}



\emph{Step 1.} We begin by defining the unit ball
\begin{equs}
\mathcal{B}_{\bar{T}} = \{ (y^{h})_{h \in \mathcal{T}_{p-1}} \in \hat{\mathcal{D}}^{(\beta, \gamma)}_{\mathbf{z}, y_{0}} (\Delta^{T}_{2}([0, \bar{T}]); \R^{m}) : ||(y^{h})_{h \in \mathcal{T}_{p-1}}||_{\mathbf{z}, (\beta, \gamma)} \leq 1 \}
\end{equs}
We will assume from now on that $|||\mathbf{z} |||_{(\alpha, \gamma), [0,T]} = M$. We consider 
\begin{equs}
\mathbf{w}^{\tau}_{ts} = (w^{h, \tau}_{ts})_{h \in \CT_{p-1}} = \mathcal{M}_{T}\left( (y^{h})_{h \in \mathcal{T}_{p}} \right)^{\tau}_{ts}
\end{equs}
 Then, by a simple extension of inequalities \ref{Rough Integration bound 1} and \ref{Rough Integration bound 2}  for $y \in \hat{\mathcal{D}}^{(\beta, \gamma)}_{\mathbf{z}, y_0}$ we have:
\begin{equs}
| w^{h, pq}_{ts} | & \leq \sum_{i=0}^d \sum_{g \in \mathcal{T}_{|h|-1}} |z^{\CI_i(g), pq}_{ts} \star f_i(y)^{g}_{s}| +
\\ & + ||f_i(y)||_{\mathbf{z}, (\beta, \gamma)} ( 1+  ||| \mathbf{z}|||_{(\beta, \gamma)})^{|h|} |p-q|^\eta |q-t|^{-\eta + \zeta} \\ & (|q-t|^{-\gamma - \zeta}|t-s|^{|h|\rho + \gamma} \wedge |q-s|^{|h|\rho - \zeta})
\end{equs}
Then, we use inequality \ref{ineq2}, take norms and sum over $h \in \mathcal{T}_{p-1}$. In this way, noting that $1 +  |||\mathbf{z}|||_{(\beta, \gamma)} \leq \lambda |||\mathbf{z}|||_{(\beta, \gamma)}$ and that $|||\mathbf{z}|||_{(\beta, \gamma)} = |||\mathbf{z}|||_{(\alpha, \gamma)} \bar{T}^{\alpha - \beta}$ we obtain that for some constant $C$ depending only on $M$, $\alpha$, $\gamma$ and $||f_i||_{\mathcal{C}^{p+1}_{b}}$ the following inequality holds:
\begin{equs}\label{bound1}
|| & \mathcal{M}_{\bar{T}}((y^{h})_{h \in \mathcal{T}_{p-1}})||_{\mathbf{z}, (\beta, \gamma)} \leq \sum_{h \in \mathcal{T}_{p-1}} \sum_{i=0}^{d} C|||\mathbf{z}|||_{(\alpha, \gamma)}(1 +  |||\mathbf{z}|||_{(\beta, \gamma)})^{|h|-1} \\ & |||f_i(y)^{h}|||_{(\beta, \gamma), |h|} \, \bar{T}^{\alpha - \beta} + C||(f_i(y)^{h}_{h \in \mathcal{T}_{p-1}}||_{\mathbf{z}, (\beta, \gamma)} \, |||\mathbf{z}|||_{(\alpha, \gamma)}( 1+  |||\mathbf{z}|||_{(\beta, \gamma)})^{|h|-1}  \, \bar{T}^{\alpha - \beta}
\end{equs}
Furthermore, by inequality \ref{bound for function composed with controlled path}, for any $h \in \mathcal{T}_{p-1}$ one has
\begin{equs}
 & |||f_i(y)^{h}|||_{(\beta, \gamma), |h|} \leq ||f_i(y)||_{\mathbf{z}, (\beta, \gamma)} \leq  C( 1+  |||\mathbf{z}|||_{(\beta, \gamma)})^{p-1} \times \\ & \left[ \left(\sum_{h \in \mathcal{T}_{p-1}} |y^h_{0}|+\|y\|_{\mathbf{z},(\beta,\gamma)} \right) \vee \left( \sum_{h \in \mathcal{T}_{p-1}} |y^h_{0}|+\|y\|_{\mathbf{z},(\beta,\gamma)}) \right)^{p-1}\right] \cdot \bar{T}^{\alpha - \beta}
\end{equs}

Therefore, under the additional asssumption that $y \in \mathcal{B}_{\bar{T}}$ we obtain a bound of the form
\begin{equs}
|| \mathcal{M}_{\bar{T}}((y^{h})_{h \in \mathcal{T}_{p-1}})||_{\mathbf{z}, (\beta, \gamma)} \leq c \cdot \bar{T}^{\alpha - \beta}
\end{equs}
Hence, for $\bar{T}$ small enough, the ball $\mathcal{B}_{\bar{T}}$ is left invariant.

\vspace{10pt}

\emph{Step 2.} Next, we will prove that $\mathcal{M}_{\bar{T}}$ is a contraction
on $\hat{\mathcal{D}}_{\mathbf{z}}^{(\alpha,\gamma)}$. To this aim, we set $F_i = f_i(y) - f_i(\tilde{y})$ and consider the controlled path $\mathbf{F} = (F^{h})_{h \in \mathcal{T}_{p-1}} \in \mathcal{D}^{(\beta, \gamma)}_{\mathbf{z}}$. Now, using inequality \eqref{bound1} which was proved in the previous step, we have
\begin{equs}
 & ||  \mathcal{M}_{\bar{T}}((y^{h})_{h \in \mathcal{T}_{p-1}})  -  \mathcal{M}_{\bar{T}}((\tilde{y}^{h})_{h \in \mathcal{T}_{p-1}}) ||_{\mathbf{z}, (\beta, \gamma)}  \leq \sum_{h \in \mathcal{T}_{p-1}} \sum_{i=0}^{d} C|||\mathbf{z}|||_{(\alpha, \gamma)}(1 +  |||\mathbf{z}|||_{(\beta, \gamma)})^{|h|-1} \\ & |||F_i^{h}|||_{(\beta, \gamma), |h|} \, \bar{T}^{\alpha - \beta} + C||(F_i^{h})_{h \in \mathcal{T}_{p-1}}||_{\mathbf{z}, (\beta, \gamma)} \, |||\mathbf{z}|||_{(\alpha, \gamma)}( 1+  |||\mathbf{z}|||_{(\beta, \gamma)})^{|h|-1}  \, \bar{T}^{\alpha - \beta}.
\end{equs}
Next, we will need to find a bound for $||(F^{h}_i)_{h \in \mathcal{T}_{p-1}}||_{\mathbf{z}, (\beta, \gamma)}$ with respect to $|| (y^{h} - \tilde{y}^{h})_{h \in \mathcal{T}_{p-1}} ||_{\mathbf{z}, (\beta, \gamma)}$. Recalling that $F_i = f_i(y) - f_i(\tilde{y})$ and keeping in mind the form of the Gubinelli derivatives using  the "chain rule" (Proposition \ref{Chain rule}) we proved earlier, we can obtain such a bound using an approach similar to that in \cite{Peter}. That is, taking into account that $y$ and $\tilde{y}$ are both in the ball $\mathcal{B}_{\bar{T}}$ and since we have assumed that $|||z|||_{\alpha, \gamma} \leq M$, one can check that there exists a constant $C' = C'_{M, \alpha, \gamma, ||f||_{\mathcal{C}^{p+1}_{b}}}$ such that
\begin{equs}
||(F^{h}_i)_{h \in \mathcal{T}_{p-1}}||_{\mathbf{z}, (\beta, \gamma)} \leq C' \, ||(y^{h} - \tilde{y}^{h})_{h \in \mathcal{T}_{p-1}}||_{\mathbf{z}, (\beta, \gamma)}
\end{equs}
Therefore, combining this inequality with the previous one we get that
\begin{equs}
|| \mathcal{M}_{\bar{T}} ( (y^{h})_{h \in \mathcal{T}_{p-1}} ) -\mathcal{M}_{\bar{T}}( (\tilde{y}^{h})_{h \in \mathcal{T}_{p-1}} ) ||_{\mathbf{z}, (\beta, \gamma)} \leq c' \,||(y^{h} - \tilde{y}^{h})_{h \in \mathcal{T}_{p-1}}||_{\mathbf{z}, (\beta, \gamma)} \, \bar{T}^{\alpha - \beta}
\end{equs}
Hence, picking $\bar{T}$ small enough, we get that $\theta := c' \bar{T}^{\alpha - \beta} < 1$. That is, for $\bar{T}$ small enough we have
\begin{equs}
|| \mathcal{M}_{\bar{T}} ( (y^{h})_{h \in \mathcal{T}_{p-1}} ) -\mathcal{M}_{\bar{T}}( (\tilde{y}^{h})_{h \in \mathcal{T}_{p-1}} ) ||_{\mathbf{z}, (\beta, \gamma)} \leq \theta \, ||(y^{h} - \tilde{y}^{h})_{h \in \mathcal{T}_{p-1}}||_{\mathbf{z}, (\beta, \gamma)}
\end{equs}
for some $\theta<1$ which establishes that $\mathcal{M}_{\bar{T}}$ is a contraction on $\hat{\mathcal{D}}^{(\beta, \gamma)}_{\mathbf{z}}(\Delta^{T}_{2}([0, \bar{T}]); \mathbf{R}^{e})$. This, together with the invariance of the ball for small enough $\bar{T}$ implies that $\mathcal{M}_{\bar{T}}$ has a unique fixed point in the ball $\mathcal{B}_{\bar{T}}$. It is clear that the fixed point inherits the regularity of the controlling noise $\mathbf{z}$. Therefore, it is also in $\hat{\mathcal{D}}^{(\alpha, \gamma)}_{\mathbf{z}}$. We conclude that this fixed point is the unique solution to the original equation in $\mathcal{B}_{\bar{T}}$. 

\vspace{10pt}

\emph{Step 3.} To finish the proof, we want to extend the solution to all of $\Delta_{2}$, 
 which we do by constructing a solution on a connected sequence of intervals of length $ \bar{T}$ covering $[0,T]$.
  We do this by constructing a solution to \eqref{RVE}
  on $\Delta_{2}^{T}\left([\bar{T},2\bar{T}]\right)$ using as a starting point $\tilde{y}_{0}$ the terminal value of the solution created on $\Delta_{2}^{T}([0,\bar{T}])$ and so on, thus constructing the solution iteratively on $\Delta_{2}^{T}\left([k\bar{T},(k+1)\bar{T}]\right)$. Since these solutions are connected on the boundaries we can then piece them together to obtain a global solution. Notice that the time step $\bar{T}$ can be made constant thanks to the fact that $f \in \mathcal{C}^{p+1}_{b}(\R^e)$.
We thus conclude that there exists a unique global  solution to \eqref{RVE} in the space $\hat{\mathcal{D}}_{\mathbf{z}}^{(\alpha,\gamma)}(\Delta_{2};\R^e)$.
\end{proof}


\end{document}